\documentclass[11pt,oneside]{article}

\usepackage[utf8]{inputenc} \usepackage[T1]{fontenc}    \usepackage[english]{babel}

\usepackage{amsfonts}
\usepackage{nicefrac}       \usepackage{microtype}      \usepackage{lmodern}
\usepackage{amssymb,amsmath,amsthm}
\usepackage{stmaryrd}
\usepackage{bbm,bm}
\usepackage{latexsym}
\usepackage{xcolor}

\usepackage{enumerate}
\usepackage{verbatim}
\usepackage{booktabs}       

\usepackage{url}            \usepackage[colorlinks=true]{hyperref}
\usepackage[numbers,sort&compress,square,comma]{natbib}
\usepackage[capitalise,noabbrev]{cleveref}

\usepackage{parskip}
\usepackage{geometry}
\usepackage[short]{optidef}
\usepackage{algorithm}

\usepackage{thmtools}

\declaretheorem{theorem}
\declaretheorem{corollary}
\declaretheorem{lemma}

\declaretheoremstyle[qed=$\square$]{definitionwithend}
\declaretheorem[style=definitionwithend]{definition}
\declaretheorem[style=definitionwithend]{assumption}

\declaretheorem[style=definitionwithend]{remark}

\crefname{assumption}{Assumption}{Assumptions}
\crefname{conjecture}{Conjecture}{Conjectures} \usepackage{enumitem}
\usepackage{graphicx}
\usepackage{import}

\definecolor{gold}{rgb}{0.85,0.65,0}

\newcommand{\abs}[1]{\ensuremath{\left\lvert #1 \right\rvert}}

\newcommand{\by}{\times}
 
\newcommand{\norm}[1]{\ensuremath{\left\lVert #1 \right\rVert}}
\newcommand{\ip}[1]{\ensuremath{\left\langle #1 \right\rangle}}
\newcommand{\grad}{\ensuremath{\nabla}}

\newcommand{\set}[1]{\left\{#1\right\}}

\def\N{{\mathbb{N}}}

\def\R{{\mathbb{R}}}
\def\S{{\mathbb{S}}}

\def\X{{\mathbb{X}}}

\def\bS{{\mathbf{S}}}

\def\cB{{\cal B}}
\def\cC{{\cal C}}
\def\cD{{\cal D}}
\def\cE{{\cal E}}
\def\cF{{\cal F}}
\def\cG{{\cal G}}

\def\cM{{\cal M}}
\def\cN{{\cal N}}

\def\cS{{\cal S}}

\def\cU{{\cal U}}

\DeclareMathOperator{\Opt}{Opt}
\DeclareMathOperator*{\argmin}{arg\,min}
\DeclareMathOperator*{\argmax}{arg\,max}

\DeclareMathOperator{\rank}{rank}
\DeclareMathOperator{\Diag}{Diag}
\DeclareMathOperator{\diag}{diag}
\DeclareMathOperator{\tr}{tr}

\DeclareMathOperator{\range}{range}

\newenvironment{smallpmatrix}
    {\left(
    \begin{smallmatrix}} 
    {\end{smallmatrix}
    \right)
    }

\DeclareMathOperator{\inter}{int}

\DeclareMathOperator{\bd}{bd}

 \usepackage{letltxmacro}
\LetLtxMacro\orgvdots\vdots
\LetLtxMacro\orgddots\ddots

\makeatletter
\DeclareRobustCommand\vdots{\mathpalette\@vdots{}}
\newcommand*{\@vdots}[2]{\sbox0{$#1\cdotp\cdotp\cdotp\m@th$}\sbox2{$#1.\m@th$}\vbox{\dimen@=\wd0 \advance\dimen@ -3\ht2 \kern.5\dimen@
\dimen@=\wd2 \advance\dimen@ -\ht2 \dimen2=\wd0 \advance\dimen2 -\dimen@
    \vbox to \dimen2{\offinterlineskip
      \copy2 \vfill\copy2 \vfill\copy2 }}}
\DeclareRobustCommand\ddots{\mathinner{\mathpalette\@ddots{}\mkern\thinmuskip
  }}
\newcommand*{\@ddots}[2]{\sbox0{$#1\cdotp\cdotp\cdotp\m@th$}\sbox2{$#1.\m@th$}\vbox{\dimen@=\wd0 \advance\dimen@ -3\ht2 \kern.5\dimen@
\dimen@=\wd2 \advance\dimen@ -\ht2 \dimen2=\wd0 \advance\dimen2 -\dimen@
    \vbox to \dimen2{\offinterlineskip
      \hbox{$#1\mathpunct{.}\m@th$}\vfill
      \hbox{$#1\mathpunct{\kern\wd2}\mathpunct{.}\m@th$}\vfill
      \hbox{$#1\mathpunct{\kern\wd2}\mathpunct{\kern\wd2}\mathpunct{.}\m@th$}}}}

\DeclareRobustCommand\bddots{\mathinner{\mathpalette\@bddots{}\mkern\thinmuskip
  }}
\newcommand*{\@bddots}[2]{\sbox0{$#1\cdotp\cdotp\cdotp\m@th$}\sbox2{$#1.\m@th$}\vbox{\dimen@=\wd0 \advance\dimen@ -3\ht2 \kern.5\dimen@
\dimen@=\wd2 \advance\dimen@ -\ht2 \dimen2=\wd0 \advance\dimen2 -\dimen@
    \vbox to \dimen2{\offinterlineskip
      \hbox{$#1\mathpunct{\kern\wd2}\mathpunct{\kern\wd2}\mathpunct{.}\m@th$}\vfill
      \hbox{$#1\mathpunct{\kern\wd2}\mathpunct{.}\m@th$}\vfill
      \hbox{$#1\mathpunct{.}\m@th$}}}}
\makeatother \usepackage{soul}

\newcommand{\ifjelse}[2]{#2}
\newcommand{\obj}{\textup{obj}}
\usepackage{siunitx}
\newcommand{\figurescale}{1.0}

\begin{document}
\title{Accelerated first-order methods for a class of semidefinite programs}
\author{
	Alex L.\ Wang
	\thanks{Carnegie Mellon University, Pittsburgh, PA, USA. Current address: Purdue University, West Lafayette, IN, USA, \url{wang5984@purdue.edu}}
	\and
	Fatma K{\i}l{\i}n\c{c}-Karzan
	\thanks{Carnegie Mellon University, Pittsburgh, PA, USA, \url{fkilinc@andrew.cmu.edu}}}
\date{\today}
\maketitle

\begin{abstract}

This paper introduces a new storage-optimal first-order method (FOM), CertSDP, for solving a special class of semidefinite programs (SDPs) to high accuracy.
The class of SDPs that we consider, the \emph{exact QMP-like SDPs}, is characterized by low-rank solutions, \textit{a priori} knowledge of the restriction of the SDP solution to a small subspace, and standard regularity assumptions such as strict complementarity.
Crucially, we show how to use a
\emph{certificate of strict complementarity}
to construct a low-dimensional strongly convex minimax problem whose optimizer coincides with a factorization of the SDP optimizer.
From an algorithmic standpoint, we show how to construct the necessary certificate and how to solve the minimax problem efficiently.
Our algorithms for strongly convex minimax problems with inexact prox maps may be of independent interest.
We accompany our theoretical results with preliminary numerical experiments suggesting that CertSDP significantly outperforms current state-of-the-art methods on large sparse exact QMP-like SDPs.

 \end{abstract}

\section{Introduction}

Semidefinite programs (SDPs) are among the most powerful tools that optimizers have for tackling both convex \emph{and} nonconvex problems.
In the former direction, SDPs are routinely used to model convex optimization problems that arise in a variety of applications such as robust optimization, engineering, and robotics~\cite{vandenberghe1996semidefinite,benTal2001lectures}. In the latter direction,
many results over the last thirty years have shown that SDPs perform provably well as convex relaxations of certain nonconvex optimization problems; see~\cite{goemans1995improved,candes2015phase,raghavendra2008optimal,benTal2001lectures} and references therein. As examples, exciting results in phase retrieval~\cite{candes2015phase} and clustering~\cite{mixon2016clustering,abbe2015exact,rujeerapaiboon2019size} 
show that these nonconvex problems have exact SDP relaxations with high probability under certain random models.
More abstractly, a line of recent work~\cite{wang2021tightness,burer2019exact,locatelli2020kkt,wang2020geometric,beck2007quadratic,beck2012new,jeyakumar2013trust,burer2014trust,sturm2003cones,argue2020necessary,kkz2021exactness,burer2017how}
has investigated general conditions under which \emph{exactness} holds between nonconvex quadratically constrained quadratic programs (QCQPs) or quadratic matrix programs (QMPs) and their standard SDP relaxations.

Despite the expressiveness and strong theoretical guarantees of SDPs, they have seen limited application
in practice and have a reputation of being prohibitively expensive, especially for large-scale applications.
Indeed, standard methods for solving SDPs, such as the interior point methods~\cite{nesterov1994interior,alizadeh1995interior}, scale poorly with problem dimension due to both their expensive iterations and also significant memory needs. See \cite[Section 8.1]{yurtsever2021scalable} for a more thorough discussion.

In this paper, we show how to derive highly efficient (in iteration complexity, per-iteration-cost, and memory usage) first-order methods (FOMs) for solving general SDPs that admit a desirable \emph{exactness} property. 
Our developments are inspired by recent results on linearly convergent FOMs for the trust-region subproblem (TRS) and the generalized trust-region subproblem (GTRS)~\cite{wang2021implicit,carmon2018analysis} that operate in the original problem space. We briefly discuss these problems now to motivate our assumptions and our problem class. We will discuss this literature in further detail in \cref{subsec:related_work}.

The TRS~\cite{more1983computing} seeks to minimize a general quadratic objective over the unit ball.
The GTRS~\cite{more1993generalizations} then replaces the unit ball constraint with a general quadratic equality or inequality constraint:
\begin{align*}
\inf_{x\in\R^{n-1}}\set{q_\obj(x):\, q_1(x) = 0}
\end{align*}
(presented as an equality constraint).
Here, both $q_\obj$ and $q_1$ may be nonconvex, but it is standard to assume that
there exists $\hat\gamma\in\R$ such that $q_\obj + \hat\gamma q_1$ is a \emph{strongly convex} quadratic function.
Under this assumption, the S-lemma~\cite{fradkov1979s-procedure} guarantees that the GTRS has an exact SDP relaxation in the following sense: Let $M_\obj,\,M_1$ be symmetric matrices such that $q_\obj(x) = \begin{smallpmatrix}
	x\\1
\end{smallpmatrix}^\intercal M_\obj \begin{smallpmatrix}
	x\\1
\end{smallpmatrix}$ and $q_1(x) = \begin{smallpmatrix}
	x\\1
\end{smallpmatrix}^\intercal M_1 \begin{smallpmatrix}
	x\\1
\end{smallpmatrix}$. Then, equality holds between the GTRS, its SDP relaxation, and the dual of the SDP relaxation:
\begin{align*}
&\min_{x\in\R^{n-1}}\set{q_\obj(x):\, q_1(x) = 0}\\
&\qquad=\min_{Y\in\S^{n}}\set{\ip{M_\obj,Y}:\, \begin{array}
	{l}
	\ip{M_1,Y} = 0\\
	Y = \begin{pmatrix}
		* & *\\ * & 1
	\end{pmatrix} \succeq 0
\end{array}}\\
&\qquad = \sup_{\gamma\in\R,\, t\in\R}\set{t:\, M_\obj + \gamma M_1 - t \begin{pmatrix}
	0_{n-1} & \\ & 1
\end{pmatrix}\succeq0}.
\end{align*}
Here, $\S^n$ is the vector space of $n\by n$ symmetric matrices, the inner product $\ip{M,Y}$ is defined as $\ip{M,Y}\coloneqq \tr(M^\intercal Y)$ and $Y\succeq 0$ indicates that $Y$ is positive semidefinite (PSD).

In particular, the SDP relaxation of the GTRS has an optimal solution $Y^*$ with rank one.
Furthermore, we know the value of $(Y^*)_{n,n} = 1$ before we even solve the SDP relaxation. We will think of this as \textit{a priori} knowledge of the restriction of $Y^*$ to a subspace of dimension $\rank(Y^*) = 1$.

Despite the fact that the SDP relaxation solves the GTRS exactly, the large computational cost of solving SDPs has spurred an extensive line of work
developing new algorithms for the GTRS (that avoid explicitly solving large SDPs).
Most relatedly, \citet{wang2021implicit} assume that the dual SDP is solvable and that there exists an optimal dual solution $(\gamma^*,t^*)$ such that $M_\obj + \gamma^* M_1 - t^* \begin{smallpmatrix}
	0_{n-1} &\\ & 1
\end{smallpmatrix}$ has rank $n-1$.
This assumption holds generically for GTRS problems conditioned on strict feasibility of the dual SDP.
\citet{wang2021implicit} then showed that it is possible to construct a strongly convex reformulation of the GTRS in the original space using a sufficiently accurate estimate of $\gamma^*$.
We will reinterpret the assumption of~\cite{wang2021implicit} as a \emph{strict complementarity}~\cite{alizadeh1997complementarity} assumption between the dual SDP and the desired rank-one solution $Y^*$.
This property will also play a crucial role in our algorithms.

In our study, we will examine general SDPs satisfying similar structural assumptions and design an efficient \emph{storage-optimal} FOM to solve them. In this respect, our approach also extends a recent line of work~\cite{ding2021optimal,yurtsever2021scalable,shinde2021memory,friedlander2016low} towards developing storage-optimal FOMs for SDPs possessing low-rank solutions.
We discuss storage optimality in SDP algorithms in \cref{subsec:related_work}.

\subsection{Problem setup and assumptions}
Consider an SDP in standard form and its dual:
\begin{align}
\tag{SDP}
\label{eq:primal_dual_SDP}
&\inf_{Y\in\S^n}\set{\ip{M_\obj,Y}:\, \begin{array}
	{l}
	\ip{M_i,Y} +d_i= 0,\,\forall i\in[m]\\
	Y\succeq 0
\end{array}}
\\
&\qquad\geq
\sup_{\gamma\in\R^m}\set{d^\intercal \gamma:\, \begin{array}
	{l}
	M_\obj + \sum_{i=1}^m \gamma_i M_i \succeq 0
\end{array}}.\nonumber
\end{align}
For notational convenience, we let $d_\obj \coloneqq 0$ and define $M(\gamma)\coloneqq M_\obj + \sum_{i=1}^m \gamma_i M_i$ and $d(\gamma) \coloneqq d_\obj + \sum_{i=1}^m \gamma_i d_i$.

In this paper, inspired by the structural properties of the GTRS that make it amenable to highly efficient FOMs, we will work under two major assumptions.
First, we will assume (\cref{as:exact_sdp}) that the primal and dual SDPs are both solvable, strong duality holds, and there exist primal and dual optimal solutions $Y^*\in\S^n$ and $\gamma^*\in\R^m$ such that $\rank(Y^*)  = k$ and $\rank(M(\gamma^*)) = n - k$. The assumption that $\rank(Y^*) + \rank(M(\gamma^*)) = n$ is referred to as \emph{strict complementarity} and is known to hold generically conditioned on primal and dual attainability~\cite{alizadeh1997complementarity}.

The strict complementarity assumption is common in the literature on algorithms for SDPs and convex optimization at large. See, for example, \cite{ding2019optimal,garber2022efficient,drusvyatskiy2018error} for FOMs that work under this assumption or \cite{goldfarb1998interior} for an analysis of interior point methods for SDPs with this assumption.

Second, we will assume (\cref{as:W_a_priori}) that the optimal primal solution $Y^*$ is known \textit{a priori} on some $k$-dimensional subspace $W^\perp$, on which it is positive definite. This assumption is inspired by QCQP and QMP applications~\cite{wang2021tightness,beck2007quadratic,shor1990dual}:
Recall that the standard SDP relaxation~\cite{shor1990dual} of an equality-constrained QCQP (in the variable $x\in\R^{n-1}$) is given by
\begin{align*}
&\inf_{x\in\R^{n-1}}\set{\begin{pmatrix}
	x\\1
\end{pmatrix}^\intercal M_\obj\begin{pmatrix}
	x\\1
\end{pmatrix}:\, \begin{array}
	{l}
	\begin{pmatrix}
	x\\1
\end{pmatrix}^\intercal M_i\begin{pmatrix}
	x\\1
\end{pmatrix} = 0,\,\forall i\in[m]
\end{array}}\\
&\qquad\geq
\inf_{Y\in\S^n}\set{\ip{M_\obj,Y}:\, \begin{array}
	{l}
	\ip{M_i,Y} = 0,\,\forall i\in[m]\\
	Y = \begin{pmatrix}
		* & *\\
		* & 1
	\end{pmatrix}\succeq 0
\end{array}}.
\end{align*}
Thus, the optimal solution (in fact, any feasible solution) to the SDP will have a $1$ in the bottom-right corner.
Taking $W$ to be the subspace corresponding to the first $(n -1)$-coordinates of $\R^n$, we have that the restriction of $Y^*$ to $W^\perp$ is known \emph{a priori} and is positive definite.
Similarly, the standard SDP relaxation~\cite{beck2007quadratic} of an equality-constrained QMP (in the variable $X\in\R^{(n-k)\times k}$) is given by
\begin{align*}
&\inf_{X\in\R^{(n-k)\times k}}\set{\tr\left(\begin{pmatrix}
	X\\
	I_k
\end{pmatrix}^\intercal M_\obj \begin{pmatrix}
	X\\
	I_k
\end{pmatrix}\right):\, \begin{array}
	{l}
	\tr\left(\begin{pmatrix}
	X\\
	I_k
\end{pmatrix}^\intercal M_i \begin{pmatrix}
	X\\
	I_k
\end{pmatrix}\right) = 0,\,\forall i\in[m]
\end{array}}\\
&\qquad\geq \inf_{Y\in\S^n}\set{\ip{M_\obj, Y}:\, \begin{array}
	{l}
	\ip{M_i, Y} = 0 ,\,\forall i\in[m]\\
	Y = \begin{pmatrix}
		* & *\\
		* & I_k
	\end{pmatrix}\succeq 0
\end{array}}.
\end{align*}
Taking $W$ to be the subspace corresponding to the first $(n-k)$ coordinates of $\R^n$, we have that the restriction of $Y^*$ to $W^\perp$ is known \emph{a priori} and positive definite.

We will refer to SDPs where \cref{as:exact_sdp,as:W_a_priori} hold as \emph{rank-$k$ exact QMP-like SDPs} or \emph{$k$-exact SDPs} for short.
While \cref{as:exact_sdp,as:W_a_priori} are both natural to assume individually, this pair of assumptions together is very powerful when assumed to hold \emph{with the same choice of $k$}. For example, in the context of SDP relaxations of QCQPs, it implies that the SDP relaxation has a unique solution of rank 1.

Naturally, the class of $k$-exact SDPs form a very special class of all SDPs.
Let us briefly comment on situations where it is natural to expect $k$-exact SDPs. First, the SDP relaxation of the GTRS is generically $1$-exact conditioning only on dual strict feasibility.
\cref{sec:strict_comp_qmp} shows that the situation is analogous in SDP relaxations of QMPs with matrix variables of dimension $X\in\R^{n\by k}$ and at most $k$ constraints. This result is new and gives an alternate proof of \cite[Theorem 3.5]{beck2007quadratic} (see also~\cite{wang2020geometric}).
More generally, it is difficult to establish whether the SDP relaxation of a given QCQP or QMP satisfies \cref{as:exact_sdp} with the appropriate value of $k$. Nonetheless, there is a long line of research establishing sufficient conditions for strict complementarity for the SDP relaxation of a QCQP~\cite{burer2019exact,locatelli2016exactness,wang2021tightness}.\footnote{Technically, these papers establish that the optimal values or optimal solutions of the SDP relaxation coincide with that of the underlying QCQP. Nonetheless, many of these sufficient conditions prove the intermediate result of strict complementarity.} 
See also \cite{de2019strict,ding2021simplicity} for work establishing strict complementarity in SDPs arising in combinatorial or statistical contexts. Finally, strict complementarity can also be shown to hold with high probability in the setting of noiseless (abstract) phase retrieval or low-rank covariance estimation~\cite{ding2023sharp}.

\subsection{Overview and outline of the paper}
In this paper, we develop a new FOM for rank-$k$ exact QMP-like SDPs.
This FOM enjoys low iteration complexity, simple iterative subprocedures, storage optimality, and strong numerical performance.
A summary of our contributions, along with an outline of the remainder of this paper, is as follows.
For the sake of presentation, we will assume that $W$ corresponds to the first $n-k$ coordinates of $\R^n$ in the following outline.

\begin{itemize}
	\item We close this section by discussing thematically related work in storage-optimal or storage-efficient FOMs for solving SDPs and FOMs for solving the GTRS. We then discuss some work on acceleration within FOMs with inexact prox oracles and FOMs for saddle-point problems as these are related to our techniques.
	\item In \cref{sec:strongly_convex_reform}, we show how to reformulate a $k$-exact SDP as a \emph{strongly convex} quadratic matrix minimax problem (QMMP) using a \emph{certificate of strict complementarity} (see \cref{def:subspace_certificate}).
	There are two key ideas here:
	First, in the setting of $k$-exact SDPs, we may parameterize the rank-$k$ matrices in $\S^n_+$ which agree with the restriction of $Y^*$ to $W^\perp$ as
	\begin{align*}
	Y(X) \coloneqq \begin{pmatrix}
		 XX^\intercal & X(Z^*)^{1/2}\\
		 (Z^*)^{1/2}X^\intercal & Z^*
	\end{pmatrix},
	\end{align*}
	where $Z^*\succ 0$ is the known restriction of $Y^*$ to $W^\perp$ and $X\in\R^{(n-k)\times k}$ is unknown.
	The task of recovering $Y^*$ then reduces to the task of recovering $X^*$.
	We replace the variable $Y\in\S^n_+$ with the parameterization $Y(X)$ in the primal SDP to derive a \emph{nonconvex} QMP in the variable $X$ whose optimizer is $X^*$.
	This first step can be compared to the Burer--Monteiro reformulation (see \cref{subsec:bm_comparison}).
	The second key idea then shows that this nonconvex QMP can be further reformulated into a strongly convex QMMP~\eqref{eq:strong_convex_reform} given a certificate of strict complementarity $\cU\subseteq\R^m$. \cref{thm:strong_convex_reform} verifies that the minimax problem
	\begin{align}
	\label{eq:intro_qmmp}
	\min_{X\in\R^{(n-k)\times k}} \max_{\gamma\in\cU} \bigg(\ip{M(\gamma), Y(X)} + d(\gamma)\bigg)
	\end{align}
	has $X^*$ as its unique optimizer and $\Opt_{\eqref{eq:primal_dual_SDP}}$ as its optimal value.
    This reformulation is inspired by a similar reformulation for the SDP relaxation of the GTRS given in \cite[Lemma 3]{wang2021implicit}. The extension from \cite[Lemma 3]{wang2021implicit} to \cref{thm:strong_convex_reform} is a conceptual contribution that recognizes that the assumptions of \cite[Lemma 3]{wang2021implicit} can be reinterpreted as strict complementarity, a standard assumption in algorithms for SDPs. Nonetheless, making \cref{thm:strong_convex_reform} algorithmic requires handling a number of challenges not present in the GTRS setting. Properly dealing with these challenges is the content of \cref{sec:qmmp,sec:solving_exact_sdps} and is the main technical contribution of this work.

\item In \cref{sec:qmmp}, we derive a two-level accelerated FOM for solving strongly convex QMMPs of the form \eqref{eq:intro_qmmp}. 
	Due to the minimax structure of  \eqref{eq:intro_qmmp}, we focus on Nesterov's optimal method for strongly convex minimax problems~\cite[Algorithm 2.3.13]{nesterov2018lectures}. This algorithm relies on a prox-map (see \cref{def:prox_map}) computation in each iteration, and its analysis assumes that prox-map is given by an explicit expression or can be computed exactly. 
In our setting, the prox-map will not admit a closed-form expression in general. Instead, we will treat the prox-map as an optimization problem in its own right and solve it via an inner FOM.
	Therefore, we suggest CautiousAGD (\cref{alg:cautious-agd}), a new variant of~\cite[Algorithm 2.3.13]{nesterov2018lectures} that handles inexact computations in the prox-map procedure.
	We extend the original estimating sequences analysis of \cite[Algorithm 2.3.13]{nesterov2018lectures} to prove bounds on the accuracy required in each individual prox-map computation to recover an accelerated linear convergence rate in terms of outer iterations (see \cref{thm:cautious_agd_oracle}).
	In our case, the prox-map can be computed efficiently using an inner loop via the strongly convex excessive gap technique~\cite[Chapter 6.2]{nesterov2018lectures}. In all, CautiousAGD computes an $\epsilon$-optimal solution of a QMMP after $O\left(\log\left(\epsilon^{-1}\right)\right)$ outer iterations and $O\left(\epsilon^{-1/2}\right)$ total inner iterations.
	\item In \cref{sec:solving_exact_sdps}, we show how to combine any method for producing iterates $\gamma^{(i)}\to\gamma^*$ with CautiousAGD to construct a certificate of strict complementarity. Combined with \cref{alg:cautious-agd}, this completes the description of our new FOM, CertSDP (\cref{alg:exact_sdp_qmmp}), for rank-$k$ exact QMP-like SDPs. Informally, we show that CertSDP returns an $\epsilon$-optimal solution to the underlying SDP after performing a fixed (i.e., independent of $\epsilon$) number of iterations of $\gamma^{(i)}\to\gamma^*$ plus either $O\left(\log\left(\epsilon^{-1}\right)\right)$ outer iterations or $O\left(\epsilon^{-1}\right)$ inner iterations in CautiousAGD. See \cref{thm:overall_exact_sdp_qmmp} for a formal statement.
    In this way, CautiousAGD \emph{can be viewed as a termination rule} for any existing SDP algorithm that produces a sequence of dual iterates: given a sufficiently accurate dual iterate, CautiousAGD produces a high-accuracy primal solution.
	\item In \cref{sec:numerical}, we present numerical experiments comparing an implementation of CertSDP with similar convex-optimization-based algorithms from the literature~\cite{ding2021optimal,yurtsever2021scalable,souto2020exploiting,odonoghue2016conic}, as well as the Burer-Monteiro Method~\cite{burer2003nonlinear}, on random sparse $k$-exact SDP instances with $n \approx 10^3$, $10^4$, and $10^5$. Our code outperforms previous state of the art convex-optimization-based algorithms and was the only such algorithm able to solve our largest instances to high accuracy. Additional experiments with stylized phase retrieval instances are summarized in \cref{subsec:additional} and discussed in detail in \cref{sec:additional_exp}.
	\item In \cref{sec:limitations}, we discuss limitations of the current work and suggest possible future directions for addressing these issues.
\end{itemize}

\subsection{Related work}
\label{subsec:related_work}

\paragraph{Storage-optimal/efficient FOMs.} A growing body of literature, itself containing multiple research strands, has explored FOMs for SDPs~\cite{lu2007large,lanprimal,baes2013randomized,ben2012solving,d2014stochastic,ding2021optimal,yurtsever2021scalable,yurtsever2019conditional,souto2020exploiting,odonoghue2016conic,yang2021inexact,friedlander2016low,majumdar2020recent,shinde2021memory}. 
Below, we recount some recent developments in this direction with a particular view towards storage-efficient or \emph{storage-optimal} FOMs for SDPs admitting low-rank solutions.
Storage-optimality alludes to the fact that a rank-$k$ PSD matrix $Y\in\S^n_+$ can be represented as the outer product of an $n\by k$ \emph{factor matrix} with itself, i.e., $Y = XX^\intercal$ for some $X\in\R^{n\by k}$, so that a primal iterate with rank $k$ can be implicitly stored using only $O(nk)$ memory.
Similarly, a dual iterate may be stored using only $O(m)$ memory. Then, a storage-optimal FOM is allowed to use only $O(m + nk)$ storage where $k$ is the rank of the \emph{true} primal SDP solution.

Low-storage and storage-optimal FOMs are particularly attractive for SDPs where $M_\obj,\,M_1,\dots,\,M_m$ are either structured or sparse, so that it is possible to not only store the instance efficiently, but also to compute matrix-vector products efficiently~\cite{ding2021optimal}. The algorithm that we develop in this paper follows this pattern and similarly interacts with $M_\obj,\,M_1,\dots,\,M_m$  via \emph{only} matrix-vector products.

One paradigm towards developing storage-optimal FOMs leverages duality to construct surrogate primal SDPs that can be solved with optimal storage.
In this paradigm,
the variable $Y\in\S^n_+$ is \emph{compressed}, i.e., replaced with $U\overline YU^\intercal$ for some matrix $U\in\R^{n\by k}$ and $\overline Y\in\S^k_+$.
\citet{ding2021optimal} give rigorous guarantees for such a method assuming strict complementarity. Specifically, they show that
if $U\in\R^{n\by k}$ corresponds to a minimum eigenspace of an approximate dual solution, then the optimal solution $\overline Y$ of the compressed SDP (in penalty form) is a good approximation of the true primal solution.
Then, combining their bounds with existing FOMs for solving the dual SDP approximately, \citet{ding2021optimal} show that $\norm{U\overline Y U^\intercal - Y^*}_F \leq \epsilon$ after $O\left(\epsilon^{-2}\right)$-many minimum eigenvector computations.
It is unclear how this convergence guarantee changes when only approximate eigenvector computations (which are the only practical option) are allowed.
\citet{friedlander2016low} explore a similar idea for trace-minimization SDPs (i.e., SDPs where $M_{\obj} = I$) from the viewpoint of \emph{gauge duality}. Specifically, they show that if $U$ corresponds to a minimum eigenspace associated with the true solution to the gauge dual, then the optimal solution of the compressed SDP exactly recovers the true primal SDP solution.
Unfortunately, they do not analyze the accuracy of the recovered primal solution when the gauge dual is solved only approximately, which is the case in practice.

A second paradigm towards developing storage efficient/optimal FOMs works simultaneously in both the primal and dual spaces by employing \emph{linear sketches}.
\citet{yurtsever2021scalable} apply the Nystr\"om sketch to the conditional gradient--augmented Lagrangian (CGAL) technique~\cite{yurtsever2019conditional} to derive SketchyCGAL. They show that it is possible to reconstruct a $(1+\zeta)$-optimal rank-$k$ approximation of an $\epsilon$-optimal solution to the primal SDP\footnote{In \cite{yurtsever2021scalable}, a rank-$k$ matrix $\tilde Y\in\S^n_+$ is a $(1+\zeta)$-optimal rank-$k$ approximation of an $\epsilon$-optimal solution $Y_\epsilon\in\S^n_+$ if $\norm{Y_\epsilon - \tilde Y}_* \leq (1+\zeta) \norm{Y_\epsilon - [Y_\epsilon]_k}_*$ where $\norm{\cdot}_*$ is the nuclear norm and $[Y_\epsilon]_k$ is the best rank-$k$ approximation of $Y_\epsilon$.} 
by tracking only the dual iterates as well as a $O(nk/\zeta)$-sized sketch of the primal iterates. 
When the true solution is unique and has rank-$k$, it is appropriate to take $\zeta = O(1)$ so that the total storage is $O(m + nk)$.
Furthermore, \citet{yurtsever2021scalable} bound the required accuracy in the approximate eigenvector computations within SketchyCGAL. In all, they show that it is possible to implement their algorithm in $O\left(\epsilon^{-2}\right)$ iterations where each iteration involves computing an eigenvector via $\tilde O\left(\epsilon^{-1/2}\right)$ matrix-vector products.
In follow-up work, \citet{shinde2021memory} combine the algorithmic architecture of SketchyCGAL with the additional observation that in specific applications (e.g., max-cut), the goal is simply to \emph{sample} from a Gaussian distribution with variance given by an approximate solution $Y_\epsilon$ to the SDP. Under this alternate goal, it is possible to further reduce the storage requirements to $O(n + m)$.

One may compare these storage-optimal FOMs for SDPs with the Burer--Monteiro method~\cite{burer2003nonlinear}. In the Burer--Monteiro method, the convex SDP in the variable $Y\in\S^n_+$ is explicitly replaced with an outer product term involving an $n\by k'$ factor matrix where $k'\geq k$. The resulting \emph{nonconvex} problem is then tackled via local optimization methods.
While results~\cite{boumal2016non,cifuentes2019polynomial,cifuentes2021burer} have shown that non-global local minima cannot exist when $k'=\Omega(\sqrt{m})$ (so that local optimization methods are certifiably correct) for large classes of SDPs, more recent work~\cite{waldspurger2020rank} has shown that such spurious local minima can in fact exist even if $k = 1$ and $k' = \Theta(\sqrt{m})$. In other words, the Burer--Monteiro approach provably cannot achieve storage-optimality.

\paragraph{FOMs for the GTRS.} The algorithms developed in the current paper are 
inspired by recent developments in FOMs for the TRS and the GTRS.
There has been extensive work~\cite{carmon2018analysis,wang2020generalized,wang2021implicit,jiang2019novel,hazan2016linear,hoNguyen2017second} 
towards developing customized algorithms for the TRS and GTRS that circumvent solving large SDPs; see \cite{wang2021implicit} and references therein for a more thorough account of algorithmic ideas for solving large-scale GTRS instances. We highlight only the two most relevant results from this area.

\citet{carmon2018analysis} consider iterative methods that produce Krylov subspace solutions to the TRS, i.e., solutions to the TRS restricted to a Krylov subspace generated by the objective function.
They show that these solutions converge to the true TRS solution \emph{linearly} as long as the linear term in $q_\obj$ is not orthogonal to the minimum eigenspace of the Hessian in $q_\obj$.
One may show that under this assumption, the quadratic function $q_\obj(x) + \gamma^*\left(\norm{x}^2 - 1\right)$ is strongly convex (here, $\gamma^*$ is the dual solution and is known to exist).
We may interpret this as strict complementarity between the SDP relaxation of the TRS and its dual. Indeed, the Hessian of this quadratic function (a positive definite matrix) shows up as a principal minor in the slack matrix $M(\gamma^*)$, proving that $\rank(M(\gamma^*)) = n- 1$.

More recently, \citet{wang2021implicit} make a connection between the GTRS and optimal FOMs for strongly convex minimax problems~\cite{nesterov2018lectures}.
The main algorithmic contributions in \cite{wang2021implicit} assume
dual strict feasibility of the SDP relaxation and that the dual problem has a maximizer $\gamma^*$ in the interior of its domain. As a consequence, the quadratic function $q_\obj(x) + \gamma^* q_1(x)$ is strongly convex. Again, we may interpret this as a strict complementarity assumption.
\citet{wang2021implicit} then show how to construct a strongly convex reformulation of the GTRS using low-accuracy eigenvalue computations.
More concretely, they show how to construct $\tilde\gamma_-$ and $\tilde\gamma_+$ such that the minimax problem
\begin{align*}
\min_{x\in\R^n}\max_{\gamma\in\left[\tilde\gamma_-, \tilde\gamma_+\right]}\left(q_\obj(x) + \gamma \cdot q_1(x)\right)
\end{align*}
is strongly convex and has as its unique optimizer the optimizer of the underlying GTRS. The resulting strongly convex minimax problem is then solved via \cite[Algorithm 2.3.13]{nesterov2018lectures} to achieve a linear convergence rate.
One may compare the strongly convex reformulation of the GTRS in \cite{wang2021implicit} with the more natural Lagrangian reformulation (through S-lemma):
\begin{align*}
\min_{x\in\R^n}\sup_{\gamma\in\Gamma} \left(q_\obj(x) + \gamma \cdot q_1(x)\right)
\end{align*}
where $\Gamma = \set{\gamma\in\R_+:\, q_\obj + \gamma q_1 \text{ is convex}}$.
Specialized FOMs have also been developed for the GTRS using this Lagrangian reformulation~\cite{wang2020generalized}.
Unfortunately, since the Lagrangian reformulation may not be strongly convex in general, the resulting algorithms can only achieve sublinear (in terms of $\epsilon$) convergence rates---specifically, rates of the form $O\left(\epsilon^{-1/2}\right)$ as opposed to rates of the form $O\left(\log(\epsilon^{-1})\right)$.

\paragraph{Accelerated FOMs for non-smooth problems via saddle-point problems.}
One may treat the QMMP reformulation \eqref{eq:intro_qmmp} of the SDP as a saddle-point problem in the variables $(X,\gamma)\in\R^{(n-k)\times k}\times \R^m$ as opposed to a non-smooth problem in just $X\in\R^{(n-k)\times k}$.
There is a vast body of work developing accelerated FOMs for non-smooth problems that leverages saddle-point structure~\cite{nesterov2018lectures,palaniappan2016stochastic,juditsky2011first,nemirovski2004prox}.
Both \citet{nesterov2005smooth} and \citet{nemirovski2004prox} achieve an accelerated convergence rate of $O(\epsilon^{-1})$ for general convex--concave saddle point problems (see also \cite{tseng2008accelerated}).
This rate can be further improved for the special case of strongly convex--concave saddle-point problems~\cite{juditsky2011first,nesterov2005excessive,chambolle2016ergodic}:
Nesterov's excessive gap technique~\cite{nesterov2005excessive,nesterov2018lectures} achieves an $O(\epsilon^{-1/2})$ convergence for 
strongly convex--concave saddle-point problems where the coupling term is linear.
This is generalized in~\cite{chambolle2016ergodic} to allow nonlinear proximal operators.
\citet{juditsky2011first,hamedani2021primal} generalize this convergence rate to the setting where the gradient of the coupling term is only assumed to be Lipschitz.
These rates match the known~\cite{ouyang2021lower} lower bound of $O(\epsilon^{-1/2})$ for any FOM on the general class of strongly convex-concave saddle-point problems.
Note that the assumption that the gradient of the coupling term is Lipschitz does not hold for our setting. Indeed, the saddle point function we are interested in, $\ip{M(\gamma),\, Y(X)}$, is jointly \emph{cubic} in the variables $(X,\gamma)$ (so that the gradients vary quadratically). Nonetheless, we will show that it is possible achieve the optimal $O(\epsilon^{-1/2})$ iteration complexity in our setting.

\paragraph{Accelerated FOMs with inexact first-order information.}
A related line of work~\cite{devolder2013first,devolder2014first} has analyzed the convergence rate of (accelerated) FOMs in the presence of \emph{inexact first-order information}. \citet{devolder2014first} analyzes FOMs for smooth convex functions. In \cite{devolder2013first}, the same authors extend these results to FOMs for smooth and strongly convex functions.
Our algorithm (\cref{alg:cautious-agd}) continues this line of work by considering an \emph{inexact prox-map} for strongly convex \emph{max-type functions}.
Our work additionally complements work on inexact gradients within prox-grad methods for composite optimization problems (note that max-type functions cannot in general be decomposed as composite optimization problems).

\subsection{Notation}
For a positive integer $n$, let $[n]\coloneqq\set{1,\dots,n}$.
Let $\S^n$ denote the set of $n\by n$ real symmetric matrices and let $I_n$ denote the $n\by n$ identity matrix.
Given $X\in\S^n$, we write $X\succeq 0$ (resp.\ $X\succ 0$) if $X$ is positive semidefinite (resp.\ positive definite). Let $\S^n_+$ be the positive semidefinite cone.
Given $X\in\R^{n\by m}$ let $X^\intercal$, $\ker(X)$, $\range(X)$, $\rank(X)$, $\norm{X}_F$ and $\norm{X}_2$  denote the transpose, kernel, range, rank, Frobenius norm, and spectral norm of $X$.
Given $X\in\R^{n\by n}$ let $\tr(X)$ denote the trace of $X$.
We endow both $\S^n$ and $\R^{n\by m}$ with the trace inner product $\ip{X,Y} \coloneqq \tr(X^\intercal Y)$.
Let $\bS^{n-1}\subseteq\R^n$ denote the unit sphere.
Given $x\in\R^n$ and $r\geq 0$, let $\mathbb{B}(x,r)$ denote the closed $\ell_2$-ball centered at $x$ with radius $r$.
Given a function in multiple arguments $f(x_1,\dots,x_m)$, we write $\grad_k f(x_1,\dots,x_m)$ to denote the gradient of $f$ in the $k$th argument evaluated at $x_1,\dots,x_m$.

Given a subspace $W\subseteq\R^n$, let $W^\perp$ denote its orthogonal complement.
Abusing notation, we write $\S^W$ for the vector space of self-adjoint operators on $W$ and $\R^{W,W^\perp}$ for the vector space of linear maps from $W^\perp$ to $W$.
Given $M\in\S^n$, let $M^{}_W\in\S^W$, $M^{}_{W,W^\perp}\in\R^{W,W^\perp}$, and $M^{}_{W^\perp}\in\S^{W^\perp}$ denote the restrictions of $M$ to the corresponding subspaces. Explicitly,
\begin{gather*}
    M^{}_W:\quad x\in W\quad\mapsto\quad (\Pi_W M \Pi_W)x \in W,\\
    M^{}_{W,W^\perp}:\quad x\in W^\perp \quad\mapsto\quad (\Pi_W M \Pi_{W^\perp}) x \in W,\\
    M^{}_{W^\perp}: \quad x\in W^\perp \quad\mapsto \quad (\Pi_{W^\perp} M \Pi_{W^\perp}) x\in W^\perp,
\end{gather*}
where $\Pi_W$ and $\Pi_{W^\perp}$ are the orthogonal projections onto $W$ and $W^\perp$.
When $W$ is the vector space corresponding to the first $n-k$ coordinates, we may identify $M^{}_W$,
$M^{}_{W,W^\perp}$, and $M^{}_{W^\perp,W^\perp}$
with the top-left $(n-k)\by (n-k)$ submatrix,  top-right $(n-k)\by k$ submatrix, and bottom-right $k\by k$ submatrix of $M$ respectively.

\section{Strongly convex reformulations of $k$-exact SDPs}
\label{sec:strongly_convex_reform}

In this section, we describe how to construct a strongly convex reformulation of a \emph{rank-$k$ exact QMP-like SDP} using a \emph{certificate of strict complementarity} (see \cref{def:subspace_certificate,def:exact_sdp}). The following sections will expand on these ideas and show how these properties can be exploited to achieve algorithmic efficiency.

\subsection{Definitions and problem setup}
\label{subsec:exact_sdp}

We make the following two assumptions on \eqref{eq:primal_dual_SDP}.

\begin{assumption}
\label{as:exact_sdp}
Assume in \eqref{eq:primal_dual_SDP} that the primal and dual problems are both solvable, strong duality holds, and there exist primal and dual optimal solutions $Y^*\in\S^n$ and $\gamma^*\in\R^m$ such that $\rank(Y^*)  = k$ and $\rank(M(\gamma^*)) = n - k$.
\end{assumption}
We fix $Y^*$ and $\gamma^*$ to be solutions to \eqref{eq:primal_dual_SDP} satisfying $\rank(Y^*) = k$ and $\rank(M(\gamma^*)) = n - k$.

\begin{assumption}
\label{as:W_a_priori}
Let $W\subseteq\R^n$ be an $n-k$-dimensional subspace such that the restriction of $Y^*$ to $W^\perp$ is known and positive definite.
\end{assumption}

\begin{definition}
\label{def:exact_sdp}
We say that an instance of \eqref{eq:primal_dual_SDP} is a \emph{rank-$k$ exact QMP-like SDP} or a \emph{$k$-exact SDP} for short if both \cref{as:exact_sdp,as:W_a_priori} hold.
\ifjelse{\qed}{}
\end{definition}

\begin{definition}
\label{def:subspace_certificate}
We say that a compact subset $\cU \subseteq \R^m$ 
\emph{certifies strict complementarity} if $\gamma^*\in \cU$ and, for all $\gamma\in\cU$, it holds that $M(\gamma)^{}_W\succ 0$.
\ifjelse{\qed}{}
\end{definition}

\begin{remark}
Suppose we are given a certificate of strict complementarity $\cU$, i.e., $\gamma^*\in\cU$ and $M(\gamma)^{}_W \succ 0$ for all $\gamma\in\cU$. We immediately deduce that $\rank(M(\gamma^*))\geq \rank(M(\gamma^*)^{}_W) = n - k$.
On the other hand, $\rank(Y^*) \geq \rank(Y^*_{W^\perp}) = k$.
This is the sense in which $\cU$ \emph{certifies strict complementarity}.
\ifjelse{\qed}{}
\end{remark}

The following lemma shows that $M(\gamma^*)_W\succ 0$ so that certificates of strict complementarity exist.
\begin{lemma}
    \label{lem:complementary_subspace}
    Suppose $M^*,\,Y^*\in\S^n_+$ have rank $n-k$ and $k$ respectively and that $\ip{M^*,Y^*}=0$. Let $W$ be an $(n-k)$-dimensional subspace. Then, $M^*_{W}\succ 0$ if and only if $Y^*_{W^\perp}\succ 0$.
    \end{lemma}
    \begin{proof}
    It suffices to prove the forward direction as we may interchange the roles of $Y^*$ and $M^*$.   
    
    We prove the contrapositive. Suppose $Y^*_{W^\perp}\not\succ 0$ so that $\ker(Y^*_{W^\perp})$ is nontrivial.
    As $Y^*\succeq 0$, we have that in fact $\ker(Y^*)\cap W^\perp$ is nontrivial.
    Then, $\range(Y^*)$ is a $k$-dimensional subspace contained in $(\ker(Y^*)\cap W^\perp)^\perp$. Similarly, $W$ is an $(n-k)$-dimensional subspace contained in $(\ker(Y^*)\cap W^\perp)^\perp$. Then, as $(\ker(Y^*)\cap W^\perp)^\perp$ has dimension at most $n - 1$, we deduce that $\range(Y^*)\cap W$ is nontrivial and $\ip{Y^*, M^*} >0$, a contradiction.
    \end{proof}

\subsection{Identifying $\S^n$ with quadratic matrix functions}
\label{subsec:Sn_to_quadratic_matrix_function}
Suppose \eqref{eq:primal_dual_SDP} is a $k$-exact SDP and that $\cU$ certifies strict complementarity.
For ease of presentation, we will assume in this subsection that $W$ is the $(n-k)$-dimensional subspace corresponding to the first $n-k$ coordinates of $\R^n$. This is without loss of generality and our results extend in the natural way to the setting where $W$ is general (see \cref{rem:coordinate_to_general_submatrix_to_restrictions}).

Our strongly convex reformulation of \eqref{eq:primal_dual_SDP} will regard the $M_i\in\S^n$ as inducing \emph{quadratic matrix functions} on the space $\R^{W\times W^\perp}\simeq \R^{(n-k)\times k}$.
We begin by writing each $M_i$, for $i\in\set{\obj}\cup[m]$, as a block matrix
\begin{align*}
M_i = \begin{pmatrix}
	A_i/2 & \tilde B_i/2\\
	\tilde B_i^\intercal /2 & C_i
\end{pmatrix},
\end{align*}
where $A_i\in\S^{n-k}$, $\tilde B_i\in\R^{(n-k)\times n}$ and $C_i\in\S^k$.

We will partition $Y^*$ as a block matrix with compatible block structure:
Define $Z^* \coloneqq Y^*_{W^\perp}$ and $X^* \coloneqq Y^*_{W,W^\perp} (Z^*)^{-1/2}$.
Note here that $Z^*$ is known \emph{a priori} due to \cref{as:W_a_priori}.
Next, by the assumption that $\rank(Y^*) = k$ (\cref{as:exact_sdp}), we have that
\begin{align*}
Y^* &= \begin{pmatrix}
	X^*X^{*\intercal} & X^*(Z^*)^{1/2}\\
	(Z^*)^{1/2}(X^*)^\intercal & Z^*
\end{pmatrix}.
\end{align*}

Finally, given $X\in\R^{(n-k)\times k}$, define
\begin{align*}Y(X) \coloneqq \begin{pmatrix}
	XX^\intercal & X(Z^*)^{1/2}\\
	(Z^*)^{1/2}X^\intercal & Z^*
\end{pmatrix}
\end{align*}
and note that $Y(X^*) = Y^*$.

One of our key ideas in building a strongly convex reformulation of \eqref{eq:primal_dual_SDP} is that $Y(X)$ is a matrix whose entries are \emph{quadratic} in $X$.  We can thus identify each $M_i$ with a quadratic matrix function. For each $i\in\set{\obj}\cup[m]$, define
\begin{align*}
q_i(X) \coloneqq \ip{M_i, Y(X)} + d_i &= \frac{\tr(X^\intercal A_i X)}{2} + \ip{\tilde B_i(Z^*)^{1/2},X} + \ip{C_i,Z^*} + d_i\\
&=\frac{\tr(X^\intercal A_i X)}{2} + \ip{B_i,X} + c_i,
\end{align*}
where we have defined $B_i \coloneqq \tilde B_i (Z^*)^{1/2}$ and $c_i \coloneqq \ip{C_i, Z^*} + d_i$.
Finally, given $\gamma\in\R^m$, define $A(\gamma) \coloneqq A_\obj + \sum_{i=1}^m \gamma_i A_i$. We define $B(\gamma),\,\tilde B(\gamma),\, c(\gamma),\,d(\gamma)$, and $q(\gamma,X)$ analogously:
\begin{align*}
    &B(\gamma) \coloneqq B_\obj + \sum_{i=1}^m \gamma_i B_i,
    \hspace{2.5em}
    \tilde B(\gamma) \coloneqq \tilde B_\obj + \sum_{i=1}^m \gamma_i \tilde B_i,
    \hspace{2.5em}    
    c(\gamma) \coloneqq c_\obj + \sum_{i=1}^m \gamma_i c_i,\\
    &d(\gamma) \coloneqq d_\obj + \sum_{i=1}^m \gamma_i d_i,
     \hspace{4em}   
    q(\gamma,X) \coloneqq q_\obj(X) + \sum_{i=1}^m \gamma_i q_i(X).
\end{align*}

\begin{remark}
\label{rem:coordinate_to_general_submatrix_to_restrictions}
It is without loss of generality to assume that $W$ is the first $(n-k)$ coordinate space when proving the \emph{structural} results of this section. Only small notational changes need to be made when we go from $W$ being the first $(n-k)$-coordinate space to a general $(n-k)$-dimensional subspace of $\R^n$.
In the general setting, we define
\begin{gather*}
    A_i = 2(M_i)^{}_W
    \qquad
    \tilde B_i = 2 (M_i)^{}_{W,W^\perp}
    \qquad
    C_i = (M_i)^{}_{W^\perp}\\
    A(\gamma) = 2(M(\gamma))^{}_W
    \qquad
    \tilde B(\gamma) = 2 (M(\gamma))^{}_{W,W^\perp}
    \qquad
    C(\gamma) = (M(\gamma))^{}_{W^\perp}.
\end{gather*}
We again define $B_i = \tilde B_i (Z^*)^{1/2}$ and $B(\gamma) = \tilde B(\gamma) (Z^*)^{1/2}$.
Quantities in $\R^{(n-k)\by k}$ from the coordinate setting are replaced by quantities in $\R^{W\by k}$, where $\R^{W\by k}$ denotes the subspace of matrices in $\R^{n\by k}$ where every column lives in $W$.\ifjelse{\qed}{}
\end{remark}

\subsection{A strongly convex reformulation of \eqref{eq:primal_dual_SDP}}
\label{subsec:exactness_qmmp}

The following theorem states that if $\cU$ certifies strict complementarity, then $X^*$ is the unique minimizer of a strongly convex \emph{quadratic matrix minimax problem} (QMMP) that can be constructed from $\cU$.
This reformulation is inspired by a recent strongly convex reformulation of the GTRS \cite{wang2021implicit}.
\begin{theorem}
\label{thm:strong_convex_reform}
Suppose \eqref{eq:primal_dual_SDP} is a rank-$k$ exact QMP-like SDP and that $\cU$ certifies strict complementarity.
Then, $X^*$ is the unique minimizer of the strongly convex QMMP
\begin{align}
\tag{$\text{QMMP}_\cU$}
\label{eq:strong_convex_reform}
\min_{X\in\R^{W\times W^\perp}}\max_{\gamma\in \cU} q(\gamma,X).
\end{align}
Furthermore, $X^* = - A(\gamma^*)^{-1}B(\gamma^*)$ and $\Opt_{\eqref{eq:strong_convex_reform}}= \Opt_{\eqref{eq:primal_dual_SDP}}$.
\end{theorem}
\begin{proof}

Without loss of generality, we work in the basis where $W$ is the first $n-k$ coordinates of $\R^n$. Note that the assumption that $\cU$ certifies strict complementarity implies that $A(\gamma^*) = M(\gamma^*)_W^{} \succ 0$.

We begin by verifying that $X^* = -A(\gamma^*)^{-1}B(\gamma^*)$. By complementary slackness, we have
\begin{align*}
0 &= \ip{M(\gamma^*), Y(X^*)}\\
&= \tr\left(\begin{pmatrix}
	X^*\\
	(Z^*)^{1/2}
\end{pmatrix}^\intercal \begin{pmatrix}
	A(\gamma^*)/2 & \tilde B(\gamma^*)/2\\
	\tilde B(\gamma^*)^\intercal/2 & C(\gamma^*)
\end{pmatrix} \begin{pmatrix}
	X^*\\
	(Z^*)^{1/2}
\end{pmatrix}\right)\\
&= \tr\left(\frac{\left(X^* + A(\gamma^*)^{-1}B(\gamma^*)\right)^\intercal A(\gamma^*)\left(X^* + A(\gamma^*)^{-1}B(\gamma^*)\right)}{2}\right)\\
&\qquad +\left[\ip{C(\gamma^*),Z^*} - \tr\left(\frac{B(\gamma^*)^\intercal A(\gamma^*)^{-1}B(\gamma^*)}{2}\right)\right].
\end{align*}
Here, the second line follows by the definitions of $M(\gamma^*)$ and $Y(X^*)$, and the third line follows from the definition $B(\gamma) \coloneqq \tilde B(\gamma) (Z^*)^{1/2}$. 
We claim that the square-bracketed term on the final line is zero: By the assumption that $\rank(M(\gamma^*)) = n-k$ and the fact that $A(\gamma^*)\succ 0$, we have that $C(\gamma^*) = \frac{\tilde B(\gamma^*)^\intercal A(\gamma^*)^{-1}\tilde B(\gamma^*)}{2}$. Pre- and post-multiplying $C(\gamma^*)$ by $(Z^*)^{1/2}$ and taking the trace of this identity gives
\begin{align*}
\ip{C(\gamma^*),Z^*} &= \tr\left(\frac{(Z^*)^{1/2} \tilde B(\gamma^*)^\intercal A(\gamma^*)^{-1}\tilde B(\gamma^*)(Z^*)^{1/2}}{2}\right)\\
&= \tr\left(\frac{B(\gamma^*)^\intercal A(\gamma^*)^{-1} B(\gamma^*)}{2}\right).
\end{align*}
Thus, we have that
\begin{align*}
0 = \tr\left(\left(X^* + A(\gamma^*)^{-1} B(\gamma^*)\right)^\intercal A(\gamma^*)\left(X^* + A(\gamma^*)^{-1} B(\gamma^*)\right)\right),
\end{align*}
so that $X^* = -A(\gamma^*)^{-1}B(\gamma^*)$ by the positive definiteness of $A(\gamma^*)$.

Next, note that by the feasibility of $Y^*$, we have $q_i(X^*) = \ip{M_i, Y^*} + d_i = 0$ for all $i\in[m]$.
Similarly, by the optimality of $Y^*$, we have $q_\obj(X^*) = \ip{M_\obj, Y^*} = \Opt_\eqref{eq:primal_dual_SDP}$. In particular, $\max_{\gamma\in \cU} q(\gamma,X^*) = q(\gamma^*,X^*) = \Opt_\eqref{eq:primal_dual_SDP}$. On the other hand, for any $X\in\R^{W\times W^\perp}$,
\begin{align*}
\max_{\gamma\in \cU} q(\gamma,X) &\geq q(\gamma^*,X) = \Opt_\eqref{eq:primal_dual_SDP} + \tr\left(\frac{(X - X^*)^\intercal A(\gamma^*)(X - X^*)}{2}\right).
\end{align*}
As $A(\gamma^*)\succ 0$, we conclude that $X^*$ is the unique minimizer of \eqref{eq:strong_convex_reform} with optimal value $\Opt_\eqref{eq:strong_convex_reform} = \Opt_\eqref{eq:primal_dual_SDP}$.

Finally, strong convexity of \eqref{eq:strong_convex_reform} follows from compactness of $\cU$ and the assumption that $\cU$ certifies strict complementarity (so that $A(\gamma) = M(\gamma)^{}_W$ is positive definite over $\cU$).
\end{proof}

\begin{remark}

    One may compare \eqref{eq:strong_convex_reform} with the more natural Lagrangian formulation of \eqref{eq:primal_dual_SDP}, which results in a QMMP in the same space:
    \begin{align}
    \label{eq:partial_dual_SDP}
    \min_{X\in\R^{W\times W^\perp}} \sup_{\gamma\in\R^m:\, A(\gamma)\succeq 0} q(\gamma, X).
    \end{align}
    Indeed, it is possible to show that $X^*$ is also the unique minimizer of \eqref{eq:partial_dual_SDP}.
    Nevertheless, the formulation~\eqref{eq:partial_dual_SDP}, in contrast to \eqref{eq:strong_convex_reform}, has two major downsides: First, it may be the case that $\sup_{\gamma\in\R^m:\, A(\gamma)\succeq 0}q(\gamma,X)$ is a convex function in $X$ that is \emph{not strongly convex}.
    Second, the domain of the supremum, $\set{\gamma\in\R^m:\, A(\gamma)\succeq 0}$, is itself a spectrahedron so that \emph{even evaluating} $\sup_{\gamma\in\R^m:\, A(\gamma)\succeq 0}q(\gamma, X)$ (that is, evaluating \emph{zeroth-order} information in the $X$ variable) requires solving an SDP.
    In contrast, \eqref{eq:strong_convex_reform} is strongly convex by construction.
    Furthermore, we may pick $\cU$ to have efficient projection and linear maximization oracles (e.g., by taking $\cU$ to be an $\ell_2$ ball).
From this viewpoint, \eqref{eq:strong_convex_reform} will be much more amenable than \eqref{eq:partial_dual_SDP} to first-order methods.\ifjelse{\qed}{}
\end{remark}

\subsection{Comparison of \eqref{eq:strong_convex_reform} with the Burer--Monteiro approach}
\label{subsec:bm_comparison}

We show that the Burer--Monteiro approach~\cite{burer2003nonlinear} may be difficult to convexify meaningfully and that it may admit spurious second order critical points even on the class of $1$-exact QMP-like SDPs.

In the Burer-Monteiro approach, one replaces the matrix variable $Y\in\S^n_+$ with a rank-$k$ matrix variable parameterized by
\begin{align*}
Y = \begin{pmatrix}
	X\\
	X'
\end{pmatrix}\begin{pmatrix}
	X\\
	X'
\end{pmatrix}^\intercal
\end{align*}
where $X\in\R^{(n-k)\times k}$ and $X'\in\R^{k\times k}$.
This transformation replaces the $O(n^2)$-dimensional variable $Y\in\S^n_+$ with the $nk$-dimensional variable $(X;X')\in\R^{n\times k}$. We thus arrive at the following QMP:
    \begin{align}
        \label{eq:bm_homogeneous}
        \min_{\substack{X\in\R^{(n-k)\by k}\\X'\in\R^{k\by k}}}\set{\ip{M_\obj, \begin{pmatrix}
            X\\ X'
            \end{pmatrix}\begin{pmatrix}
                X\\ X'
                \end{pmatrix}^\intercal}:\, \begin{array}{r}
                    \ip{M_i, \begin{pmatrix}
                        X\\ X'
                        \end{pmatrix}\begin{pmatrix}
                            X\\ X'
                            \end{pmatrix}^\intercal} +d_i = 0,\hspace{2em}\,\\
                    \forall i\in[m]
                \end{array}}.
    \end{align}
    The abundant symmetries in formulation \eqref{eq:bm_homogeneous} is an obstacle towards its (meaningful) convexification. Specifically, given any optimal $(X;X')$, we have that $(XU;X' U)$ is also optimal for any $k\times k$ orthogonal matrix $U$. In particular, $0$ is always in the convex hull of the optimizers of \eqref{eq:bm_homogeneous}. Thus, any convex reformulation of \eqref{eq:bm_homogeneous} that preserves its optimizers must admit $0\in\R^{n\by k}$ as an optimal solution. Such a solution to the convex reformulation would provide no information on an actual solution $(X;X')$ to the original problem. The strongly convex reformulation \eqref{eq:strong_convex_reform} that we present is only possible because we take advantage of the additional information $X'(X')^\intercal = Z^*$ to break the rotational invariance by \emph{fixing} $X' = (Z^*)^{1/2}$.

Additionally, the Burer-Monteiro method may fail on the class of $1$-exact QMP-like SDPs even when the factors are allowed to have rank up to the Barvinok-Pataki threshold. 
    Consider an SDP of the form
\begin{align}
    \label{eq:maxcut}
    \max_{Y\in\S^n_+}\set{\ip{C,Y}:\, \begin{array}{l}
    Y_{i,i} = 1,\,\forall i\in[n]\\
    Y\succeq 0
    \end{array}}.
\end{align}
The projection of the feasible region of this SDP is known as the elliptope. 
Let $\Pi$ denote the projection operator from $\S^n$ to $\R^{\binom{n}{2}}$ mapping a symmetric matrix to its strict upper triangular entries and let $\Pi^{-1}(\bullet)$ denote the preimage of a given subset of $\R^{\binom{n}{2}}$.
For $\sigma\in\set{\pm 1}^n$, let $\cN_\sigma\subseteq\R^{\binom{n}{2}}$ denote the normal cone 
of the elliptope at the point $\Pi(\sigma\sigma^\intercal)$. It was shown in \cite{laurent1995positive} that $\cN_\sigma$ is a full-dimensional convex cone for every $\sigma\in\set{\pm1}^n$.
Moreover, it was recently shown in \cite{waldspurger2020rank} that for any $p$ satisfying $p(p+1)/2 + p \leq n$, there exists a set $\cC$ of positive Lebesgue measure contained in $\bigcup_{\sigma} \cN_\sigma$ such that for every $C\in\Pi^{-1}(\cC)$, the Burer--Monteiro factorization with rank $p$ for \eqref{eq:maxcut} admits spurious second-order critical points.
On the other hand, for every $C\in\cC$, the SDP \eqref{eq:maxcut} has a rank-$1$ optimal solution.
As $\cN_\sigma$ is convex, its boundary $\bd(\cN_\sigma)$ has measure zero.
Then $\cC' \coloneqq \cC\setminus \bigcup_\sigma \bd(\cN_\sigma)$ continues to have positive Lebesgue measure, whence $\Pi^{-1}(\cC')$ also has positive Lebesgue measure. We claim that for all $C\in\Pi^{-1}(\cC')$, \eqref{eq:maxcut} is an exact QMP-like SDP. It is known that strict complementarity holds for $C\in \Pi^{-1}\left(\bigcup_{\sigma}\inter(\cN_\sigma)\right)$~\cite{de2019strict}. It remains to exhibit a subspace $W$ of dimension $n-1$ such that $Y^*_{W^\perp}$ is known \emph{a priori} and positive definite. Taking $W$ to be the space spanned by the first $n-1$ coordinates, we have that $Y_{W^\perp} = Y_{n,n}=1$.

\section{Algorithms for strongly convex QMMPs}
\label{sec:qmmp}
In this section, we describe and analyze an accelerated first-order method (FOM) for solving strongly convex QMMPs. While we will apply this algorithm to problems arising from the application of \cref{thm:strong_convex_reform}, the algorithms from this section can handle \emph{general} strongly convex QMMPs or general strongly convex minimax problems given an inaccurate \emph{prox-map} oracle (see \cref{def:prox_map} below).

\begin{remark}
As in \cref{sec:strongly_convex_reform}, we will assume throughout these sections that $W$ corresponds to the first $(n-k)$ coordinate subspace of $\R^n$.
When going from the coordinate subspace setting to the general setting, we will need to make changes to the definitions of $A(\gamma)$ and $B(\gamma)$ as described in \cref{rem:coordinate_to_general_submatrix_to_restrictions}.
We will additionally need to make minor modifications to the algorithms presented below.
As we will see, the algorithms
in the coordinate subspace setting
only require access to $A(\gamma)$ and $B(\gamma)$ via the following operations: First, given $\gamma\in\R^m$ and $X\in\R^{(n-k)\by k}$, we need to be able to form $A(\gamma)X\in\R^{(n-k)\by k}$.
Second, given $\gamma\in\R^m$, we need to be able to evaluate $B(\gamma)\in \R^{(n-k)\by k}$.
To adapt these algorithms to the general setting, we will require access to an ordered orthonormal basis $\cB\in \R^{n\by k}$ of $W^\perp$ and the orthogonal projection operator onto $W$, denoted $\Pi_W$.
Then, the two operations above are replaced by
\begin{gather*}
    \gamma\in\R^m, X\in\R^{W\by k} \quad\mapsto\quad A(\gamma)X = 2\Pi_W(M(\gamma) X) \in \R^{W \by k}\\
    \gamma\in\R^m \quad\mapsto\quad B(\gamma)\cB = 2\Pi_W(M(\gamma)\cB) \in \R^{W\by k}.
\end{gather*}
Note that $\Pi_W$ can be applied as $I - \cB\cB^\intercal$.

To summarize: in the general setting, there is a storage overhead of $O(nk)$ for storing $\cB$. 
The storage cost for each $X\in\R^{W\by k}$ continues to be $O(nk)$. The computational cost of each of the operations involving $A(\gamma)$ or $B(\gamma)$ is the cost of applying $M(\gamma)$ to an $n\by k$ matrix, the cost of applying $\cB^\intercal$ to an $n\by k$ matrix, and the cost of applying $\cB$ to a $k\by k$ matrix.\ifjelse{\qed}{}
\end{remark}

We state explicitly the setup and assumptions of this section.
Let $q_\obj,q_1,\dots,q_m:\R^{(n-k)\times k}\to\R$ be quadratic matrix functions of the form
\begin{align*}
q_i(X) = \frac{\tr(X^\intercal A_i X)}{2} + \ip{B_i, X} + c_i.
\end{align*}
Given $\gamma\in\R^m$, let $A(\gamma)\coloneqq A_\obj + \sum_{i=1}^m \gamma_i A_i$. Define $B(\gamma),\, c(\gamma),\, q(\gamma,X)$ analogously.

Let $\cU\subseteq\R^m$ be a compact convex set with exact projection and linear maximization oracles.
Our goal is to find an $\epsilon$-optimal solution to
\begin{align}
\tag{QMMP}
\label{eq:qmmp}
\min_{X\in\R^{(n-k)\by k}} \max_{\gamma\in\cU} q(\gamma,X).
\end{align}
That is, our goal is to find some $\tilde X\in\R^{(n-k)\times k}$ satisfying $\max_{\gamma\in\cU}q(\gamma,\tilde X) \leq \Opt_{\eqref{eq:qmmp}} + \epsilon$.
For notational convenience, we will define 
\[Q(X)\coloneqq \max_{\gamma\in\cU} q(\gamma,X).\] 
While we will treat $\cU$ as fixed in this section, in future sections, we will explicitly call attention to the dependence of the function $Q$ on the set $\cU$ and write $Q_\cU$ instead.

We present a FOM for \eqref{eq:qmmp} under two assumptions.
Both assumptions require algorithmic access to different regularity parameters of \eqref{eq:qmmp}. We will later show how to construct parameters to satisfy these assumptions but will treat them as given in the present section. 
The first assumption (\cref{as:qmmp}) requires uniform strong convexity and smoothness of $q(\gamma,X)$ over $\cU$.
\begin{assumption}
	\label{as:qmmp}
	We will assume algorithmic access to parameters $0<\mu\leq L$ such that $\mu I \preceq A(\gamma)\preceq L I$ for all $\gamma\in\cU$. 
\end{assumption}
When \cref{as:qmmp} holds, we define the \emph{condition number} of \eqref{eq:qmmp} as $\kappa \coloneqq L/\mu$.
We will state our second assumption (which bounds the norms of various quantities) when needed in \cref{subsec:fom_explicit_oracle}.

Our FOM will closely follow Nesterov's accelerated gradient descent scheme for strongly convex minimax functions \cite[Algorithm 2.3.13]{nesterov2018lectures} (henceforth AGD-MM) with one major difference. In contrast to the presentation in \cite{nesterov2018lectures} and its application in \cite{wang2020generalized}, the necessary prox-map in the QMMP setting cannot be computed explicitly or exactly.

We break our FOM for strongly convex QMMPs into two levels, presented as the first two subsections in this section. In \cref{subsec:fom_assuming_oracle}, we give a convergence analysis for a modified version of AGD-MM using an \emph{inexact} prox-map oracle.
In particular, we will bound the necessary accuracy of the prox-map to recover accelerated convergence rates.
In \cref{subsec:approximating_prox_map}, we show how to implement the approximate prox-map oracle efficiently for each iteration using the strongly convex excessive gap technique~\cite[Algorithm 6.2.37]{nesterov2018lectures}.
Finally, in \cref{subsec:fom_explicit_oracle}, we state an assumption (\cref{as:regularity}) that allows us to bound the iteration cost of the prox-map oracle uniformly across iterations. Taken together with the results from the previous subsections, this will give a rigorous guarantee for the overall FOM.

\subsection{An FOM for strongly convex QMMPs using an inexact prox-map oracle}
\label{subsec:fom_assuming_oracle}
This subsection generalizes AGD-MM by allowing \emph{inexact} prox-map computations.
We first recall the definition of the prox-map and the fundamental relation \eqref{eq:combined_inequality_exact} that is used in the convergence rate analysis of AGD-MM. 
Next, we show how to recover a similar inequality \eqref{eq:combined_inequality_inexact} when the prox-map is computed only approximately.
Finally, we show how to modify the step-sizes in AGD-MM to prevent error accumulation that may otherwise build up from inexact prox-map computations. These step-sizes allow us to recover the accelerated linear convergence rates of AGD-MM even with inexact prox-map computations.
We defer the proofs of lemmas and corollaries in this subsection to \cref{sec:deferred} as many are straightforward or follow verbatim from the exact prox-map case.

\subsubsection{The prox-map}
AGD-MM requires computing the prox-map $X_L(\Xi)$ (defined in \cref{def:prox_map}) \emph{exactly} in every iteration (adapted from \cite[Definition 2.3.2]{nesterov2018lectures}).
\begin{definition}
\label{def:prox_map}
Let $\Xi\in\R^{(n-k)\times k}$. Define
\begin{align*}
Q(\Xi; X)&\coloneqq \max_{\gamma\in \cU}\left(q(\gamma,\Xi) + \ip{\grad_2 \, q(\gamma,\Xi), X - \Xi}\right)\\
Q_L(\Xi;X) & \coloneqq Q(\Xi; X) + \frac{L}{2}\norm{X - \Xi}_F^2\\
Q_L^*(\Xi) &\coloneqq \min_{X\in\R^{(n-k)\times k}} Q_L(\Xi; X)\\
X_L(\Xi) &\coloneqq \argmin_{X\in\R^{(n-k)\times k}} Q_L(\Xi;X)\\
g^{}_L(\Xi) &\coloneqq L(\Xi - X_L(\Xi)).
\end{align*}
Here, $\grad_2 \, q(\gamma,\Xi)$ is the gradient of $q(\gamma,X)$ in $X$ at $\Xi$ and is an affine function of $\gamma$ (more explicitly, $\grad_2\, q(\gamma,\Xi) = A(\gamma)\Xi + B(\gamma)$).
Note that the function $Q(\Xi;X)$ simply replaces the inside function $q(\gamma, X)$ in the definition of $Q(X)$ with its \emph{linearization} around $\Xi$. The quantities $X_L$ and $g^{}_L$ are the \emph{prox-map} and the \emph{grad-map}.
\ifjelse{\qed}{}
\end{definition}

Recall also the main property of the prox-map and grad-map that is used in the analysis of the convergence rate of AGD-MM as given in the following lemma (adapted from \cite[Theorem 2.3.2]{nesterov2018lectures}).

\begin{lemma}
Let $\Xi\in\R^{(n-k)\times k}$. Then, for all $X\in\R^{(n-k)\times k}$,
\begin{align}
\label{eq:combined_inequality_exact}
Q(X) \geq Q(X_L(\Xi)) + \frac{1}{2L} \norm{g^{}_L(\Xi)}_F^2 + \ip{g^{}_L(\Xi), X - \Xi} + \frac{\mu}{2}\norm{X - \Xi}_F^2.
\end{align}
\end{lemma}

\subsubsection{An approximate prox-map inequality}
In the setting of general QMMPs, it is not possible to compute the prox-map exactly.
Instead, we will apply an inner FOM to solve the prox-map $X_L(\Xi)$ to some prescribed accuracy.
This necessitates an analysis of (a variant of) AGD-MM that works with inexact prox-map computations.
To this end, we show how to recover a version of \eqref{eq:combined_inequality_exact} where $X_L(\Xi)$ is computed only approximately.

Define
\begin{align}
\label{eq:tilde_L_mu_kappa}
\tilde \mu &\coloneqq \mu/2,\qquad
\tilde L \coloneqq L - \mu/2,\quad\text{and}\quad
\tilde \kappa \coloneqq \tilde L/\tilde \mu.
\end{align}

The following geometric fact (\cref{lem:quadratic_moving_center}) will allow us to derive a version of a variant of \eqref{eq:combined_inequality_exact} which only uses an approximate prox-map (\cref{thm:combined_inequality_eps}).
\begin{lemma}
\label{lem:quadratic_moving_center}
Let $\tilde X,\,X_L\in\R^{(n-k)\times k}$ be such that $\norm{\tilde X - X_L}_F \leq \delta$. Then, for all $X\in\R^{(n-k)\times k}$,
\begin{align*}
\frac{L}{2}\norm{X - X_L}_F^2 \geq \frac{\tilde L}{2}\norm{X - \tilde X}_F^2 - \frac{L\delta^2}{2}\left(2\kappa - 1\right).
\end{align*}
\end{lemma}

\begin{theorem}
\label{thm:combined_inequality_eps}
Let $\Xi\in\R^{(n-k)\times k}$. Suppose $\tilde X$ satisfies
\begin{align*}
Q_L(\Xi; \tilde X) \leq Q^*_L(\Xi) + \epsilon.
\end{align*}
Set $\tilde g \coloneqq \tilde L (\Xi - \tilde X)$. Then, for all $X\in\R^{(n-k)\times k}$,
\begin{align}
\label{eq:combined_inequality_inexact}
Q(X) &\geq Q(\tilde X) + \frac{1}{2\tilde L} \norm{\tilde g}_F^2 + \ip{\tilde g, X - \Xi} + \frac{\tilde \mu}{2}\norm{X - \Xi}_F^2 - 2\kappa\epsilon.
\end{align}
\end{theorem}
\begin{proof}
As $Q_L(\Xi; X)$ is $L$-strongly convex, from the premise of the lemma we have $\norm{\tilde X - X_L(\Xi)}_F\leq \sqrt{2\epsilon/L}$.

We bound
\begin{align*}
Q(X) &\geq Q(\Xi; X) + \frac{\mu}{2}\norm{X - \Xi}_F^2\\
&= Q_L(\Xi; X) - \frac{\tilde L}{2}\norm{X - \Xi}_F^2  + \frac{\tilde \mu}{2}\norm{X - \Xi}_F^2\\
&\geq Q^*_L(\Xi) + \frac{L}{2}\norm{X - X_L(\Xi)}_F^2  - \frac{\tilde L}{2}\norm{X - \Xi}_F^2  + \frac{\tilde \mu}{2}\norm{X - \Xi}_F^2\\
&\geq Q(\tilde X) + \frac{\tilde L}{2}\norm{X - \tilde X}_F^2  - \frac{\tilde L}{2}\norm{X - \Xi}_F^2  + \frac{\tilde \mu}{2}\norm{X - \Xi}_F^2 - 2\kappa\epsilon\\
&= Q(\tilde X) + \frac{\tilde L}{2}\left(2\ip{X - \Xi, \Xi - \tilde X} + \norm{\Xi - \tilde X}_F^2\right)  + \frac{\tilde \mu}{2}\norm{X - \Xi}_F^2 - 2\kappa\epsilon\\
&= Q(\tilde X) + \ip{\tilde g, X - \Xi} + \frac{1}{2\tilde L} \norm{\tilde g}_F^2 + \frac{\tilde \mu}{2}\norm{X - \Xi}_F^2 - 2\kappa\epsilon.
\end{align*}
Here, the first inequality follows from $\mu$-strong convexity of $Q$, the first equation follows from the definitions of $Q_L(\Xi;X)$, $\tilde L$ and $\tilde \mu$, 
the second inequality follows from optimality of $X_L(\Xi)$, the third inequality follows from \cref{lem:quadratic_moving_center} applied with $\delta = \sqrt{2\epsilon/L}$ and the $L$-smoothness of $q(\gamma,X)$ for each $\gamma\in\cU$, and the last two equations follow from expanding the squares and the definition of $\tilde g$.
\end{proof}

\subsubsection{Estimating sequences}
We now modify the estimating sequences analysis of AGD-MM to use \eqref{eq:combined_inequality_inexact} instead of \eqref{eq:combined_inequality_exact}:
Fix $X_0\in\R^{(n-k)\times k}$ and let $\set{\epsilon_t}\subseteq\R_{++}$ and $\set{\Xi_t}\subseteq \R^{(n-k)\times k}$ to be fixed later. Define
\begin{align*}
\phi_0(X) &\coloneqq Q(X_0) + \frac{\tilde \mu}{2}\norm{X - X_0}_F^2.
\end{align*}
For $t \geq 0$, let $X_{t+1}$ be an $\epsilon$-approximate prox-map, i.e., $X_{t+1}$ satisfies
\begin{align*}
Q^{}_L(\Xi_t; X_{t+1}) \leq Q^*_L(\Xi_t) + \epsilon_t,
\end{align*}
and set $\tilde g_t \coloneqq \tilde L(\Xi_t - X_{t+1})$.
Let $\alpha \coloneqq \tilde\kappa^{-1/2}$ and recursively define
\begin{align*}
\phi_{t+1}(X) &\coloneqq (1-\alpha)\phi_t(X) + \alpha \left( Q(X_{t+1}) + \frac{1}{2\tilde L}\norm{\tilde g_t}_F^2 + \ip{\tilde g_t, X - \Xi_t} + \frac{\tilde \mu}{2}\norm{X - \Xi_t}_F^2 \right).
\end{align*}

The following lemma shows how $\phi_t(X)$ evolves. Its proof follows verbatim from the standard proof~\cite[Lemma 2.3.3]{nesterov2018lectures}, which makes no assumption on how $\Xi_t$ and $X_t$ are related.

\begin{lemma}
\label{lem:estimating_sequence_recurrence}
For all $t\geq 0$, $\phi_t(X)$ is a quadratic matrix function in $X$ of the form
\begin{align}
\label{eq:estimating_sequence_form}
\phi_t(X) = \phi_t^* + \frac{\tilde \mu}{2}\norm{X - V_t}_F^2.
\end{align}
The sequences $\set{\phi_t^*}$, $\set{V_t}$
are given by $V_0 = X_0$, $\phi_0^* = Q(X_0)$ and the recurrences
\begin{align*}
V_{t+1} &= (1-\alpha) V_t + \alpha \left(\Xi_t - \frac{1}{\tilde \mu}\tilde g_t\right),\text{ and}\\
\phi_{t+1}^*&= (1-\alpha)\phi_t^* + \alpha\left(Q(X_{t+1}) + \frac{1}{2\tilde L}\norm{\tilde g_t}_F^2\right) -\frac{\alpha^2}
{2\tilde \mu}\norm{\tilde g_t}_F^2\\
&\qquad 
+ \alpha(1-\alpha)\left(\frac{\tilde \mu}{2}\norm{\Xi_t - V_t}_F^2 + \ip{\tilde g_t, V_t - \Xi_t}\right).
\end{align*}
\end{lemma}

For all $t \geq 0$, we will henceforth set 
\begin{align*}
\Xi_t \coloneqq \frac{X_t + \alpha V_t}{1+\alpha}.
\end{align*}
The following lemma shows that $\Xi_{t+1}$ can be written as an \emph{extragradient} step from $X_t$ towards $X_{t+1}$. Its proof follows verbatim from the standard proof~\cite[Page 92]{nesterov2018lectures}, which only requires the relation $X_{t+1} = \Xi_t - \tilde g_t/\tilde L$.
\begin{lemma}
\label{lem:Xi_extragradient}
It holds that $\Xi_0 = X_0$ and $\Xi_{t+1} = X_{t+1} + \frac{1-\alpha}{1+\alpha}(X_{t+1} - X_t)$ for all $t\geq 0$.
\end{lemma}

The following two lemmas bound the two types of errors that result from inexact prox-map computations. Define $E_0^{(1)} \coloneqq 0$, $E_0^{(2)}\coloneqq 0$, and for all $t\geq 0$ inductively set
\begin{align*}
E_{t+1}^{(1)} \coloneqq (1-\alpha)E_t^{(1)} + (1-\alpha)\epsilon_t
\qquad\text{and}\qquad
E_{t+1}^{(2)} \coloneqq (1-\alpha) E_t^{(2)} + \alpha \epsilon_t.
\end{align*}
Let $E_t \coloneqq E^{(1)}_t + E^{(2)}_t$ be the sum of the two types of errors. Equivalently, let $E_t \coloneqq 0$ and inductively set $E_{t+1} = (1-\alpha)E_t + \epsilon_t$ for all $t\geq 0$. Then, it holds that:

\begin{lemma}
\label{lem:error_inexact_prox_1}
It holds that $Q(X_t) \leq \phi_t^* + 2\kappa E^{(1)}_t$ for all $t \geq 0$.
\end{lemma}

\begin{lemma}
\label{lem:error_inexact_prox_2}
For all $t\geq 0$, it holds that
\begin{align*}
\phi_t(X) \leq (1 - (1-\alpha)^t)Q(X) + (1-\alpha)^t \phi_0(X) + 2\kappa E_t^{(2)},\quad\forall X\in\R^{(n-k)\times k}.
\end{align*}
\end{lemma}

Combining \cref{lem:error_inexact_prox_1,lem:error_inexact_prox_2}, we get a bound on the total error due to inexact prox-maps as a function of the accuracy of each individual prox-map.
\begin{corollary}
\label{cor:qmmp_subopt_E}
For all $t\geq 0$, it holds that
\begin{align*}
Q(X_t)-\Opt_{\eqref{eq:qmmp}} &\leq (1-\alpha)^t\left[2\left(Q(X_0) - \Opt_{\eqref{eq:qmmp}}\right)\right] + 2\kappa E_t.
\end{align*}
\end{corollary}

We are now ready to present CautiousAGD (\cref{alg:cautious-agd}) and its guarantee.

\begin{algorithm}
\caption{CautiousAGD}
\label{alg:cautious-agd}
Given $q(\gamma,X)$ and $\cU$ satisfying \cref{as:qmmp}; $X_0\in\R^{(n-k)\times k}$, and a bound $\textup{gap}_0\in\R$ such that $Q(X_0) - \Opt_{\eqref{eq:qmmp}} \leq \textup{gap}_0$
{\small\begin{enumerate}[topsep=0pt,itemsep=0pt,parsep=0pt]
	\item Set $\tilde\mu,\tilde L,\tilde\kappa$ as in \eqref{eq:tilde_L_mu_kappa} and $\alpha  \coloneqq \tilde\kappa^{-1/2}$. Set $\Xi_0 \coloneqq X_0$.
	\item For $t \geq 0$
	\begin{enumerate}[topsep=0pt,itemsep=0pt,parsep=0pt]
		\item Compute an inexact prox-map $X_{t+1}$ satisfying
		\begin{align}
		\label{eq:cautious_agd_prox_map}
		Q_L(\Xi_t; X_{t+1}) \leq Q^*_L(\Xi_t) + \epsilon_t,\quad
		\text{where}\quad
		\epsilon_t = \begin{cases}
			\frac{\textup{gap}_0}{\kappa}\left(1-\frac{\alpha}{2}\right), & \text{if }t = 0, \text{ and}\\
			\frac{\textup{gap}_0}{\kappa}\left(1 - \frac{\alpha}{2}\right)^{t} \frac{\alpha}{2}, & \text{else}.
		\end{cases} 
		\end{align}
		\item Set $\Xi_{t+1}  \coloneqq X_{t+1} + \frac{1-\alpha}{1+\alpha} \left(X_{t+1} - X_t\right)$
	\end{enumerate}
\end{enumerate}}
\end{algorithm}

\begin{theorem}
\label{thm:cautious_agd_oracle}
Let $q(\gamma,X)$ and $\cU$ satisfy \cref{as:qmmp}.
Let $\textrm{gap}_0$ be a known upper bound on $Q(X_0) - \Opt_{\eqref{eq:qmmp}}$ and
let $X_t$ denote the iterates produced by \cref{alg:cautious-agd} with starting point $X_0$.
Then, for all $t \geq 1$, the iterate $X_t$ satisfies
\begin{align*}
Q(X_t)-\Opt_{\eqref{eq:qmmp}} &\leq \left(1 - \frac{\alpha}{2}\right)^t \left(4\cdot\textup{gap}_0\right).
\end{align*}
In particular, $Q(X_T)-\Opt_{\eqref{eq:qmmp}} \leq \epsilon$ after at most
$T = O\left(\sqrt{\kappa}\log\left(\frac{\textup{gap}_0}{\epsilon}\right)\right)$
iterations. The $t$-th iteration requires computing a prox-map $X_{t+1}$ satisfying \eqref{eq:cautious_agd_prox_map}.
\end{theorem}
\begin{proof}
We first claim that $E_t = (\textup{gap}_0/\kappa)(1- \alpha/2)^t$ for all $t \geq 1$.
Indeed, this claim holds for $t = 1$ as $E_1 = \epsilon_0 = (\textup{gap}_0/\kappa)\left(1 - \alpha/2\right)$ by construction (see \eqref{eq:cautious_agd_prox_map}). Then, by induction
\begin{align*}
E_{t+1} &= (1-\alpha)E_t + \epsilon_t\\
&= (1-\alpha)\frac{\textup{gap}_0}{\kappa}\left(1-\frac{\alpha}{2}\right)^{t} + \frac{\textup{gap}_0}{\kappa}\left(1-\frac{\alpha}{2}\right)^{t}\frac{\alpha}{2} = \frac{\textup{gap}_0}{\kappa}\left(1 - \frac{\alpha}{2}\right)^{t+1}.
\end{align*}
Then, the bound on $Q(X_t) - \Opt_{\eqref{eq:qmmp}}$ follows from \cref{cor:qmmp_subopt_E} and the starting condition $Q(X_0) - \Opt_{\eqref{eq:qmmp}} \leq \textup{gap}_0$.
\end{proof}

\begin{remark}
We refer to \cref{alg:cautious-agd} as CautiousAGD to allude to the fact that \cref{alg:cautious-agd} is simply AGD-MM with inexact prox-maps and smaller extra-gradient steps. Specifically, AGD-MM and CautiousAGD set
\begin{gather*}
\Xi_{t+1} = X_{t+1} + \left(\frac{1-\kappa^{-1/2}}{1+\kappa^{-1/2}}\right)(X_{t+1} - X_t),\quad\text{and}\\
\Xi_{t+1} = X_{t+1} + \left(\frac{1-\tilde\kappa^{-1/2}}{1+\tilde\kappa^{-1/2}}\right)(X_{t+1} - X_t)
\end{gather*}
respectively. Note that $\kappa \leq \tilde \kappa \leq 2\kappa$.
\ifjelse{\qed}{}
\end{remark}

\subsection{Approximating the prox-map}
\label{subsec:approximating_prox_map}

Recall that the prox-map $X_L(\Xi)$ is the minimizer of $Q_L(\Xi;X)$:
\begin{align*}
&\min_{X\in\R^{(n-k)\times k}} Q_L(\Xi;X)= \min_{X\in\R^{(n-k)\times k}}\max_{\gamma\in\cU}\left(\frac{L}{2}\norm{X - \Xi}_F^2 + \ip{\grad_2\,q(\gamma,\Xi),X - \Xi} + q(\gamma,\Xi)\right).
\end{align*}
There are a number of ways to solve for $X_L(\Xi)$.
For example, when $m$ is small, one may apply an interior point method to solve for $\gamma$ in the dual problem:
\begin{align}
\label{eq:adjoint_problem}
&\max_{\gamma\in\cU}\left[\min_{X\in\R^{(n-k)\times k}}\left(\frac{L}{2}\norm{X - \Xi}_F^2 + \ip{\grad_2\,q(\gamma,\Xi),X - \Xi} + q(\gamma,\Xi)\right)\right]\nonumber\\
&\qquad = \max_{\gamma\in\cU} 
\left(-\frac{1}{2L}\norm{\grad_2\,q(\gamma,\Xi)}_F^2 + q(\gamma,\Xi)\right).
\end{align}
Note here that strong duality holds as the term inside the parenthesis is linear in $\gamma$ and convex quadratic in $X$ and $\cU$ is a compact convex set so we can apply Sion's Minimax Theorem~\cite{sion1958general}. 
An approximate primal solution $\tilde X$ can then be reconstructed from an approximate solution $\tilde\gamma$ of the dual problem by setting $\tilde X = \Xi - \frac{\grad_2\, q(\tilde\gamma, \Xi)}{L}$.

Sticking with FOMs, one may apply the strongly convex excessive gap technique \cite[Chapter 6.2]{nesterov2018lectures} to compute the prox-map $X_L(\Xi)$ as well.
We will rewrite $Q_L(\Xi;X)$ in a form that is more natural for applying the excessive gap technique~\cite[Algorithm 6.2.37]{nesterov2018lectures}.
Note that $\grad_2\, q(\gamma,\Xi) = A(\gamma)\Xi + B(\gamma)$.
Thus, defining the matrix $G_\obj \coloneqq A_\obj \Xi + B_\obj$ and the linear operator $\cG:\gamma \mapsto \sum_{i=1}^m \gamma_i\left(A_i\Xi + B_i\right)$, we have $\grad_2\, q(\gamma,\Xi) = A(\gamma)\Xi + B(\gamma) = G_\obj + \cG\gamma$. Hence, we arrive at 
\begin{align}
\label{eq:prox_map}
Q_L(\Xi;X) &= \frac{L}{2}\norm{X - \Xi}_F^2+ \ip{G_\obj, X - \Xi} + \max_{\gamma\in\cU}\set{\ip{\cG\gamma, X} + \left(q(\gamma,\Xi) - \ip{\cG\gamma,\Xi}\right)}.
\end{align}
The inner saddle-point function 
is strongly convex in $X$ and linear in $\gamma$ so that we may approximate the prox-map by approximately solving a strongly convex--concave saddle point problem.
Thus, applying \cite[Theorem 6.2.4]{nesterov2018lectures} to $Q_L(\Xi;X)$ in the form \eqref{eq:prox_map} gives the following result.
\begin{theorem}
\label{thm:oracle_cost}
Initialize \cite[Theorem 6.2.4]{nesterov2018lectures} with initial iterate $\gamma_0\in\cU$.
Let $(\tilde \gamma, \tilde X)$ denote the output of \cite[Algorithm 6.2.37]{nesterov2018lectures} after
\begin{align*}
O\left(\frac{\max_{\gamma\in\bS^{m-1}}\norm{\cG\gamma}_F \cdot \max_{\gamma\in\cU}\norm{\gamma - \gamma_0}_2}{\sqrt{L\epsilon}}\right)
\end{align*}
iterations. Here, each iteration may require two exact projections onto $\cU$. Then,
\begin{align}
\label{eq:saddle_point_error}
Q_L(\Xi;\tilde X) - Q^*_L(\Xi) \leq Q_L(\Xi;\tilde X) - \left(q(\bar\gamma_k, \Xi) - \frac{\norm{\grad_2\, q(\bar \gamma_k,\Xi)}_F^2}{2L}\right) \leq \epsilon.
\end{align}
\end{theorem}
\begin{remark}\label{rem:prox_map_agd}
For simplicity, in our numerical implementation of CertSDP, we opt to run the accelerated gradient descent method for simple sets~\cite[Algorithm 2.2.63]{nesterov2018lectures} on the dual problem \eqref{eq:adjoint_problem}.
\ifjelse{\qed}{}
\end{remark}

\subsection{Putting the pieces together}
\label{subsec:fom_explicit_oracle}
We conclude this section by showing how to combine \cref{thm:cautious_agd_oracle,thm:oracle_cost} to get a guarantee on the total iteration count (including iterations \emph{within} the inexact prox-map calls).
To this end, we will need an additional assumption on the norms of various quantities.

\begin{assumption}
\label{as:regularity}
	Suppose \cref{alg:cautious-agd} starts at $X_0 = 0_{(n-k)\times k}$ and
$R>0$ satisfies
	\begin{gather*}
	\frac{\norm{\grad_2\,q(\tilde\gamma,X_0)}_F}{L} = \frac{\norm{B(\tilde\gamma)}_F}{L}\leq R
	\end{gather*}
	where $\tilde\gamma\in\argmax_{\gamma\in\cU}q(\gamma,X_0)$. Let $D$ denote the diameter of $\cU$; this is the natural scale parameter for the dual iterates.
	We will see soon that $R$ is a natural scale parameter for the primal iterates $X_t,\,\Xi_t\in\R^{(n-k)\times k}$. 
	Suppose $H\geq 1$ bounds
	\begin{gather*}
	\frac{D\norm{\sum_{i=1}^m\gamma_i A_i}_2}{\mu},
	\quad\text{and}\quad
	\frac{D\norm{\sum_{i=1}^m\gamma_i B_i}_F}{\mu\kappa R}
	\end{gather*}
	for all $\gamma\in \bS^{m-1}$. We will assume algorithmic access to $D$, $H$, and $R$.
\end{assumption}

\begin{remark}
The algorithm CautiousAGD will be called repeatedly from a parent algorithm CertSDP (to be defined in \cref{sec:solving_exact_sdps}). This parent algorithm CertSDP will define $\mu, L, D, H, R$ in such a way that \cref{as:qmmp,as:regularity} are satisfied. The necessary $\textup{gap}_0$ used in CautiousAGD is also defined in terms of these quantities in the following lemma.
\ifjelse{\qed}{}
\end{remark}

The following two lemmas allow us to bound $\norm{\Xi_{t}}_F$ and the operator norm
$\max_{\gamma\in\bS^{m-1}} \norm{\cG\gamma}_F$ in
\cref{thm:oracle_cost}. Their proofs are deferred to \cref{sec:deferred}.

\begin{lemma}
    \label{lem:xi_t_norm_bound}
Under~\cref{as:regularity}, it holds that $Q(X_0)- \Opt_{\eqref{eq:qmmp}} \leq \frac{\mu\kappa^2R^2}{2}$. In particular, we may take $\textup{gap}_0 = \frac{\mu\kappa^2R^2}{2}$ in \cref{alg:cautious-agd}. Then, for every $t\geq 0$, the iterate $\Xi_t$ computed by \cref{alg:cautious-agd} satisfies
$\norm{\Xi_{t}}_F \leq 10\kappa R$.
\end{lemma}

\begin{lemma}
\label{lem:grad_hessian_psi}
Suppose~\cref{as:regularity} holds and we set $\textup{gap}_0 = \frac{\mu\kappa^2 R^2}{2}$ in \cref{alg:cautious-agd}. Then, for every iterate $t\geq 0$, we have
\begin{gather*}
\max_{\gamma\in\bS^{m-1}} \norm{\cG\gamma}_F \leq 11\frac{\mu \kappa H R}{D}.
\end{gather*}
\end{lemma}

The following theorem
gives the iteration complexity of \cref{alg:cautious-agd}
instantiated with the excessive gap technique to compute the prox-map.
It follows as a corollary to \cref{thm:oracle_cost,thm:cautious_agd_oracle,lem:grad_hessian_psi}.
\begin{theorem}
\label{thm:overall_qmmp_iterations}
Let $q(\gamma, X)$ and $\cU$ satisfy \cref{as:qmmp,as:regularity}.
Suppose $\textup{gap}_0$ is set to $\frac{\mu\kappa^2R^2}{2}$ in \cref{alg:cautious-agd}.
Let $X_t$ denote the iterates produced by \cref{alg:cautious-agd} with starting point $X_0 = 0_{(n-k)\times k}$. Then, for all $t\geq 1$, the iterate $X_t$ satisfies
\begin{align*}
Q(X_t) - \Opt_{\eqref{eq:qmmp}} \leq \left(1 - \frac{\alpha}{2}\right)^t \left(2\mu\kappa^2 R^2\right).
\end{align*}
In particular, $Q(X_T) - \Opt_{\eqref{eq:qmmp}} \leq \epsilon$ after at most
\begin{align*}
T = O\left(\sqrt{\kappa}\log\left(\frac{\mu\kappa^2 R^2}{\epsilon}\right)\right)
\end{align*}
outer iterations of \cref{alg:cautious-agd}. The iterate $X_T$ is computed after a total (including iterations within the inexact prox-map computations) of $O\left(\frac{\kappa^{5/4} HR \sqrt{L}}{\sqrt{\epsilon}}\right)$ iterations.
\end{theorem}
\begin{proof}
We will take $T$ to be the first positive integer such that
\begin{align*}
\left(1 - \frac{\alpha}{2}\right)^{T}\left(2\mu\kappa^2R^2\right) \leq \epsilon.
\end{align*}
Clearly, $T = O\left(\sqrt{\kappa}\log\left(\kappa LR^2/\epsilon\right)\right)$.
Next, if $T>1$, then by the maximality of $T$ we have 
\begin{align*}
\left(1 - \frac{\alpha}{2}\right)^{T} \geq \left(1-\frac{\alpha}{2}\right)\left(\frac{\epsilon}{2\mu\kappa^2R^2}\right) \geq \frac{\epsilon}{4\mu\kappa^2R^2}.
\end{align*}
From~\eqref{eq:cautious_agd_prox_map} and $\textup{gap}_0=\frac{\mu\kappa^2R^2}{2}$, we deduce that $\epsilon_t \geq \frac{\mu\kappa R^2}{2}\left(1 - \frac{\alpha}{2}\right)^t \frac{\alpha}{2}$. 
By \cref{lem:grad_hessian_psi} and \cref{as:regularity}, we may bound
\begin{align*}
\max_{\gamma\in\bS^{m-1}}\norm{\cG\gamma}_F \cdot \max_{\gamma\in\cU}\norm{\gamma - \gamma_0}_2 \leq 11 \mu\kappa HR.
\end{align*}

Thus, $X_t$ can be computed in
\begin{align*}
O\left(\frac{\mu\kappa H R}{\sqrt{L\epsilon_t}}\right) = O\left(\kappa^{1/4}H(1-\alpha/2)^{-t/2}\right)
\end{align*}
iterations. Summing over the first $T$ outer iterations and observing our lower bound on $\left(1 - \frac{\alpha}{2}\right)^T$, we have that
\begin{align*}
\sum_{t=0}^T \left(1 - \frac{\alpha}{2}\right)^{-t/2} &\leq O\left(\frac{\left(1-\frac{\alpha}{2}\right)^{-T/2}}{\alpha}\right) = O\left(\frac{\kappa R\sqrt{L}}{\sqrt{\epsilon}}\right).\qedhere
\end{align*}
\end{proof}

\section{Solving $k$-exact SDPs via strongly convex QMMP algorithms}
\label{sec:solving_exact_sdps}
In this section, we show how to combine \cref{thm:overall_qmmp_iterations,thm:strong_convex_reform} to develop first-order methods for approximately solving rank-$k$ exact QMP-like SDPs. We will use the following notion of an approximate solution to \eqref{eq:primal_dual_SDP}.

\begin{definition}
\label{def:eps_opt_eps_feas}
We will say that $\tilde Y\in\S^n$ is \emph{$\epsilon$-optimal and $\delta$-feasible} for \eqref{eq:primal_dual_SDP} if $\tilde Y \succeq 0$,
\begin{gather*}
\ip{M_\obj, \tilde Y} \leq \Opt_{\eqref{eq:primal_dual_SDP}} + \epsilon,\quad\text{and}\quad
\left(\sum_{i=1}^m \left(\ip{M_i, \tilde Y} + d_i\right)^2\right)^{1/2} \leq \delta.\ifjelse{\qed}{}
\end{gather*}
\end{definition}

The final piece towards this goal is developing algorithms for constructing a certificate of strict complementarity $\cU$.

By \cref{def:subspace_certificate}, the properties we need to ensure for $\cU$ are that $\gamma^*\in\cU$ and $A(\gamma)\succ 0$ for all $\gamma\in\cU$.
We will construct $\cU$ by taking it to be an $\ell_2$-ball centered at a sufficiently accurate estimate $\tilde\gamma$ of $\gamma^*$.

Recall that $A(\gamma^*)\succ 0$ (see \cref{lem:complementary_subspace}).
Clearly then,
for all $\tilde\gamma$ close enough to $\gamma^*$, we have that $A(\tilde\gamma)\succ 0$ and there exists some $r>0$ such that $\tilde \cU\coloneqq \mathbb{B}(\tilde\gamma, r)$ satisfies $A(\gamma)\succ 0$ for all $\gamma\in\tilde \cU$.
We consider one setting for $r$ below.
It remains to ask, does the condition that $\gamma^*\in\tilde \cU$ hold?
Below, we show
that this condition indeed holds when $\tilde\gamma$ is a sufficiently accurate estimate of $\gamma^*$ and that we can effectively check this condition using CautiousAGD.

\begin{figure}[t]
	\centering
	 \def\svgwidth{250pt}
     \begingroup%
     \makeatletter%
     \providecommand\color[2][]{%
       \errmessage{(Inkscape) Color is used for the text in Inkscape, but the package 'color.sty' is not loaded}%
       \renewcommand\color[2][]{}%
     }%
     \providecommand\transparent[1]{%
       \errmessage{(Inkscape) Transparency is used (non-zero) for the text in Inkscape, but the package 'transparent.sty' is not loaded}%
       \renewcommand\transparent[1]{}%
     }%
     \providecommand\rotatebox[2]{#2}%
     \newcommand*\fsize{\dimexpr\f@size pt\relax}%
     \newcommand*\lineheight[1]{\fontsize{\fsize}{#1\fsize}\selectfont}%
     \ifx\svgwidth\undefined%
       \setlength{\unitlength}{216bp}%
       \ifx\svgscale\undefined%
         \relax%
       \else%
         \setlength{\unitlength}{\unitlength * \real{\svgscale}}%
       \fi%
     \else%
       \setlength{\unitlength}{\svgwidth}%
     \fi%
     \global\let\svgwidth\undefined%
     \global\let\svgscale\undefined%
     \makeatother%
     \begin{picture}(1,0.45833333)%
       \lineheight{1}%
       \setlength\tabcolsep{0pt}%
       \put(0.65643724,0.01870465){\color[rgb]{0,0,0}\makebox(0,0)[lt]{\lineheight{1.25}\smash{\begin{tabular}[t]{l}$\set{\gamma\in\R^m:\,A(\gamma)\succeq\tfrac{\hat\mu}{2} I}$\end{tabular}}}}%
       \put(0,0){\includegraphics[width=\unitlength,page=1]{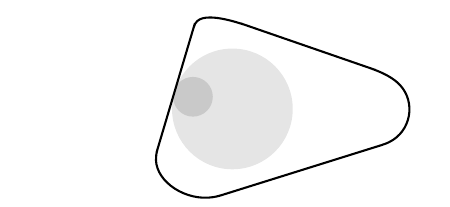}}%
       \put(0.57103893,0.10889402){\color[rgb]{0,0,0}\makebox(0,0)[lt]{\lineheight{1.25}\smash{\begin{tabular}[t]{l}$\gamma^*$\end{tabular}}}}%
       \put(0,0){\includegraphics[width=\unitlength,page=2]{subgradient.pdf}}%
       \put(0.06234405,0.06421759){\color[rgb]{0,0,0}\makebox(0,0)[lt]{\lineheight{1.25}\smash{\begin{tabular}[t]{l}$\gamma^{(1)}$\end{tabular}}}}%
       \put(0,0){\includegraphics[width=\unitlength,page=3]{subgradient.pdf}}%
       \put(0.27813723,0.15354747){\color[rgb]{0,0,0}\makebox(0,0)[lt]{\lineheight{1.25}\smash{\begin{tabular}[t]{l}$\gamma^{(2)}$\end{tabular}}}}%
       \put(0.39240662,0.17408517){\color[rgb]{0,0,0}\makebox(0,0)[lt]{\lineheight{1.25}\smash{\begin{tabular}[t]{l}$\gamma^{(3)}$\end{tabular}}}}%
       \put(0.47983827,0.1484704){\color[rgb]{0,0,0}\makebox(0,0)[lt]{\lineheight{1.25}\smash{\begin{tabular}[t]{l}$\gamma^{(4)}$\end{tabular}}}}%
       \put(0.47151788,0.27911434){\color[rgb]{0,0,0}\makebox(0,0)[lt]{\lineheight{1.25}\smash{\begin{tabular}[t]{l}$\cU^{(3)}$\end{tabular}}}}%
       \put(0.65028477,0.29396835){\color[rgb]{0,0,0}\makebox(0,0)[lt]{\lineheight{1.25}\smash{\begin{tabular}[t]{l}$\cU^{(4)}$\end{tabular}}}}%
     \end{picture}%
   \endgroup%
   
	\caption{CertSDP (\cref{alg:exact_sdp_qmmp}) produces a series of iterates $\gamma^{(i)}\to\gamma^*$. For each $\gamma^{(i)}$, CertSDP constructs a ball $\cU^{(i)}$ around $\gamma^{(i)}$.
	Intuitively, we want to pick $\cU^{(i)}$
to be the largest ball around $\gamma^{(i)}$ for which we can solve the associated QMMP efficiently, in hopes of enclosing $\gamma^*$. We will thus choose $\cU^{(i)}$ to satisfy certain regularity estimates (see \eqref{eq:r} and \cref{lem:parameters_for_qmmp}).	At the minimum, we will ensure
	$A(\gamma)\succeq \hat\mu/2$ for all $\gamma\in\cU^{(i)}$.
	}
\end{figure}

The following assumption collects the regularity assumptions we need to make in our algorithm CautiousAGD.
\begin{assumption}
\label{as:primal_dual_regularity}
	Suppose we have algorithmic access to
	\begin{itemize}
		\item parameters $0<\hat\mu\leq \hat L$ such that $\hat\mu I \preceq A(\gamma^*)\preceq \hat LI$,
		\item parameters $\hat R_p,\, \hat R_d>0$ such that $\norm{X^*}_F\leq \hat R_p$ and $\norm{\gamma^*}_2 \leq \hat R_d$, and
		\item a parameter $\hat \rho > 0$ upper bounding
		\begin{gather*}
		\frac{\hat\mu}{\hat R_d},\qquad
		\norm{\sum_{i=1}^m \gamma_i A_i}_2, \quad\text{and}\quad
		\frac{\norm{\sum_{i=1}^m \gamma_i B_i}_F}{\hat R_p}\quad\forall \gamma\in\bS^{m-1}.
		\end{gather*}
	\end{itemize}
	For notational simplicity, we will additionally assume $\hat R_p,\,\hat R_d \geq 1$. This is not strictly necessary and simply allows us to write $O(\hat R_p)$ and $O(\hat R_d)$ in place of $O(1+\hat R_p)$ and $O(1 + \hat R_d)$.
\end{assumption}

Note from the identity $X^* = -A(\gamma^*)^{-1}B(\gamma^*)$ that $\norm{B(\gamma^*)}_F \leq \hat L \hat  R_p$.
Now, suppose $\gamma^{(1)},\,\gamma^{(2)},\dots$ is a sequence converging to $\gamma^*$ (such a sequence can be constructed via subgradient methods~\cite{levy2018online}; see also \cite[Section 6.2.2]{ding2021optimal}).
Given $\gamma^{(i)}$, define
\begin{align}
\label{eq:r}
r^{(i)}\coloneqq &\min\left(
\frac{\hat\mu}{2\hat\rho},\,
2 \hat R_d - \norm{\gamma^{(i)}}_2,~\frac{\lambda_{\min}\left(A\left(\gamma^{(i)}\right)\right) - \hat\mu/2}{\hat\rho},\,\right.\nonumber\\
&\qquad\qquad\qquad\left.
\frac{2\hat L - \lambda_{\max}\left(A\left(\gamma^{(i)}\right)\right)}{\hat\rho},~
\frac{2\hat L\hat R^{}_p - \norm{B\left(\gamma^{(i)}\right)}_F}{\hat\rho\hat R^{}_p}
\right).
\end{align}
If $r^{(i)}$ is positive, define $\cU^{(i)}\coloneqq\mathbb{B}(\gamma^{(i)},r^{(i)})$.

We present three lemmas below.
The first lemma states that $r^{(i)}$ is positive and $\gamma^*\in\cU^{(i)}$ for all $\gamma^{(i)}$ sufficiently close to $\gamma^*$.
The second lemma establishes parameters for which the regularity conditions of \cref{as:regularity} hold for $q(\gamma,X)$ along with $\cU^{(i)}$. 
Finally, the third lemma shows that for each $\cU^{(i)}$, an approximate solution of the corresponding strongly convex QMMP (which can be computed using \cref{alg:cautious-agd}) can be used to either produce an approximate optimizer of the underlying SDP or declare $\gamma^*\notin\cU^{(i)}$.
The proofs of the first two lemmas are deferred to \cref{sec:deferred}.

\begin{lemma}
\label{lem:gamma_necessary_accuracy}
Suppose \cref{as:primal_dual_regularity} holds. Then, $r^{(i)}$ is positive and $\gamma^*\in\cU^{(i)}$ if
\begin{align*}
\norm{\gamma^{(i)} - \gamma^*}_2 \leq \frac{\hat\mu}{4\hat\rho}.
\end{align*}
\end{lemma}

\begin{lemma}
\label{lem:parameters_for_qmmp}
Suppose \cref{as:primal_dual_regularity} holds and $r^{(i)}$ is positive. Then, $q(\gamma,X)$ and $\cU^{(i)}$ satisfy \cref{as:regularity} with
$\mu = \frac{\hat\mu}{2}$,
$L = 2\hat L$, $R = \hat R_p$,
$D = 2r^{(i)}$, and $H = 2$.
\end{lemma}

\begin{lemma}
\label{lem:sdp_to_qmmp_error_parameters}
Suppose \cref{as:primal_dual_regularity} holds, $r^{(i)}$ is positive, and $0<\epsilon\leq 9\hat\rho\hat R_d\hat R_p^2$.
Set $\delta \coloneqq \frac{\hat\mu \epsilon^2}{\left(9\hat\rho\hat R_d \hat R_p\right)^2}$ and $\eta \coloneqq \frac{4\epsilon}{9\hat R_d}$. Suppose $\tilde X\in\R^{(n-k)\times k}$ satisfies
\begin{align*}
Q_{\cU^{(i)}}\left(\tilde X\right) \leq \min_{X\in\R^{(n-k)\times k}} Q_{\cU^{(i)}}(X) + \delta.
\end{align*}
Then,
\begin{itemize}
	\item If $\gamma^*\in\cU^{(i)}$, then $Y(\tilde X)$ is $\eta$-feasible. 
	\item If $Y(\tilde X)$ is $\eta$-feasible, then $Y(\tilde X)$ is $\epsilon$-optimal and $\epsilon$-feasible.
\end{itemize}
\end{lemma}
\begin{proof}
Suppose $\gamma^*\in\cU^{(i)}$ and define $\Delta \coloneqq \tilde X - X^*$. Strong convexity and \cref{thm:strong_convex_reform} imply
$\frac{\hat\mu}{2}\norm{\Delta}_F^2 \leq \delta$.
Next, recalling that $q_i(X^*) = 0$ for all $i\in[m]$, we deduce
\begin{align*}
\left(\sum_{i=1}^m \left(\ip{M_i, Y(\tilde X)} + d_i\right)^2\right)^{1/2} &= \left(\sum_{i=1}^m q_i(\tilde X)^2\right)^{1/2}\\
&= \max_{\norm{\gamma}_2 = 1} \sum_{i=1}^m \gamma_i\left(\frac{\tr\left(\Delta^\intercal A_i \Delta\right)}{2} + \ip{A_iX^* + B_i, \Delta}\right)\\
&\leq \hat\rho\left(\frac{\delta}{\hat\mu}\right) + \sqrt{8}\hat\rho\hat R_p\sqrt{\frac{\delta}{\hat\mu}}\\
&\leq \frac{4\hat\rho\hat R_p}{\sqrt{\hat\mu}} \sqrt{\delta} = \eta.
\end{align*}
Here, the first inequality follows from $\norm{\Delta}_F^2 \leq \frac{2\delta}{\hat\mu}$ and \cref{as:primal_dual_regularity}, and the last inequality follows as $\delta = \frac{\hat\mu\epsilon^2}{\left(9\hat\rho\hat R_d\hat R_p\right)^2} \leq \hat\mu\hat R_p^2$ since $0<\epsilon\leq 9\hat\rho\hat R_d\hat R_p^2$.

Now, suppose $Y(\tilde X)$ is $\eta$-feasible. Note that $\eta\leq \epsilon$ (as $\hat R_d\geq 1$) and thus $Y(\tilde X)$ is immediately $\epsilon$-feasible. Let $\tilde\gamma \in\argmax_{\gamma\in \cU^{(i)}} q(\gamma, \tilde X)$ so that $Q_{\cU^{(i)}}(\tilde X) = q(\tilde \gamma, \tilde X)$. Then,
\begin{align*}
\ip{M_\obj, Y(\tilde X)} = q_\obj(\tilde X) &= q(\tilde \gamma, \tilde X) - \sum_{i=1}^m \tilde\gamma_i q_i(\tilde X)\\
&\leq Q_{\cU^{(i)}}(\tilde X) + \norm{\tilde \gamma}_2\eta\\
&\leq \left(\min_{X\in\R^{(n-k)\times k}}Q_{\cU^{(i)}}(X) + \delta \right)+ 2\hat R_d \eta\\
&\leq \Opt_{\eqref{eq:primal_dual_SDP}} + \left(\delta + 2 \hat R_d \eta\right),
\end{align*}
where the first inequality follows from the $\eta$-feasibility of $Y(\tilde X)$, the second inequality from the premise of the lemma on $\tilde X$ and the fact that $\norm{\tilde\gamma}_2\leq \norm{\tilde\gamma-\gamma^{(i)}}_2+\norm{\gamma^{(i)}}_2\leq2\hat R_d$ (this holds because $\tilde\gamma\in\cU^{(i)}$, $\cU^{(i)}$ is the $\ell_2$-ball of radius $r^{(i)}$ centered at $\gamma^{(i)}$,  and by definition of $r^{(i)}$ we have $r^{(i)}\leq 2\hat R_d -\norm{\gamma^{(i)}}_2$). 
We may then use the definitions of $\delta$ and $\eta$ to bound
\begin{align*}
\delta + 2 \hat R_d \eta &= \frac{\hat\mu\epsilon^2}{\left(9\hat\rho\hat R_d\hat R_p\right)^2} + \frac{8\epsilon}{9} \leq \frac{\hat\mu\epsilon}{9\hat\rho\hat R_d} + \frac{8\epsilon}{9}\leq \epsilon.
\end{align*}
Here, the first inequality follows from the upper bound on $\epsilon$ and the second one follows from $\frac{\hat\mu}{\hat R_d}\leq\hat\rho$ (\cref{as:primal_dual_regularity}). This then shows that $Y(\tilde X)$ is $\epsilon$-optimal.
\end{proof}

We are now ready to present our full algorithm for computing approximate solutions to \eqref{eq:primal_dual_SDP}. CertSDP (\cref{alg:exact_sdp_qmmp}) assumes access to a sequence $\gamma^{(i)}\to \gamma^*$ and applies a guess-and-double scheme to guess when $\norm{\gamma^{(i)} - \gamma^*}_2$ is sufficiently small. It then applies \cref{alg:cautious-agd} to either compute an $\epsilon$-optimal $\epsilon$-feasible solution $Y(\tilde X)$ or to declare that $\gamma^*\notin \cU^{(i)}$.
\begin{algorithm}
\caption{CertSDP}
\label{alg:exact_sdp_qmmp}
Given a rank-$k$ exact QMP-like SDP satisfying \cref{as:primal_dual_regularity}, a sequence $\gamma^{(1)},\gamma^{(2)},\dots \to \gamma^*$, and $0<\epsilon\leq9\hat\rho\hat R_d\hat R_p^2$
	{\small\begin{enumerate}[topsep=0pt,itemsep=0pt,parsep=0pt]
		\item Set $\delta$ and $\eta$ as in \cref{lem:sdp_to_qmmp_error_parameters}
		\item For each $i = 2^0,\,2^1,\,2^2,\dots$
		\begin{itemize}[topsep=0pt,itemsep=0pt,parsep=0pt]
			\item If $r^{(i)}>0$
			\begin{enumerate}[topsep=0pt,itemsep=0pt,parsep=0pt]
				\item Let $\cU^{(i)}\coloneqq \mathbb{B}\left(\gamma^{(i)}, r^{(i)}\right)$ and compute $\tilde X$ satisfying
				\begin{align*}
				Q_{\cU^{(i)}}(\tilde X) \leq \min_{X\in\R^{(n-k)\times k}} Q_{\cU^{(i)}}(X) + \delta
				\end{align*}
				using \cref{alg:cautious-agd}
				\item If $Y(\tilde X)$ is $\eta$-feasible, output $Y(\tilde X)$
			\end{enumerate}
		\end{itemize}
	\end{enumerate}}
\end{algorithm}

The next theorem gives rigorous guarantees on CertSDP and follows from \cref{lem:parameters_for_qmmp,lem:gamma_necessary_accuracy,lem:sdp_to_qmmp_error_parameters,thm:overall_qmmp_iterations}.
\begin{theorem}
\label{thm:overall_exact_sdp_qmmp}
Suppose \eqref{eq:primal_dual_SDP} is a rank-$k$ exact QMP-like SDP satisfying \cref{as:primal_dual_regularity}, $\gamma^{(1)},\,\gamma^{(2)},\,\dots\to \gamma^*$ and $0<\epsilon\leq 9\hat\rho\hat R_d\hat R_p^2$.
Let $T$ be such that $\norm{\gamma^{(t)}- \gamma^*}_2 \leq \frac{\hat\mu}{4\hat\rho}$ for all $t\geq T$. Then, CertSDP (\cref{alg:exact_sdp_qmmp}) accesses at most $2T$ iterates of the sequence $\gamma^{(i)}$ and outputs an $\epsilon$-optimal and $\epsilon$-feasible solution in
\begin{gather*}
O\left(\sqrt{\hat\kappa} \log\left(\frac{\hat\kappa\hat\rho\hat R_p\hat R_d}{\epsilon}\right) \cdot\log\left(T\right)\right)\text{ prox-map calls, and}\\
O\left(\frac{\hat\kappa^{7/4}\hat\rho\hat R_p^2\hat R_d}{\epsilon} \cdot \log(T)\right)\text{iterations within all prox-map calls.}
\end{gather*}
\end{theorem}

We present \cref{thm:overall_exact_sdp_qmmp} in its current form to emphasize that it can be combined with any existing algorithm that produces an estimate for $\gamma^*$. For example, we could combine CertSDP with the subgradient method employed in \cite{ding2019optimal} for solving the penalized version of the dual SDP. Alternatively, CertSDP could also be combined with augmented-Lagrangian-based algorithms such as~\cite{yurtsever2021scalable,burer2003nonlinear} or as a postprocessing step on a general SDP solver that produces a sufficiently accurate dual solution.

The following corollary combines CertSDP with the algorithm suggested in \cite[Section 6.2.2]{ding2019optimal} for approximating $\gamma^*$.
\begin{corollary}
    Suppose \eqref{eq:primal_dual_SDP} is a rank-$k$ exact QMP-like SDP satisfying \cref{as:primal_dual_regularity} and $0<\epsilon\leq 9\hat\rho\hat R_d\hat R_p^2$. Furthermore, suppose $\gamma^*$ is unique.
    Then, CertSDP (\cref{alg:exact_sdp_qmmp}) combined with \cite[Section 6.2.2]{ding2019optimal} outputs an $\epsilon$-optimal and $\epsilon$-feasible solution in
    \begin{gather*}
    O\left(\frac{\hat\rho^2}{\hat\mu^2}\cdot \frac{\hat L\hat R_d+ \Opt_{\eqref{eq:primal_dual_SDP}}}{\lambda_{\min>0}(Y^*)\sigma_{\min}(\cD)^2}\right)\text{ dual iterations},\\
    \tilde O\left(\sqrt{\hat\kappa} \log\left(\frac{\hat\kappa\hat\rho\hat R_p\hat R_d}{\epsilon}\right)\right)\text{ prox-map calls, and}\\
    \tilde O\left(\frac{\hat\kappa^{7/4}\hat\rho\hat R_p^2\hat R_d}{\epsilon}\right)\text{iterations within all prox-map calls.}
    \end{gather*}
    Here, $\lambda_{\min>0}(Y^*)$ is the magnitude of the smallest nonzero eigenvalue of $Y^*$ and $\sigma_{\min}(\cD)$ is the minimum singular value of an operator $\cD$ (see \cite[Lemma 3.1]{ding2019optimal}) and is guaranteed to be positive.
\end{corollary} 

\section{Numerical experiments}
\label{sec:numerical}

We investigate the numerical performance of our new FOM, CertSDP, on rank-$k$ exact QMP-like SDPs that are both large and sparse.
Similar experiments on SDPs inspired by the phase retrieval problem are summarized in \cref{subsec:additional} and presented in detail in \cref{sec:additional_exp}.

In this section, we consider random instances of distance-minimization QMPs and their primal and dual SDP relaxations of the form
\begin{align}
\label{eq:dist_min_qmp}
&\inf_{X\in\R^{(n-k)\times k}}\set{\frac{\norm{X}_F^2}{2}:\, \begin{array}
	{l}
	\tr\left(\frac{X^\intercal A_i X}{2}\right) + \ip{B_i, X} + c_i = 0,\,\forall i\in[m]
\end{array}}\\
&\qquad\geq \inf_{Y\in\S^n}\set{\ip{\begin{pmatrix}
	I_{n-k}/2 &\\& 0_k
\end{pmatrix}, Y}:\, \begin{array}
	{l}
	\ip{\begin{pmatrix}
	A_i/2 & B_i/2\\ B_i^\intercal /2 & \tfrac{c_i}{k}I_k 	
	\end{pmatrix}, Y} = 0,\,\forall i\in[m]\\
	Y = \begin{pmatrix}
		* & *\\
		* & I_k
	\end{pmatrix}\succeq 0
\end{array}}\nonumber\\
&\qquad\geq \sup_{\gamma\in\R^m,\, T\in\S^k}\set{\tr(T):\, \begin{pmatrix}
	A(\gamma)/2 & B(\gamma)/2\\
	B(\gamma)^\intercal / 2 & \frac{c(\gamma)}{k}I_k - T
\end{pmatrix}\succeq 0}.\nonumber
\end{align}
In our instance generation procedure, we ensure that equality holds throughout this chain of inequalities.

We will compare the performance of CertSDP on instances of \eqref{eq:dist_min_qmp} to that of several first-order methods from the literature: the complementary slackness SDP algorithm (CSSDP)~\cite{ding2021optimal}, 
SketchyCGAL~\cite{yurtsever2021scalable}, ProxSDP~\cite{souto2020exploiting}, and the splitting cone solver (SCS)~\cite{odonoghue2016conic}.
In addition to these convex-optimization-based algorithms, we also compare our results with the nonconvex Burer-Monteiro method~\cite{burer2003nonlinear}.
We discuss these algorithms and relevant implementation details in \cref{subsec:implementation_algs} and the instance generation procedure in \cref{subsec:random_instances} before presenting the numerical results in \cref{subsec:numerical_results}.

\cref{subsec:additional} includes a summary of additional experiments inspired by the phase retrieval problem~\cite{candes2015phase}. We defer the details of these experiments and their implementation details to \cref{sec:additional_exp}.

All algorithms and experiments are implemented in Julia and run on a machine with an AMD Opteron 4184 processor with 12 CPUs and 70GB of RAM. Our code is publicly available at:
\begin{center}
\url{https://github.com/alexlihengwang/CertSDP}
\end{center}

\begin{remark}
    \label{rem:virtual_size_memory}
    For all of our experiments, we track the memory consumption of each algorithm by monitoring the virtual memory size (\texttt{vsz}) of the process throughout the run of the algorithm and report the difference between the maximum value and the starting value.
    This is the same measurement that is performed in \cite{yurtsever2021scalable}.
    We caution that this number should only be treated as a \emph{very rough} estimate of the storage requirements. Indeed, virtual memory need not be allocated at all for small enough programs (so that some algorithms register as using no memory at all for small enough values of $n-k$) and furthermore, when it is allocated, it is not always fully used. Experimentally, we found that on our machine, storage of up to $\approx 1.0$ MB was often measured as not using any memory at all.
    We report such measurements as 0.0 MB in our tables (\cref{table:1e3,table:1e4,table:1e5}) and as 1.0 MB in our log-scale plot~\cref{fig:memory_usage}.
    \ifjelse{\qed}{}
\end{remark}

\subsection{Implementation details}
\label{subsec:implementation_algs}

\paragraph{CertSDP.}
We implemented CertSDP (\cref{alg:exact_sdp_qmmp}) as presented in this paper except a few modifications. In addition to simplifying the overall algorithm, these modifications enable CertSDP to be run \emph{without} knowledge of the parameters $\hat\mu$ and $\hat L$. While the convergence guarantees of \cref{thm:overall_exact_sdp_qmmp} may no longer hold, we find empirically that CertSDP continues to perform very effectively with these modifications.
\begin{itemize}
	\item We instantiate CertSDP with Accelegrad~\cite{levy2018online} as the iterative method for producing iterates $\gamma^{(i)}$. As in \cite{ding2021optimal}, we apply Accelegrad to the penalized dual problem
	\begin{align*}
	\max_{\gamma\in\R^m,\, T\in\S^k}\tr\left(T\right) + \text{penalty}\cdot\min\left(0,\,  \lambda_{\min}\begin{pmatrix}
		A(\gamma)/2 & B(\gamma)/2\\
		B(\gamma)^\intercal/2 & \frac{c(\gamma)}{k}I_k - T
	\end{pmatrix}\right)
	\end{align*} 
	for some large value for the penalty parameter. It can be shown that the optimal value and optimizers of this penalized dual problem coincide with that of the dual SDP whenever the penalty parameter is larger than $\tr(Y^*)$; see \cite{ding2021optimal}. In our experiments, we set the penalty parameter to be $20\cdot\tr(Y^*)$.
	\item In practice, it is extremely cheap to solve \eqref{eq:strong_convex_reform} even to high accuracy. Thus, we replace the guess-and-double scheme in \cref{alg:exact_sdp_qmmp} with a linear schedule after
	a fixed number of iterations. Specifically, we solve \eqref{eq:strong_convex_reform} on iterations 1, 2, 4, 8, \dots, 256. After that, we solve \eqref{eq:strong_convex_reform} once every 256 iterations.
    Additionally, we replace the excessive gap technique used in \cref{thm:overall_qmmp_iterations} with accelerated gradient descent (see \cref{rem:prox_map_agd}).
	\item We set
	\begin{align*}
	 r^{(i)} = \frac{1}{\hat\rho}\cdot\frac{\lambda_{\max}(A(\gamma^{(i)}))\lambda_{\min}(A(\gamma^{(i)}))}{2\lambda_{\max}(A(\gamma^{(i)})) + \lambda_{\min}(A(\gamma^{(i)}))}
	\end{align*}
	if $A(\gamma^{(i)})\succ 0$, and $r^{(i)} = 0$ else. Equivalently, $\cU^{(i)}\coloneqq \mathbb{B}(\gamma^{(i)},r^{(i)})$ is the largest ball centered at $\gamma^{(i)}$ for which the condition number of $A(\gamma)$ for any $\gamma\in\mathbb{B}(\gamma^{(i)},r^{(i)})$ is guaranteed to be at most twice the condition number of $A(\gamma^{(i)})$. To see this, 
	recall that $\hat\rho$ satisfies $\hat\rho \ge \max_{\gamma'\in\bS^{m-1}}\norm{\sum_{i=1}^m \gamma'_iA_i}_2$.
	Then, for any $\gamma\in \mathbb{B}(\gamma^{(i)}, r^{(i)})$, we have
    \begin{align*}
        \lambda_{\min}(A(\gamma)) &\geq \lambda_{\min}(A(\gamma^{(i)})) - \hat\rho \cdot r^{(i)}\\
        &= \lambda_{\min}(A(\gamma^{(i)})) \frac{\lambda_{\min}(A(\gamma^{(i)})) + \lambda_{\max}(A(\gamma^{(i)}))}{\lambda_{\min}(A(\gamma^{(i)})) + 2 \lambda_{\max}(A(\gamma^{(i)}))} \eqqcolon \hat\mu^{(i)}.
    \end{align*}
    Similarly, one can verify that
    \begin{align*}
        \lambda_{\max}(A(\gamma)) \leq 2\lambda_{\max}(A(\gamma^{(i)})) \frac{\lambda_{\min}(A(\gamma^{(i)})) + \lambda_{\max}(A(\gamma^{(i)}))}{\lambda_{\min}(A(\gamma^{(i)})) + 2 \lambda_{\max}(A(\gamma^{(i)}))}\eqqcolon \hat L^{(i)}.
    \end{align*}
    Thus, $\hat\mu^{(i)}$ and $\hat L^{(i)}$ serve as heuristic choices of $\hat\mu$ and $\hat L$. Furthermore, $\frac{\hat L^{(i)}}{\hat \mu^{(i)}} \leq 2 \frac{\lambda_{\max}(A(\gamma^{(i)}))}{\lambda_{\min}(A(\gamma^{(i)}))}$.

	Note that it still holds that $A(\gamma)\succ 0$ for all $\gamma\in\cU^{(i)}$ (as long as $r^{(i)}$ is positive) and that $\gamma^*\in\cU^{(i)}$ for all $\gamma^{(i)}$ close enough to $\gamma^*$.
	\item In CautiousAGD (\cref{alg:cautious-agd}), we terminate early if $\max_{i\in[m]}\abs{q_i(X_t)}$ does not decrease to zero geometrically. Indeed, this can only happen if $\gamma^*\notin\cU^{(i)}$.
	\item \cref{thm:oracle_cost} gives an \textit{a priori} guarantee on the number of inner iterations required for solving each prox-map. Instead of using this number of iterations, in our code, we will monitor the saddle point gap, i.e., the second term in \eqref{eq:saddle_point_error}, and break as soon as the saddle point gap is small enough.
	\item We warm-start the iterate $X$ in CautiousAGD using the last iterate of the previous run of CautiousAGD and warm-start $\gamma$ in the prox-map computation using the last iterate of the previous run of the prox-map computation.
	\item Unless the time limit is met first, the overall algorithm is terminated once CautiousAGD produces a $\left(10^{-13}\right)$-optimal solution of \eqref{eq:strong_convex_reform} that satisfies $\max_{i\in[m]}\abs{q_i(X_t)}\leq 10^{-13}$.
\end{itemize}

\paragraph{CSSDP.} The complementary slackness SDP algorithm (CSSDP)~\cite{ding2021optimal} similarly constructs a sequence of iterates $\gamma^{(i)}\to\gamma^*$ and occasionally solves a compressed $k$-dimensional SDP~\cite[MinFeasSDP]{ding2021optimal} in the vector space corresponding to the $k$-many minimum eigenvalues of the slack matrix $M(\gamma^{(i)})$.
As in our implementation of CertSDP, we instantiate CSSDP with Accelegrad~\cite{levy2018online} as the iterative method for producing iterates $\gamma^{(i)}$ and solve the compressed SDP at iterations $1$, $2$, $4$, $\dots$, $256$ and then once every $256$ iterations thereafter. The compressed SDPs are solved using SCS solver with all error parameters set to $10^{-13}$.
Since CSSDP needs to solve the compressed SDP frequently, we make sure to instantiate the optimization problem just once in order to amortize the cost of allocating the $k\by k$ symmetric matrix variable.

\paragraph{SketchyCGAL.} \citet{yurtsever2021scalable} observe that one may track any \emph{linear image} of the primal matrix iterates (as opposed to the matrix iterate itself) in the CGAL~\cite{yurtsever2019conditional} algorithm. Combining this observation with the Nystr\"om sketch gives SketchyCGAL. For \eqref{eq:dist_min_qmp}, we implement a variant of this idea, where we replace the Nystr\"om sketch with the linear map sending a matrix in $\S^n$ to its top-right $(n-k)\times k$ submatrix.

\paragraph{ProxSDP and SCS.} ProxSDP~\cite{souto2020exploiting} and the splitting cone solver (SCS)~\cite{odonoghue2016conic} are FOMs that can be used to tackle large-scale SDPs.
ProxSDP combines the primal-dual hybrid gradient method with an approximate projection operation that allows it to replace a full eigendecomposition with a partial one whenever the rank of the true SDP solution is small.
SCS employs an FOM to tackle the homogeneous self-dual embedding but does not explicitly take advantage of possible low rank solutions.

In our experiments, we pass the SDP relaxations of our QMPs to the corresponding Julia interfaces \
\href{https://github.com/mariohsouto/ProxSDP.jl}{\texttt{ProxSDP.jl}}
and
\href{https://github.com/jump-dev/SCS.jl}{\texttt{SCS.jl}}
with all error parameters set to $10^{-13}$.
In contrast to CertSDP, CSSDP, and SketchyCGAL, which achieve storage optimality, ProxSDP and SCS both store matrix iterates and thus require substantially more memory.

\paragraph{Burer-Monteiro.}
We implement the Burer-Monteiro method as outlined originally in~\cite{burer2003nonlinear}. This method is naturally storage-optimal but may fail (at least theoretically) to converge even on 1-exact QMP-like SDPs (see \cref{subsec:bm_comparison}).

\subsection{Random instance generation}
\label{subsec:random_instances}
We generate random sparse instances of distance-minimization QMPs~\eqref{eq:dist_min_qmp} as follows: Let $(n,k,m,\mu^*, \texttt{nnz})$ be input parameters. Here, $(n,k,m)$ control the size of \eqref{eq:dist_min_qmp}, $\mu^*$ is the desired value of $\lambda_{\min}(A(\gamma^*))$ and \texttt{nnz} approximately controls the number of nonzero entries in each $A_1,\dots,A_m$.
\begin{itemize}
	\item Let $A_1,\dots, A_m\in\S^{n-k}$ be sparse matrices each with $\approx \texttt{nnz}$ nonzero entries that are i.i.d.\ normal. We scale $A_1,\dots,A_m$ such that $\norm{A_i}_2  = 1$ for all $i\in[m]$.
	\item Let $B_1,\dots,B_m\in\R^{(n-k)\times k}$ be matrices where all entries are i.i.d.\ normal. We scale $B_1,\dots,B_m$ such that $\norm{B_i}_F = 1$ for all $i\in[m]$.
	\item Pick a direction $\hat\gamma$ uniformly from the surface of the sphere $\bS^{m-1}$, then set $\gamma^* \coloneqq r\hat\gamma$ where $r>0$ solves $\lambda_{\min}\left(A(\gamma^*)\right) = 1 + r \lambda_{\min}\left(\sum_{i=1}^m \hat\gamma_i A_i\right) = \mu^*$.
	Let $X^*\coloneqq - A(\gamma^*)^{-1}B(\gamma^*)$.
	\item Finally, for each $i\in[m]$, set $c_i$ such that
	$\tr\left(\frac{X^{*\intercal} A_i X^*}{2}\right) + \ip{B_i, X^*} + c_i = 0$.
\end{itemize}
Exactness is guaranteed to hold throughout \eqref{eq:dist_min_qmp} as $(\gamma^*,\,T^*)$, where
\begin{align*}
T^* \coloneqq \frac{c(\gamma^*)}{k}I_k - \frac{B(\gamma^*)^\intercal A(\gamma^*)^{-1}B(\gamma^*)}{2},
\end{align*}
achieves the value $\frac{\norm{X^*}_F^2}{2}$ in the third line of \eqref{eq:dist_min_qmp} (see \cref{lem:correctness_of_random_generation} in \cref{sec:deferred}).

\subsection{Numerical results}
\label{subsec:numerical_results}

To investigate the scalability of CertSDP in terms of $n$, we fix $k = 10$, $m = 10$, $\mu^* = 0.1$ and $\texttt{nnz} = n$.
Note that in this regime, the $A_i$ matrices are each individually very sparse with approximately one nonzero entry per row or column.
We then vary $n$ such that the height of the matrix variable $X\in\R^{(n-k)\times k}$, i.e., $n-k$, takes the values $10^3,\, 10^4,\, 10^5$. For each value of $n-k$, we generate 10 random instances of \eqref{eq:dist_min_qmp} according to \cref{subsec:random_instances} and measure the time, error, and memory consumption of the tested algorithms.

We ran each algorithm with time limits of $3\times 10^3$, $10^4$, and $5\times 10^4$ seconds for $n-k = 10^3$, $10^4$, $ 10^5$ respectively.
SCS is not tested for $n-k = 10^4$ as it was unable to complete a single iteration within the time limits and utilized over 70GB of memory.
Similarly, ProxSDP and SCS were not tested for $n-k = 10^5$ as both came to complete failures due to excessive memory allocation.

Detailed numerical results are reported in \cref{table:1e3,table:1e4,table:1e5} for $n-k = 10^3$, $10^4$, and $10^5$, respectively.
\cref{fig:memory_usage} shows the average memory usage of the algorithms.
We compare the convergence behavior of CertSDP with that of CSSDP and SketchyCGAL on a single instance of each size in \cref{fig:convergence_behavior}.
The plots on the left in \cref{fig:convergence_behavior} show the primal squared distance $\norm{X - X^*}_F^2$
and the dual suboptimality
\begin{align*}
\Opt_{\eqref{eq:dist_min_qmp}} - \left(\tr\left(T\right) + \text{penalty}\cdot\min\left(0,\,  \lambda_{\min}\begin{pmatrix}
		A(\gamma)/2 & B(\gamma)/2\\
		B(\gamma)^\intercal/2 & \frac{c(\gamma)}{k}I_k - T
	\end{pmatrix}\right)\right)
\end{align*}
for the iterates produced by CertSDP, CSSDP, and SketchyCGAL as a function of time.
The plots on the right of \cref{fig:convergence_behavior} show the primal squared distance for the iterates produced by CertSDP within the final call to CautiousAGD.

\begin{table}[htbp]
\centering
\begin{tabular}{lllllll}
\toprule 
Algorithm & time (s) & std. & $\norm{X - X^*}_F^2$ & std. & memory (MB) & std. \\
\midrule 
CertSDP & \num{1.3e+03} & \num{7.6e+02} & \num{1.9e-22} & \num{4.2e-23} & \num{0.0e+00} & \num{0.0e+00} \\
CSSDP & \num{3.0e+03} & \num{5.8e-01} & \num{7.3e-02} & \num{3.4e-02} & \num{0.0e+00} & \num{0.0e+00} \\
SketchyCGAL & \num{3.0e+03} & \num{8.5e+00} & \num{1.1e+00} & \num{6.6e-01} & \num{1.0e+01} & \num{1.0e+01} \\
ProxSDP & \num{2.1e+02} & \num{1.1e+01} & \num{1.2e-19} & \num{3.2e-19} & \num{4.8e+01} & \num{1.9e+01} \\
SCS & \num{3.1e+03} & \num{2.5e+01} & \num{5.1e-05} & \num{9.5e-05} & \num{5.3e+02} & \num{4.3e+01} \\
BM & \num{4.2e+01} & \num{2.5e+00} & \num{1.1e-14} & \num{6.8e-15} & \num{2.3e+01} & \num{1.2e+01}\\
\bottomrule 
\end{tabular}
 \caption{Experimental results for $(n-k) = 10^3$ (10 instances) with time limit $3\times 10^3$ seconds.}
\label{table:1e3}
\end{table}

\begin{table}[htbp]
\centering
\begin{tabular}{lllllll}
\toprule 
Algorithm & time (s) & std. & $\norm{X - X^*}_F^2$ & std. & memory (MB) & std. \\
\midrule 
CertSDP & \num{4.5e+03} & \num{7.0e+02} & \num{1.9e-22} & \num{5.2e-23} & \num{8.5e+00} & \num{1.2e+01} \\
CSSDP & \num{1.0e+04} & \num{6.6e-01} & \num{2.7e+00} & \num{9.4e-01} & \num{6.2e+00} & \num{1.5e+01} \\
SketchyCGAL & \num{9.7e+03} & \num{1.8e+02} & \num{4.0e+00} & \num{1.4e+00} & \num{2.7e+01} & \num{2.2e+01} \\
ProxSDP & \num{1.2e+04} & \num{1.1e+02} & \num{2.9e+00} & \num{9.9e-01} & \num{1.9e+04} & \num{1.2e+02} \\
BM & \num{4.9e+02} & \num{2.8e+01} & \num{1.1e-14} & \num{1.0e-14} & \num{5.2e+01} & \num{6.9e+01}\\
\bottomrule 
\end{tabular}
 \caption{Experimental results for $(n-k) = 10^4$ (10 instances) with time limit $10^4$ seconds. SCS was unable to complete a single iteration within the time limit and utilized over 70GB of memory.}
\label{table:1e4}
\end{table}

\begin{table}[htbp]
\centering
\begin{tabular}{lllllll}
\toprule 
Algorithm & time (s) & std. & $\norm{X - X^*}_F^2$ & std. & memory (MB) & std. \\
\midrule 
CertSDP & \num{5.0e+04} & \num{6.2e+02} & \num{2.5e-02} & \num{6.5e-02} & \num{2.3e+02} & \num{2.0e+02} \\
CSSDP & \num{5.0e+04} & \num{4.7e+00} & \num{2.8e+00} & \num{5.1e-01} & \num{2.0e+02} & \num{2.5e+02} \\
SketchyCGAL & \num{4.7e+04} & \num{3.3e+03} & \num{4.0e+00} & \num{2.1e+00} & \num{3.7e+02} & \num{2.0e+02} \\
BM & \num{6.5e+03} & \num{3.0e+02} & \num{7.1e-15} & \num{3.5e-15} & \num{1.2e+03} & \num{3.4e+02} \\
\bottomrule 
\end{tabular}
 \caption{Experimental results for $(n-k) = 10^5$ (10 instances) with time limit $5\times 10^4$ seconds. SCS and ProxSDP are not tested as they both come to complete failure due to memory allocation. $^\dag$CSSDP failed due to numerical issues within the eigenvalue subroutine on three instances and SketchyCGAL failed due to numerical issues within the eigenvalue subroutine on one instance.}
\label{table:1e5}
\end{table}

\begin{figure}
	\centering
	\includegraphics[scale=\figurescale]{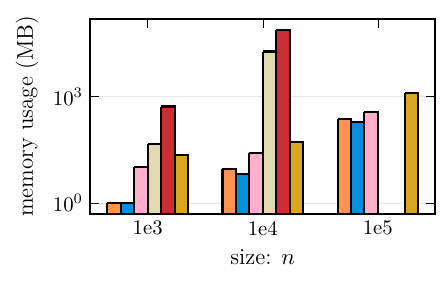}

	\includegraphics[scale=\figurescale]{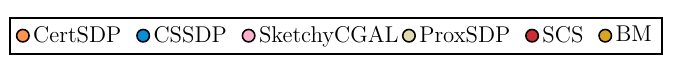}

	\caption{Memory usage of different algorithms as a function of the size $n-k$. In this chart, we plot $0.0$ MB at $1.0$ MB (see \cref{rem:virtual_size_memory} for a discussion on measuring memory usage).	}
	\label{fig:memory_usage}
\end{figure}

\begin{figure}
	\centering

	\includegraphics[scale=\figurescale]{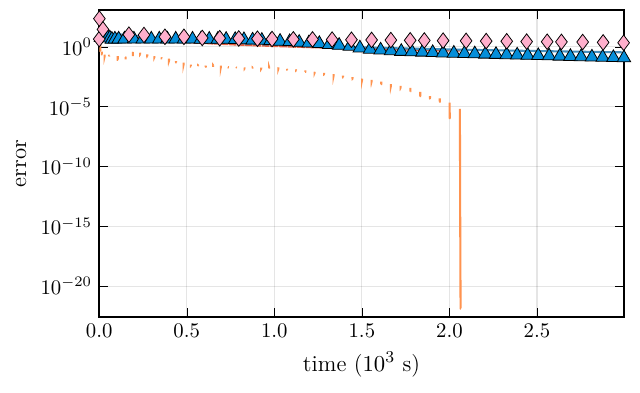}
	\includegraphics[scale=\figurescale]{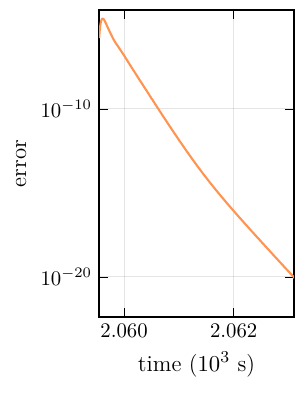}

	\vspace{-0.5em}

	\includegraphics[scale=\figurescale]{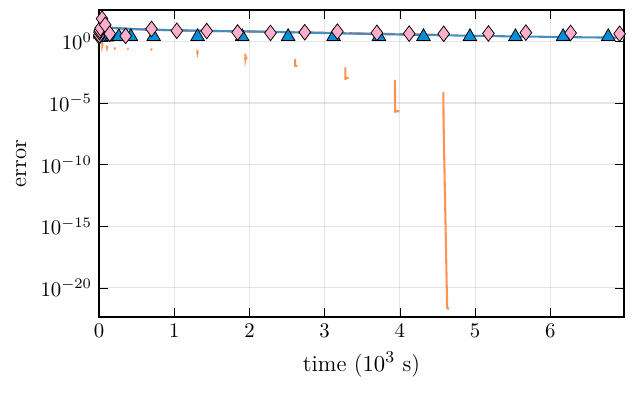}
	\includegraphics[scale=\figurescale]{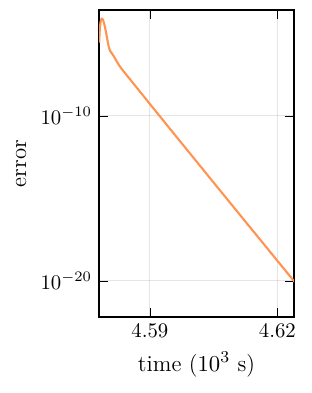}

	\vspace{-0.5em}

	\includegraphics[scale=\figurescale]{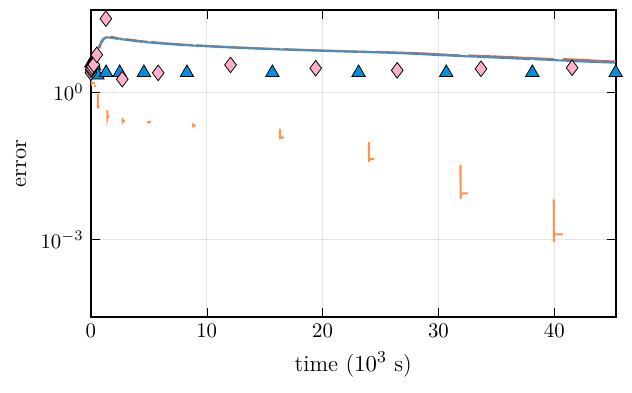}
	\includegraphics[scale=\figurescale]{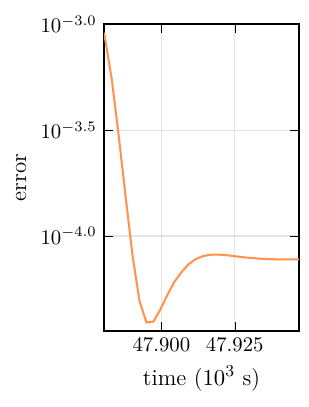}

	\vspace{-0.5em}

	\includegraphics[scale=\figurescale]{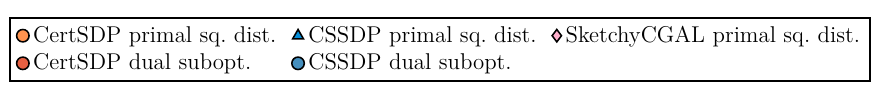}

	\caption{Comparison of convergence behavior between CertSDP (\cref{alg:exact_sdp_qmmp}), CSSDP, and SketchyCGAL.
	The first, second, and third rows show experiments with $n-k=10^3$, $10^4$, and $10^5$ respectively.
	The right subplots give zoomed-in views of the primal squared distance in CertSDP on the final call to \cref{alg:cautious-agd}.}
	\label{fig:convergence_behavior}
\end{figure}

We make a few observations:
\begin{itemize}
	\item For $n -k =10^3$ (see \cref{table:1e3}), both CertSDP and ProxSDP were able to achieve high accuracy within the time limit, while CSSDP, SketchyCGAL, and SCS could not.
	ProxSDP was faster than CertSDP while CertSDP used significantly less memory.
	\item For $n -k =10^4$ (see \cref{table:1e4}), CertSDP was the only convex-optimization-based algorithm that was able to achieve high accuracy within the time limit.
	The measured memory usage of CertSDP and CSSDP both had high variance, however it is clear that these algorithms use much less memory than SketchyCGAL, ProxSDP and SCS. As previously mentioned, SCS used over 70GB of memory at this size.
	\item For $n -k =10^5$ (see \cref{table:1e5}), CertSDP, CSSDP and SketchyCGAL were the only convex-optimization-based algorithms that could be run without memory allocation errors. While neither algorithm was able to achieve the desired accuracy within the time limit, CertSDP (average primal squared distance of $2.5\times 10^{-2}$) significantly outperformed CSSDP and SketchyCGAL (average primal squared distances of $2.8$ and $4.0$, respectively). 
	\item The dual suboptimality for CertSDP and CSSDP behave identically. This is expected as we employ Accelegrad to generate both sequences. 
	\item The primal squared distance and the dual suboptimality for CSSDP track quite closely. This is expected from \cite[Theorem 4.1, Table 3]{ding2021optimal}, as the primal squared distance is bounded by a constant factor of the dual suboptimality for CSSDP.
	\item The convergence behavior of CautiousAGD depends on whether $\cU^{(i)}$ in CertSDP is a certificate of strict complementarity.

	When $\cU^{(i)}$ is \emph{not} a certificate of strict complementarity, CautiousAGD behaves as in the bottom-right plot of \cref{fig:convergence_behavior}: It briefly converges linearly before plateauing. This makes sense as the iterates in CautiousAGD should converge linearly to $\argmin_X Q_{\cU^{(i)}}(X) \neq X^*$.
	
	When $\cU^{(i)}$ is a certificate of strict complementarity, the iterates of CautiousAGD converge linearly to $X^*$ (see the top-right and middle-right plots of \cref{fig:convergence_behavior}).

    \item The Burer-Monteiro method is an order of magnitude faster than the other approaches, and it consistently delivers high quality solutions comparable to CertSDP in small and medium scale instances, and much better quality in the large-scale instances.
\end{itemize}

\subsection{Additional experiments}
\label{subsec:additional}

\cref{sec:additional_exp} contains additional experiments on the PhaseLift SDP~\cite{candes2015phase} for the Gaussian model of phase retrieval. In the Gaussian model of phase retrieval, the goal is to recover an unknown $x^*\in\R^n$ from $m = O(n)$ observations $\beta_i = (g_i^\intercal x^*)^2$ where $g_i$ are Gaussian with an appropriate normalization. We generate instances of this problem with one highly correlated observation and compare the performance of the different algorithms.
We test 10 instances at each size $n = 30,\,100,\,300$ and $m = 5n$. The main bottleneck for scaling the size of these experiments is in storing the instance itself---each instance has size $\approx 5n^2$. Nonetheless, we expect that the behavior we see in these experiments is indicative of what to expect in the true model of phase retrieval where the observation vectors $g_1,\dots,g_m$ are only stored implicitly.
Detailed results for these experiments are shown in \cref{table:phret_3e1,table:phret_1e2,table:phret_3e2}. Comparisons of memory usage and convergence behavior are illustrated in \cref{fig:phret_memory_usage,fig:phret_convergence_behavior}.
Similar to the experiments in this section, CertSDP is able to achieve high accuracy within the given time limits and outperforms CSSDP and SketchyCGAL in terms of primal accuracy. On the other hand,
the ``crossover'' point, where CertSDP outperforms CSSDP and SketchyCGAL, seems to occur later after CSSDP and SketchyCGAL already produce iterates with moderate accuracy.
Additionally, CertSDP seems to outperform SketchyCGAL on our phase retrieval instances despite their behaviors being similar in the sparse QMP setting.
ProxSDP and SCS are able to achieve high accuracy on some instance sizes, but use significantly more memory than CertSDP.
Again,
the Burer-Monteiro method performs quite well and outperforms all convex-optimization-based algorithms. \section{Conclusions, limitations and future directions}
\label{sec:limitations}

In this paper, we presented a fast storage-optimal FOM to solve SDPs satisfying strict complementarity and exactness guarantees. While the numerical results are promising, there are limitations. We discuss these here and offer some thoughts on how to address them and future research directions.

\paragraph{Finding $\gamma^*$.}
In \cref{thm:overall_exact_sdp_qmmp}, it is assumed that CertSDP has access to iterates from a sequence $\gamma^{(1)},\gamma^{(2)},\dots$ that approach to a dual solution $\gamma^*$ satisfying $A(\gamma^*)\succ 0$.
In \cref{subsec:implementation_algs}, we follow \cite{ding2019optimal} and apply Accelegrad to a penalized dual problem to construct this sequence; this method works when the dual SDP has a unique solution $\gamma^*$.
Fortunately, it is possible to extend this scheme even to SDPs where the dual has multiple optimal solutions, some of which do not satisfy the assumption $A(\gamma)\succ 0$.
Again, let penalty denote a bound on $\tr(Y^*)$. Then, the problem
\begin{align*}
	\max_{\gamma\in\R^m,\, T\in\S^k}\tr\left(T\right) + \text{penalty}\cdot\min\left(0,\,  \lambda_{\min}\begin{pmatrix}
		A(\gamma)/2 & B(\gamma)/2\\
		B(\gamma)^\intercal/2 & \frac{c(\gamma)}{k}I_k - T
	\end{pmatrix}\right) + \min(0, \lambda_{\min}(A(\gamma)) - \hat\mu)
\end{align*}
has as its minimizers, the minimizers $(\gamma,T)$ of the dual SDP for which $A(\gamma^*)\succeq \hat\mu$.
This ``trick'' was used in our phase retrieval experiments (see \cref{sec:additional_exp}).

\paragraph{Estimating regularity parameters.} CautiousAGD (\cref{alg:cautious-agd}) requires a number of regularity parameters (see \cref{as:regularity,as:qmmp}) for its theoretical guarantees to hold. Fortunately, these parameters are supplied by CertSDP (\cref{alg:exact_sdp_qmmp}). On the other hand, CertSDP itself requires access to a number of regularity parameters  $\hat\mu,\,\hat L,\, \hat R_p,\,\hat R_d,\,\hat\rho$ (see \cref{as:primal_dual_regularity}).
It is possible to set $\hat L$ and $\hat\rho$ in terms of $\hat R_d$ and the efficiently computable quantities $\norm{A_0}_2,\norm{A_1}_2,\dots,\norm{A_m}_2$ and $\norm{B_1}_F,\dots, \norm{B_m}_F$:
\begin{gather*}
    \hat L \!=\!\norm{A_0}_2\! + \norm{\!\begin{pmatrix}
        \norm{A_1}_2\\
        \vdots\\
        \norm{A_m}_2
        \end{pmatrix}\!}_2\hat R_d,\quad
    \hat\rho \!=\! \max\!\left(\!\frac{\hat\mu}{\hat R_d},\,
    \norm{\begin{pmatrix}
        \norm{A_1}_2\\
        \vdots\\
        \norm{A_m}_2
    \end{pmatrix}}_2\!\!,\,
    \frac{1}{\hat R_p}    
    \norm{\begin{pmatrix}
        \norm{B_1}_F\\
        \vdots\\
        \norm{B_m}_F
        \end{pmatrix}}_2
    \right).
\end{gather*}
These bounds may be much larger than other available estimates for $\hat L$ and $\hat\rho$ depending on problem context.
We currently see no rigorous way of removing $\hat \mu,\,\hat R_p,\, \hat R_d$.
In \cref{sec:numerical}, we offer a heuristic for guessing $\hat\mu$ and $\hat L$ that seems to work well numerically.

Having access to regularity parameters such as $\hat\mu$, $\hat R_p$, and $\hat R_d$ is common in the literature on first-order methods. Nonetheless, important future work includes deriving variants of CertSDP that may attempt to learn $\hat\mu$ or other quantities without \emph{a priori} knowledge of such quantities.

\paragraph{Restrictiveness of the class of exact QMP-like SDPs.}
Beyond regularity parameters, the class of exact QMP-like SDPs includes major structural assumptions, namely explicit knowledge of $Y^*$ on a given subspace and strict complementarity.
The algorithm CertSDP and CautiousAGD are designed with these structural assumptions in mind and are brittle to these assumptions being violated. For example, suppose $Y^*$ has rank $> k$ and we are only given access to the restriction of $Y^*$ to a rank $k$ subspace. In this case, it is impossible to construct a certificate of strict complementarity. To see why, suppose in the idealized setting that we have access to $\gamma^*$. By the assumption that $\rank(Y^*)>k$, it holds that $\rank(A(\gamma^*))< n -k$. In this case, it is not possible to contain $\gamma^*$ in a ball contained in the set $\set{\gamma:\, A(\gamma)\succ 0}$.

While CertSDP and CautiousAGD depend heavily on the structural assumptions of the class of exact QMP-like SDPs, we believe that these assumptions may be relaxed \emph{significantly} and leave this as important future work.

\section*{Acknowledgments}
This research is supported in part by ONR grant N00014-19-1-2321 and AFOSR grant FA9550-22-1-0365.
The authors wish to thank the review team for their feedback and suggestions that led to an improved presentation of the material.

\newpage
{
	\bibliographystyle{plainnat}

\begin{thebibliography}{76}
\providecommand{\natexlab}[1]{#1}
\providecommand{\url}[1]{\texttt{#1}}
\expandafter\ifx\csname urlstyle\endcsname\relax
  \providecommand{\doi}[1]{doi: #1}\else
  \providecommand{\doi}{doi: \begingroup \urlstyle{rm}\Url}\fi

\bibitem[Abbe et~al.(2015)Abbe, Bandeira, and Hall]{abbe2015exact}
E.~Abbe, A.~S. Bandeira, and G.~Hall.
\newblock Exact recovery in the stochastic block model.
\newblock \emph{IEEE Trans.\ Inform.\ Theory}, 62\penalty0 (1):\penalty0
  471--487, 2015.

\bibitem[Alizadeh(1995)]{alizadeh1995interior}
F.~Alizadeh.
\newblock Interior point methods in semidefinite programming with applications
  to combinatorial optimization.
\newblock \emph{{SIAM J.\ Optim.}}, 5\penalty0 (1):\penalty0 13--51, 1995.

\bibitem[Alizadeh et~al.(1997)Alizadeh, Haeberly, and
  Overton]{alizadeh1997complementarity}
F.~Alizadeh, J.~A. Haeberly, and M.~L. Overton.
\newblock Complementarity and nondegeneracy in semidefinite programming.
\newblock \emph{{Math.\ Program.}}, 77:\penalty0 111--128, 1997.

\bibitem[Argue et~al.(2023)Argue, K{\i}l{\i}n\c{c}-Karzan, and
  Wang]{argue2020necessary}
C.J. Argue, F.~K{\i}l{\i}n\c{c}-Karzan, and A.~L. Wang.
\newblock Necessary and sufficient conditions for rank-one generated cones.
\newblock \emph{{Math.\ Oper.\ Res.}}, 48\penalty0 (1):\penalty0 100--126,
  2023.

\bibitem[Baes et~al.(2013)Baes, Burgisser, and Nemirovski]{baes2013randomized}
M.~Baes, M.~Burgisser, and A.~Nemirovski.
\newblock A randomized mirror-prox method for solving structured large-scale
  matrix saddle-point problems.
\newblock \emph{{SIAM J.\ Optim.}}, 23\penalty0 (2):\penalty0 934--962, 2013.

\bibitem[Beck(2007)]{beck2007quadratic}
A.~Beck.
\newblock Quadratic matrix programming.
\newblock \emph{{SIAM J.\ Optim.}}, 17\penalty0 (4):\penalty0 1224--1238, 2007.

\bibitem[Beck et~al.(2012)Beck, Drori, and Teboulle]{beck2012new}
A.~Beck, Y.~Drori, and M.~Teboulle.
\newblock A new semidefinite programming relaxation scheme for a class of
  quadratic matrix problems.
\newblock \emph{{Oper.\ Res.\ Lett.}}, 40\penalty0 (4):\penalty0 298--302,
  2012.

\bibitem[Ben-Tal and Nemirovski(2001)]{benTal2001lectures}
A.~Ben-Tal and A.~Nemirovski.
\newblock \emph{Lectures on Modern Convex Optimization}, volume~2 of
  \emph{MPS-SIAM Ser.\ Optim.}
\newblock SIAM, 2001.

\bibitem[Ben-Tal and Nemirovski(2012)]{ben2012solving}
A.~Ben-Tal and A.~Nemirovski.
\newblock Solving large scale polynomial convex problems on $\ell_1$/nuclear
  norm balls by randomized first-order algorithms.
\newblock \emph{CoRR}, 2012.

\bibitem[Boumal et~al.(2016)Boumal, Voroninski, and Bandeira]{boumal2016non}
N.~Boumal, V.~Voroninski, and A.~Bandeira.
\newblock The non-convex {Burer--Monteiro} approach works on smooth
  semidefinite programs.
\newblock In \emph{Advances in Neural Information Processing Systems},
  volume~29, 2016.

\bibitem[Burer and K{\i}l{\i}n\c{c}-Karzan(2017)]{burer2017how}
S.~Burer and F.~K{\i}l{\i}n\c{c}-Karzan.
\newblock How to convexify the intersection of a second order cone and a
  nonconvex quadratic.
\newblock \emph{{Math.\ Program.}}, 162:\penalty0 393--429, 2017.

\bibitem[Burer and Monteiro(2003)]{burer2003nonlinear}
S.~Burer and R.~D.C. Monteiro.
\newblock A nonlinear programming algorithm for solving semidefinite programs
  via low-rank factorization.
\newblock \emph{{Math.\ Program.}}, 95:\penalty0 329--357, 2003.

\bibitem[Burer and Yang(2014)]{burer2014trust}
S.~Burer and B.~Yang.
\newblock The trust region subproblem with non-intersecting linear constraints.
\newblock \emph{{Math.\ Program.}}, 149:\penalty0 253--264, 2014.

\bibitem[Burer and Ye(2019)]{burer2019exact}
S.~Burer and Y.~Ye.
\newblock Exact semidefinite formulations for a class of (random and
  non-random) nonconvex quadratic programs.
\newblock \emph{{Math.\ Program.}}, 181:\penalty0 1--17, 2019.

\bibitem[Cand\`{e}s et~al.(2015)Cand\`{e}s, Eldar, Strohmer, and
  Voroninski]{candes2015phase}
E.~J. Cand\`{e}s, Y.~C. Eldar, T.~Strohmer, and V.~Voroninski.
\newblock Phase retrieval via matrix completion.
\newblock \emph{{SIAM Rev.}}, 57\penalty0 (2):\penalty0 225--251, 2015.

\bibitem[Carmon and Duchi(2018)]{carmon2018analysis}
Y.~Carmon and J.~C. Duchi.
\newblock Analysis of {Krylov} subspace solutions of regularized nonconvex
  quadratic problems.
\newblock In \emph{Proceedings of the 32nd International Conference on Neural
  Information Processing Systems}, pages 10728--10738, 2018.

\bibitem[Chambolle and Pock(2016)]{chambolle2016ergodic}
A.~Chambolle and T.~Pock.
\newblock On the ergodic convergence rates of a first-order primal--dual
  algorithm.
\newblock \emph{{Math.\ Program.}}, 159:\penalty0 253--287, 2016.

\bibitem[Cifuentes(2021)]{cifuentes2021burer}
D.~Cifuentes.
\newblock On the {Burer--Monteiro} method for general semidefinite programs.
\newblock \emph{Opt.\ Lett.}, 15\penalty0 (6):\penalty0 2299--2309, 2021.

\bibitem[Cifuentes and Moitra(2022)]{cifuentes2019polynomial}
D.~Cifuentes and A.~Moitra.
\newblock Polynomial time guarantees for the {Burer-Monteiro} method.
\newblock \emph{Advances in Neural Information Processing Systems},
  35:\penalty0 23923--23935, 2022.

\bibitem[d'Aspremont and {El Karoui}(2014)]{d2014stochastic}
A.~d'Aspremont and N.~{El Karoui}.
\newblock A stochastic smoothing algorithm for semidefinite programming.
\newblock \emph{{SIAM J.\ Optim.}}, 24\penalty0 (3):\penalty0 1138--1177, 2014.

\bibitem[{de Carli Silva} and Tun\c{c}el(2019)]{de2019strict}
M.~K. {de Carli Silva} and L.~Tun\c{c}el.
\newblock Strict complementarity in semidefinite optimization with elliptopes
  including the maxcut {SDP}.
\newblock \emph{{SIAM J.\ Optim.}}, 29\penalty0 (4):\penalty0 2650--2676, 2019.

\bibitem[Devolder et~al.(2013)Devolder, Glineur, and
  Nesterov]{devolder2013first}
O.~Devolder, F.~Glineur, and Y.~Nesterov.
\newblock First-order methods with inexact oracle: the strongly convex case.
\newblock Technical Report 2013016, 2013.

\bibitem[Devolder et~al.(2014)Devolder, Glineur, and
  Nesterov]{devolder2014first}
O.~Devolder, F.~Glineur, and Y.~Nesterov.
\newblock First-order methods of smooth convex optimization with inexact
  oracle.
\newblock \emph{{Math.\ Program.}}, 146\penalty0 (1):\penalty0 37--75, 2014.

\bibitem[Ding and Udell(2021)]{ding2021simplicity}
L.~Ding and M.~Udell.
\newblock On the simplicity and conditioning of low rank semidefinite programs.
\newblock \emph{{SIAM J.\ Optim.}}, 31\penalty0 (4):\penalty0 2614--2637, 2021.

\bibitem[Ding and Wang(2023)]{ding2023sharp}
L.~Ding and A.~L. Wang.
\newblock Sharpness and well-conditioning of nonsmooth convex formulations in
  statistical signal recovery.
\newblock \emph{arXiv preprint arXiv:2307.06873}, 2023.

\bibitem[Ding et~al.(2021{\natexlab{a}})Ding, Yurtsever, Cevher, Tropp, and
  Udell]{ding2019optimal}
L.~Ding, A.~Yurtsever, V.~Cevher, J.~A. Tropp, and M.~Udell.
\newblock An optimal-storage approach to semidefinite programming using
  approximate complementarity.
\newblock \emph{{SIAM J.\ Optim.}}, 31\penalty0 (4):\penalty0 2695--2725,
  2021{\natexlab{a}}.

\bibitem[Ding et~al.(2021{\natexlab{b}})Ding, Yurtsever, Cevher, Tropp, and
  Udell]{ding2021optimal}
L.~Ding, A.~Yurtsever, V.~Cevher, J.~A. Tropp, and M.~Udell.
\newblock An optimal-storage approach to semidefinite programming using
  approximate complementarity.
\newblock \emph{{SIAM J.\ Optim.}}, 31\penalty0 (4):\penalty0 2695--2725,
  2021{\natexlab{b}}.

\bibitem[Drusvyatskiy and Lewis(2018)]{drusvyatskiy2018error}
D.~Drusvyatskiy and A.~S. Lewis.
\newblock Error bounds, quadratic growth, and linear convergence of proximal
  methods.
\newblock \emph{{Math.\ Oper.\ Res.}}, 43\penalty0 (3):\penalty0 919--948,
  2018.

\bibitem[Fradkov and Yakubovich(1979)]{fradkov1979s-procedure}
A.~L. Fradkov and V.~A. Yakubovich.
\newblock The {S}-procedure and duality relations in nonconvex problems of
  quadratic programming.
\newblock \emph{Vestnik Leningrad Univ.\ Math.}, 6:\penalty0 101--109, 1979.

\bibitem[Friedlander and Mac\^edo(2016)]{friedlander2016low}
M.~P. Friedlander and I.~Mac\^edo.
\newblock Low-rank spectral optimization via gauge duality.
\newblock \emph{SIAM Journal on Scientific Computing}, 38\penalty0
  (3):\penalty0 A1616--A1638, 2016.

\bibitem[Garber and A.~Kaplan(2022)]{garber2022efficient}
D.~Garber and Atara A.~Kaplan.
\newblock On the efficient implementation of the matrix exponentiated gradient
  algorithm for low-rank matrix optimization.
\newblock \emph{{Math.\ Oper.\ Res.}}, 2022.

\bibitem[Goemans and Williamson(1995)]{goemans1995improved}
M.~X. Goemans and D.~P. Williamson.
\newblock Improved approximation algorithms for maximum cut and satisfiability
  problems using semidefinite programming.
\newblock \emph{J. ACM}, 42\penalty0 (6):\penalty0 1115--1145, 1995.

\bibitem[Goldfarb and Scheinberg(1998)]{goldfarb1998interior}
D.~Goldfarb and K.~Scheinberg.
\newblock Interior point trajectories in semidefinite programming.
\newblock \emph{{SIAM J.\ Optim.}}, 8\penalty0 (4):\penalty0 871--886, 1998.

\bibitem[Hamedani and Aybat(2021)]{hamedani2021primal}
E.~Y. Hamedani and N.~C. Aybat.
\newblock A primal-dual algorithm with line search for general convex-concave
  saddle point problems.
\newblock \emph{{SIAM J.\ Optim.}}, 31\penalty0 (2):\penalty0 1299--1329, 2021.

\bibitem[Hazan and Koren(2016)]{hazan2016linear}
E.~Hazan and T.~Koren.
\newblock A linear-time algorithm for trust region problems.
\newblock \emph{{Math.\ Program.}}, 158:\penalty0 363--381, 2016.

\bibitem[Ho-Nguyen and K{\i}l{\i}n\c{c}-Karzan(2017)]{hoNguyen2017second}
N.~Ho-Nguyen and F.~K{\i}l{\i}n\c{c}-Karzan.
\newblock A second-order cone based approach for solving the {Trust Region
  Subproblem} and its variants.
\newblock \emph{{SIAM J.\ Optim.}}, 27\penalty0 (3):\penalty0 1485--1512, 2017.

\bibitem[Jeyakumar and Li(2014)]{jeyakumar2013trust}
V.~Jeyakumar and G.~Y. Li.
\newblock Trust-region problems with linear inequality constraints: {E}xact
  {SDP} relaxation, global optimality and robust optimization.
\newblock \emph{{Math.\ Program.}}, 147:\penalty0 171--206, 2014.

\bibitem[Jiang and Li(2019)]{jiang2019novel}
R.~Jiang and D.~Li.
\newblock Novel reformulations and efficient algorithms for the {Generalized
  Trust Region Subproblem}.
\newblock \emph{{SIAM J.\ Optim.}}, 29\penalty0 (2):\penalty0 1603--1633, 2019.

\bibitem[Juditsky and Nemirovski(2011)]{juditsky2011first}
A.~Juditsky and A.~Nemirovski.
\newblock First order methods for nonsmooth convex large-scale optimization,
  ii: utilizing problems structure.
\newblock \emph{Optimization for Machine Learning}, 30\penalty0 (9):\penalty0
  149--183, 2011.

\bibitem[K{\i}l{\i}n\c{c}-Karzan and Wang(2021)]{kkz2021exactness}
F.~K{\i}l{\i}n\c{c}-Karzan and A.~L. Wang.
\newblock Exactness in {SDP} relaxations of {QCQP}s: {T}heory and applications.
\newblock Tut.\ in Oper.\ Res. INFORMS, 2021.

\bibitem[Lan et~al.(2011)Lan, Lu, and Monteiro]{lanprimal}
G.~Lan, Z.~Lu, and R.~D.C. Monteiro.
\newblock Primal-dual first-order methods with $o(1/\epsilon)$
  iteration-complexity for cone programming.
\newblock \emph{{Math.\ Program.}}, 126:\penalty0 1--29, 2011.

\bibitem[Laurent and Poljak(1995)]{laurent1995positive}
M.~Laurent and S.~Poljak.
\newblock On a positive semidefinite relaxation of the cut polytope.
\newblock \emph{{Linear Algebra Appl.}}, 223-224:\penalty0 439--461, 1995.

\bibitem[Levy et~al.(2018)Levy, Yurtsever, and Cevher]{levy2018online}
K.~Y. Levy, A.~Yurtsever, and V.~Cevher.
\newblock Online adaptive methods, universality and acceleration.
\newblock In \emph{Advances in Neural Information Processing Systems}, 2018.

\bibitem[Locatelli(2016)]{locatelli2016exactness}
M.~Locatelli.
\newblock Exactness conditions for an {SDP} relaxation of the extended trust
  region problem.
\newblock \emph{{Oper.\ Res.\ Lett.}}, 10\penalty0 (6):\penalty0 1141--1151,
  2016.

\bibitem[Locatelli(2023)]{locatelli2020kkt}
M.~Locatelli.
\newblock {KKT}-based primal-dual exactness conditions for the {Shor}
  relaxation.
\newblock \emph{{J.\ Global Optim.}}, 86\penalty0 (2):\penalty0 285--301, 2023.

\bibitem[Lu et~al.(2007)Lu, Nemirovski, and Monteiro]{lu2007large}
Z.~Lu, A.~Nemirovski, and R.~D.C. Monteiro.
\newblock Large-scale semidefinite programming via a saddle point mirror-prox
  algorithm.
\newblock \emph{{Math.\ Program.}}, 109:\penalty0 211--237, 2007.

\bibitem[Majumdar et~al.(2020)Majumdar, Hall, and Ahmadi]{majumdar2020recent}
A.~Majumdar, G.~Hall, and A.~A. Ahmadi.
\newblock Recent scalability improvements for semidefinite programming with
  applications in machine learning, control, and robotics.
\newblock \emph{Annual Review of Control, Robotics, and Autonomous Systems},
  3:\penalty0 331--360, 2020.

\bibitem[Mixon et~al.(2017)Mixon, Villar, and Ward]{mixon2016clustering}
D.~G. Mixon, S.~Villar, and R.~Ward.
\newblock Clustering subgaussian mixtures by semidefinite programming.
\newblock \emph{Information and Inference: A Journal of the IMA}, 6\penalty0
  (4):\penalty0 389--415, 2017.

\bibitem[Mor{\'e}(1993)]{more1993generalizations}
J.~J. Mor{\'e}.
\newblock Generalizations of the {Trust Region Problem}.
\newblock \emph{Optimization methods and Software}, 2\penalty0 (3-4):\penalty0
  189--209, 1993.

\bibitem[Mor{\'e} and Sorensen(1983)]{more1983computing}
J.~J. Mor{\'e} and D.~C. Sorensen.
\newblock Computing a trust region step.
\newblock \emph{SIAM Journal on Scientific and Statistical Computing},
  4:\penalty0 553--572, 1983.

\bibitem[Nemirovski(2004)]{nemirovski2004prox}
A.~Nemirovski.
\newblock Prox-method with rate of convergence o(1/t) for variational
  inequalities with lipschitz continuous monotone operators and smooth
  convex-concave saddle point problems.
\newblock \emph{{SIAM J.\ Optim.}}, 15\penalty0 (1):\penalty0 229--251, 2004.

\bibitem[Nesterov(2005{\natexlab{a}})]{nesterov2005excessive}
Y.~Nesterov.
\newblock Excessive gap technique in nonsmooth convex minimization.
\newblock \emph{{SIAM J.\ Optim.}}, 16\penalty0 (1):\penalty0 235--249,
  2005{\natexlab{a}}.

\bibitem[Nesterov(2005{\natexlab{b}})]{nesterov2005smooth}
Y.~Nesterov.
\newblock Smooth minimization of non-smooth functions.
\newblock \emph{{Math.\ Program.}}, 103:\penalty0 127--152, 2005{\natexlab{b}}.

\bibitem[Nesterov(2018)]{nesterov2018lectures}
Y.~Nesterov.
\newblock \emph{Lectures on convex optimization}, volume 137 of \emph{Springer
  Optimization and Its Applications}.
\newblock Springer, 2018.

\bibitem[Nesterov and Nemirovskii(1994)]{nesterov1994interior}
Y.~Nesterov and A.~Nemirovskii.
\newblock \emph{Interior-point polynomial algorithms in convex programming}.
\newblock SIAM, 1994.

\bibitem[O'Donoghue et~al.(2016)O'Donoghue, Chu, Parikh, and
  Boyd]{odonoghue2016conic}
B.~O'Donoghue, E.~Chu, N.~Parikh, and S.~Boyd.
\newblock Conic optimization via operator splitting and homogeneous self-dual
  embedding.
\newblock \emph{Journal of Optimization Theory and Applications}, 169\penalty0
  (3):\penalty0 1042--1068, 2016.

\bibitem[Ouyang and Xu(2021)]{ouyang2021lower}
Y.~Ouyang and Y.~Xu.
\newblock Lower complexity bounds of first-order methods for convex-concave
  bilinear saddle-point problems.
\newblock \emph{{Math.\ Program.}}, 185:\penalty0 1--35, 2021.

\bibitem[Palaniappan and Bach(2016)]{palaniappan2016stochastic}
B.~Palaniappan and F.~Bach.
\newblock Stochastic variance reduction methods for saddle-point problems.
\newblock In \emph{Advances in Neural Information Processing Systems},
  volume~29, 2016.

\bibitem[Raghavendra(2008)]{raghavendra2008optimal}
P.~Raghavendra.
\newblock Optimal algorithms and inapproximability results for every {CSP}?
\newblock In \emph{Proceedings of the fortieth annual ACM symposium on Theory
  of computing}, pages 245--254, 2008.

\bibitem[Rujeerapaiboon et~al.(2019)Rujeerapaiboon, Schindler, Kuhn, and
  Wiesemann]{rujeerapaiboon2019size}
N.~Rujeerapaiboon, K.~Schindler, D.~Kuhn, and W.~Wiesemann.
\newblock Size matters: Cardinality-constrained clustering and outlier
  detection via conic optimization.
\newblock \emph{{SIAM J.\ Optim.}}, 29\penalty0 (2):\penalty0 1211--1239, 2019.

\bibitem[Sard(1942)]{sard1942measure}
A.~Sard.
\newblock The measure of the critical values of differentiable maps.
\newblock \emph{Bull.\ Amer.\ Math.\ Soc.}, 48\penalty0 (12):\penalty0
  883--890, 1942.

\bibitem[Shinde et~al.(2021)Shinde, Narayanan, and
  Saunderson]{shinde2021memory}
N.~Shinde, V.~Narayanan, and J.~Saunderson.
\newblock Memory-efficient structured convex optimization via extreme point
  sampling.
\newblock \emph{SIAM Journal on Mathematics of Data Science}, 3\penalty0
  (3):\penalty0 787--814, 2021.

\bibitem[Shor(1990)]{shor1990dual}
N.~Z. Shor.
\newblock Dual quadratic estimates in polynomial and boolean programming.
\newblock \emph{Ann.\ Oper.\ Res.}, 25:\penalty0 163--168, 1990.

\bibitem[Sion(1958)]{sion1958general}
M.~Sion.
\newblock On general minimax theorems.
\newblock \emph{Pacific J.\ Math.}, 8\penalty0 (1):\penalty0 171--176, 1958.

\bibitem[Souto et~al.(2020)Souto, Garcia, and Veiga]{souto2020exploiting}
M.~Souto, J.~D. Garcia, and \'A. Veiga.
\newblock Exploiting low-rank structure in semidefinite programming by
  approximate operator splitting.
\newblock \emph{Optimization}, pages 1--28, 2020.

\bibitem[Sturm and Zhang(2003)]{sturm2003cones}
J.~F. Sturm and S.~Zhang.
\newblock On cones of nonnegative quadratic functions.
\newblock \emph{{Math.\ Oper.\ Res.}}, 28\penalty0 (2):\penalty0 246--267,
  2003.

\bibitem[Tseng(2008)]{tseng2008accelerated}
P.~Tseng.
\newblock On accelerated proximal gradient methods for convex-concave
  optimization, 2008.

\bibitem[Vandenberghe and Boyd(1996)]{vandenberghe1996semidefinite}
L.~Vandenberghe and S.~Boyd.
\newblock Semidefinite programming.
\newblock \emph{{SIAM Rev.}}, 38\penalty0 (1):\penalty0 49--95, 1996.

\bibitem[Waldspurger and Waters(2020)]{waldspurger2020rank}
I.~Waldspurger and A.~Waters.
\newblock Rank optimality for the {Burer--Monteiro} factorization.
\newblock \emph{{SIAM J.\ Optim.}}, 30\penalty0 (3):\penalty0 2577--2602, 2020.

\bibitem[Wang and K{\i}l{\i}n\c{c}-Karzan(2020)]{wang2020geometric}
A.~L. Wang and F.~K{\i}l{\i}n\c{c}-Karzan.
\newblock A geometric view of {SDP} exactness in {QCQPs} and its applications.
\newblock \emph{{arXiv preprint}}, 2011.07155, 2020.

\bibitem[Wang and
  K{\i}l{\i}n\c{c}-Karzan(2022{\natexlab{a}})]{wang2020generalized}
A.~L. Wang and F.~K{\i}l{\i}n\c{c}-Karzan.
\newblock The generalized trust region subproblem: solution complexity and
  convex hull results.
\newblock \emph{{Math.\ Program.}}, 191\penalty0 (2):\penalty0 445--486,
  2022{\natexlab{a}}.

\bibitem[Wang and
  K{\i}l{\i}n\c{c}-Karzan(2022{\natexlab{b}})]{wang2021tightness}
A.~L. Wang and F.~K{\i}l{\i}n\c{c}-Karzan.
\newblock On the tightness of {SDP} relaxations of {QCQP}s.
\newblock \emph{{Math.\ Program.}}, 193\penalty0 (1):\penalty0 33--73,
  2022{\natexlab{b}}.

\bibitem[Wang et~al.(2023)Wang, Lu, and
  K{\i}l{\i}n\c{c}-Karzan]{wang2021implicit}
A.~L. Wang, Y.~Lu, and F.~K{\i}l{\i}n\c{c}-Karzan.
\newblock Implicit regularity and linear convergence rates for the generalized
  trust-region subproblem.
\newblock \emph{{SIAM J.\ Optim.}}, 33\penalty0 (2):\penalty0 1250--1278, 2023.

\bibitem[Yang et~al.(2023)Yang, Liang, Carlone, and Toh]{yang2021inexact}
H.~Yang, L.~Liang, L.~Carlone, and K.~Toh.
\newblock An inexact projected gradient method with rounding and lifting by
  nonlinear programming for solving rank-one semidefinite relaxation of
  polynomial optimization.
\newblock \emph{{Math.\ Program.}}, 201\penalty0 (1-2):\penalty0 409--472,
  2023.

\bibitem[Yurtsever et~al.(2019)Yurtsever, Fercoq, and
  Cevher]{yurtsever2019conditional}
A.~Yurtsever, O.~Fercoq, and V.~Cevher.
\newblock A conditional-gradient-based augmented {Lagrangian} framework.
\newblock In \emph{International Conference on Machine Learning}, pages
  7272--7281, 2019.

\bibitem[Yurtsever et~al.(2021)Yurtsever, Tropp, Fercoq, Udell, and
  Cevher]{yurtsever2021scalable}
A.~Yurtsever, J.~A. Tropp, O.~Fercoq, M.~Udell, and V.~Cevher.
\newblock Scalable semidefinite programming.
\newblock \emph{SIAM Journal on Mathematics of Data Science}, 3\penalty0
  (1):\penalty0 171--200, 2021.

\end{thebibliography}

}

\newpage
\begin{appendix}

\section{Deferred proofs}
\label{sec:deferred}

\begin{proof}[Proof of \cref{lem:quadratic_moving_center}]
    Let $\Delta \coloneqq \tilde X - X_L$. Then,
    \begin{align*}
    \frac{L}{2}\norm{X - X_L}_F^2 &= \frac{L}{2}\norm{X - \tilde X + \Delta}_F^2\\
&= \frac{\tilde L}{2}\norm{X - \tilde X}_F^2 + \frac{\tilde \mu}{2}\norm{X - \tilde X}_F^2 + L\ip{X - \tilde X, \Delta} + \frac{L}{2}\norm{\Delta}_F^2,
\end{align*}
    where the second equality follows from expanding the square and the fact that $L = \tilde L + \tilde \mu$.
    Moreover, 
    \begin{align*}
    0\leq \frac{L}{2}\norm{\sqrt{\frac{\tilde\mu}{L}} (X - \tilde X) + \sqrt{\frac{L}{\tilde\mu}}\Delta}_F^2 = \frac{\tilde\mu}{2}\norm{X - \tilde X}_F^2 + L\ip{X - \tilde X, \Delta} + L\kappa\norm{\Delta}_F^2.
    \end{align*}
    Combining these two inequalities gives
    \begin{align*}
    \frac{L}{2}\norm{X - X_L}_F^2 &\geq \frac{\tilde L}{2}\norm{X - \tilde X}_F^2 - \frac{L\delta^2}{2}\left(2\kappa - 1\right).\qedhere
    \end{align*}
    \end{proof}

The following proof is adapted from \cite[]{nesterov2018lectures}.
\begin{proof}
[Proof of \cref{lem:estimating_sequence_recurrence}]
It is evident that $\phi_t(X)$ are quadratic matrix functions of the form \eqref{eq:estimating_sequence_form} with $V_0=X_0$ and $\phi_0^*=Q(X_0)$. The remainder of the proof verifies the recurrences on $V_{t+1}$ and $\phi^*_{t+1}$. We suppose that the stated form holds for some $t$, and we will show that it will hold for $t+1$ as well. We compute
\begin{align*}
\frac{1}{\tilde \mu}\grad \phi_{t+1}(X) &= (1-\alpha)(X - V_t) + \alpha \left(X - \left(\Xi_t - \frac{1}{\tilde \mu}\tilde g_t\right)\right).
\end{align*}
We deduce that
$V_{t+1} = (1-\alpha)V_t + \alpha\left(\Xi_t - \frac{1}{\tilde \mu}\tilde g_t\right)$. Noting that $\phi_{t+1}^*=\phi_{t+1}(V_{t+1})$, and applying the recursive definition of $\phi_{t+1}(X)$ gives us
\begin{align*}
\phi_{t+1}^* &= (1-\alpha)\left(\phi_t^* + \frac{\tilde \mu}{2}\norm{V_{t+1} - V_t}_F^2\right)\\
&\qquad + \alpha \left( Q(X_{t+1}) + \frac{1}{2\tilde L}\norm{\tilde g_t}_F^2 + \ip{\tilde g_t, V_{t+1} - \Xi_t} + \frac{\tilde \mu}{2}\norm{V_{t+1} - \Xi_t}_F^2 \right)\\
&= (1-\alpha)\phi_t^* + \alpha\left(Q(X_{t+1}) + \frac{1}{2\tilde L}\norm{\tilde g_t}_F^2\right)\\
&\qquad + (1-\alpha)\frac{\tilde\mu}{2}\norm{V_{t+1} - V_t}_F^2 + \frac{\alpha\tilde \mu}{2}\norm{V_{t+1} - (\Xi_t - \tfrac{1}{\tilde\mu}\tilde g_t)}_F^2 - \frac{\alpha}{2\tilde \mu}\norm{\tilde g_t}_F^2\\
&= (1-\alpha)\phi_t^* + \alpha\left(Q(X_{t+1}) + \frac{1}{2\tilde L}\norm{\tilde g_t}_F^2\right)\\
&\qquad + \frac{\tilde\mu(1-\alpha)\alpha^2}{2}\norm{V_{t} - (\Xi_t - \tfrac{1}{\tilde\mu}\tilde g_t)}_F^2 + \frac{\tilde \mu\alpha (1-\alpha)^2}{2}\norm{V_{t} - (\Xi_t - \tfrac{1}{\tilde\mu}\tilde g_t)}_F^2 - \frac{\alpha}{2\tilde \mu}\norm{\tilde g_t}_F^2\\
&= (1-\alpha)\phi_t^* + \alpha\left(Q(X_{t+1}) + \frac{1}{2\tilde L}\norm{\tilde g_t}_F^2\right)\\
&\qquad + \alpha(1-\alpha)\left(\frac{\tilde \mu}{2}\norm{\Xi_t - V_t}_F^2 + \ip{\tilde g_t, V_t - \Xi_t} \right) -\frac{ \alpha^2}
{2\tilde \mu}\norm{\tilde g_t}_F^2, \end{align*}
where the third equation follows from substituting the expression for $V_{t+1}$, and the last one from regrouping the terms.
\end{proof}

The following proof is adapted from \cite[Page 92]{nesterov2018lectures}.
\begin{proof}
[Proof of \cref{lem:Xi_extragradient}]
Note that
\begin{align*}
\Xi_t &= \frac{X_t + \alpha V_t}{1+\alpha}\\
X_{t+1} &= \Xi_t - \frac{\tilde g_t}{\tilde L}\\
V_{t+1} &= (1-\alpha)V_t + \alpha\left(\Xi_t - \frac{1}{\tilde\mu}\tilde g_t\right).
\end{align*}
Therefore,
\begin{align*}
V_{t+1} &= (1-\alpha)\frac{(1+\alpha)\Xi_t - X_t}{\alpha} + \alpha\left(\Xi_t - \frac{1}{\tilde\mu}\tilde g_t\right)\\
&= X_t + \frac{1}{\alpha}\left(\Xi_t - X_t - \frac{1}{\tilde L}\tilde g_t\right)\\
&= X_t + \frac{1}{\alpha}\left(X_{t+1} - X_t\right).
\end{align*}
Then,
\begin{align*}
\Xi_{t+1} &= X_{t+1} + \frac{\alpha}{1+\alpha}\left(V_{t+1} - X_{t+1}\right)\\
&= X_{t+1} + \frac{1-\alpha}{1+\alpha}\left(X_{t+1} - X_t\right).\qedhere
\end{align*}
\end{proof}

\begin{proof}[Proof of \cref{lem:error_inexact_prox_1}]
    It is clear that $Q(X_0)\leq \phi_0^*$. Thus, consider $X_{t+1}$ with $t\geq 0$. By induction and \cref{lem:estimating_sequence_recurrence},
    \begin{align*}
    \phi_{t+1}^* &\geq (1-\alpha )Q(X_t) + \alpha Q(X_{t+1}) + \left(\frac{\alpha}{2\tilde L} - \frac{\alpha^2}{2\tilde \mu}\right)\norm{\tilde g_t}_F^2\\
    &\qquad + \alpha (1-\alpha)\ip{\tilde g_t, V_t - \Xi_t} - (1-\alpha)\left(2\kappa E^{(1)}_t\right).
    \end{align*}
    As $X_{t+1}$ satisfies $Q_L(\Xi_t; X_{t+1}) \leq Q^*(\Xi_t) + \epsilon_t$, we deduce (see \cref{thm:combined_inequality_eps}) that
    \begin{align*}
    Q(X_t) &\geq Q(X_{t+1}) + \frac{1}{2\tilde L}\norm{\tilde g_t}_F^2 + \ip{\tilde g_t, X_t - \Xi_t}   + \frac{\tilde \mu}{2}\norm{X_t - \Xi_t}_F^2 - 2\kappa\epsilon_t.
    \end{align*}
    These two inequalities together lead to
    \begin{align*}
    \phi^*_{t+1}
    &\geq Q(X_{t+1}) - 2\kappa(1-\alpha)(E^{(1)}_t + \epsilon_t)\\
    &\qquad  + \left(\frac{\alpha}{2\tilde L} - \frac{\alpha^2}{2\tilde \mu} + \frac{1-\alpha}{2\tilde L}\right)\norm{\tilde g_t}_F^2 + (1-\alpha)\ip{\tilde g_t, \alpha(V_t-\Xi_t) + (X_t - \Xi_t)}.
    \end{align*}
    It is straightforward to show that the two quantities on the final line are identically zero using the relations $\alpha^2 = \tilde \mu / \tilde L$ and $\Xi_t = \tfrac{X_t + \alpha V_t}{1+\alpha}$ (see \cref{lem:Xi_extragradient}).
    \end{proof}
    
    \begin{proof}[Proof of \cref{lem:error_inexact_prox_2}]
        The statement holds holds for $t = 0$. Thus, consider $\phi_{t+1}$ for $t\geq 0$. 
        By definition
        \begin{align*}
        \phi_{t+1}(X) &= (1-\alpha)\phi_t(X) + \alpha \left( Q(X_{t+1}) + \frac{1}{2\tilde L}\norm{\tilde g_t}_F^2 + \ip{\tilde g_t, X - \Xi_t} + \frac{\tilde \mu}{2}\norm{X - \Xi_t}_F^2 \right).
        \end{align*}
        As $X_{t+1}$ satisfies $Q_L(\Xi_t; X_{t+1}) \leq Q^*(\Xi_t) + \epsilon_t$, we deduce (see \cref{thm:combined_inequality_eps}) that
        \begin{align*}
        Q(X) \geq Q(X_{t+1}) + \frac{1}{2\tilde L} \norm{\tilde g_t}_F^2 + \ip{\tilde g_t, X - \Xi_t} + \frac{\tilde \mu}{2}\norm{X - \Xi_t}_F^2 - 2\kappa\epsilon.
        \end{align*}
        Then, these inequalities combined with the inductive hypothesis give
        \begin{align*}
        \phi_{t+1}(X) &\leq (1-\alpha) \phi_t(X) + \alpha Q(X) + 2\kappa\alpha\epsilon_t\\
        &= (1 - (1-\alpha)^{t+1}) Q(X) + (1-\alpha)(\phi_t(X) -(1- (1-\alpha)^{t})Q(X)) + 2\kappa\alpha\epsilon_t\\
        &\leq (1 - (1-\alpha)^{t+1}) Q(X) + (1-\alpha)^{t+1}\phi_0(X) + 2\kappa\left((1-\alpha)E^{(2)}_t+\alpha\epsilon_t\right). \qedhere
        \end{align*}
\end{proof}

\begin{proof}
[Proof of \cref{cor:qmmp_subopt_E}]
    Let $X^*_\cU$ denote the optimizer of \eqref{eq:qmmp} so that $Q(X^*_\cU) = \Opt_{\eqref{eq:qmmp}}$.
    Then, \cref{lem:error_inexact_prox_1,lem:error_inexact_prox_2} give
    \begin{align*}
    Q(X_t)- \Opt_{\eqref{eq:qmmp}} &\leq \phi_t^* + 2\kappa E_t^{(1)} - Q(X^*_\cU)\\
    &\leq\phi_t(X^*_\cU) +2\kappa E_t^{(1)} - Q(X^*_\cU)\\
    &\leq (1 - (1-\alpha)^t)Q(X^*_\cU) + (1-\alpha)^t \phi_0(X^*_\cU) + 2\kappa E_t - Q(X^*_\cU)\\
    &= (1-\alpha)^t \left(\phi_0(X^*_\cU)- \Opt_{\eqref{eq:qmmp}}\right) + 2\kappa E_t.
    \end{align*}
Note also that by the definition of $\phi_0(\cdot)$ and the $\mu$-strong convexity of $Q$, we have
    \begin{align*}
    \phi_0(X^*_\cU) - \Opt_{\eqref{eq:qmmp}} &= Q(X_0) - \Opt_{\eqref{eq:qmmp}} + \frac{\tilde\mu}{2}\norm{X^*_\cU - X_0}_F^2\\
    &\leq 2\left(Q(X_0) - \Opt_{\eqref{eq:qmmp}}\right).
    \end{align*}
    Combining the two inequalities completes the proof.
\end{proof}

\begin{proof}[Proof of \cref{lem:xi_t_norm_bound}]
    Let $\tilde \gamma\in\argmax_{\gamma\in\cU}q(\gamma,X_0)$.
    By $\mu$-strong convexity of $Q(X)$, we have that
    \begin{align*}
    Q(X) &\geq q(\tilde\gamma, X)\\
    &\geq q(\tilde\gamma, X_0) + \ip{\grad_2\, q(\tilde\gamma, X_0), X - X_0} + \frac{\mu}{2}\norm{X - X_0}_F^2\\
    &= Q(X_0) - \frac{1}{2\mu} \norm{\grad_2\, q(\tilde\gamma, X_0)}_F^2 + \frac{\mu}{2}\norm{X - X_0 + \frac{\grad_2\,q(\tilde\gamma, X_0)}{\mu}}_F^2.
    \end{align*}
    In particular,
    taking $X = \argmin_{X\in\R^{(n-k)\times k}} Q(X)$ gives
    \begin{align*}
    Q(X_0) - \Opt_{\eqref{eq:qmmp}}\leq \frac{\norm{\grad_2\, q(\tilde \gamma, X_0)}_F^2}{2\mu}\leq \frac{\mu\kappa^2 R^2}{2}, 
    \end{align*}
    where the last inequality follows from \cref{as:regularity}. 
    This proves the first claim.
    Next, by \cref{thm:cautious_agd_oracle}, we have that for all $t\geq 0$, that $Q(X_t)-Q(X_0)\leq Q(X_t)-\Opt_{\eqref{eq:qmmp}} \leq 2\mu\kappa^2 R^2$ and hence
    \begin{align*}
    \frac{\mu}{2}\norm{X_t - X_0 + \frac{\grad_2\, q(\tilde\gamma, X_0)}{\mu}}_F^2 \leq Q(X_t) - Q(X_0) + \frac{\norm{\grad_2\, q(\tilde\gamma, X_0)}_F^2}{2\mu}\leq \frac{5\mu\kappa^2 R^2}{2}.
    \end{align*}
    Using the assumption $X_0 = 0_{(n-k)\times k}$ in~\cref{as:regularity} and applying triangle inequality together with the bound $\norm{\grad_2\,q(\tilde\gamma, X_0)}_F^2\leq \mu^2\kappa^2 R^2$ derived from~\cref{as:regularity}, we deduce that for all $t\geq 0$,
    \begin{align*}
    \norm{X_t}_F &\leq \left(1+\sqrt{5}\right)\kappa R.
    \end{align*}
    Then, as $\Xi_{t+1} = X_{t+1} + \frac{1-\alpha}{1+\alpha}\left(X_{t+1} - X_t\right)$, we have
    \begin{align*}
    \norm{\Xi_{t+1}}_F &\leq 3\left(1+\sqrt{5}\right)\kappa R\leq 10\kappa R.\qedhere
    \end{align*}
    \end{proof}

    \begin{proof}[Proof of \cref{lem:grad_hessian_psi}]
        Recall that by definition, the linear operator $\cG$ maps $\gamma$ to $\sum_{i=1}^m \gamma_i \left(A_i \Xi_t + B_i\right)$.
        Thus, for any $\gamma\in\bS^{m-1}$,
        \begin{align*}
        \norm{\cG\gamma}_F &= 
        \norm{\sum_{i=1}^m \gamma_i \left(A_i \Xi_t + B_i\right)}_F\\
        &\leq \norm{\sum_{i=1}^m \gamma_i A_i}_2\norm{\Xi_t}_F+ \norm{\sum_{i=1}^m \gamma_i B_i}_F\\
        &\leq 11\frac{\mu \kappa H R}{D}.\qedhere
        \end{align*}

\end{proof}

\begin{proof}[Proof of \cref{lem:gamma_necessary_accuracy}]
    Let $r \coloneqq \norm{\gamma^{(i)} - \gamma^*}_2$.
    Using~\cref{as:primal_dual_regularity}, we may bound the individual terms within the definition of $r^{(i)}$ as
    \begin{gather*}
    2\hat R_d - \norm{\gamma^{(i)}}_2 \geq \hat R_d - r \geq \frac{\hat\mu}{\hat \rho} - r,\\
    \frac{\lambda_{\min}\left(A\left(\gamma^{(i)}\right)\right) - \hat\mu/2}{\hat\rho} \geq \frac{\hat\mu/2 - \hat\rho r}{\hat\rho} = \frac{\hat\mu}{2\hat\rho} - r,\\
    \frac{2\hat L - \lambda_{\max}\left(A\left(\gamma^{(i)}\right)\right)}{\hat\rho} \geq \frac{\hat L - \hat\rho r}{\hat\rho} = \frac{\hat L}{\hat\rho} - r,\text{ and}\\
    \frac{2\hat L\hat R_p - \norm{B\left(\gamma^{(i)}\right)}_F}{\hat\rho\hat R_p} \geq \frac{\hat L - \hat\rho r}{\hat\rho} = \frac{\hat L}{\hat\rho} - r.
    \end{gather*}
    Thus, $r^{(i)} \geq \min\left(\frac{\hat\mu}{2\hat\rho},\,\frac{\hat\mu}{2\hat\rho} - r\right)= \frac{\hat\mu}{2\hat\rho} - r$. Then,  when $r\leq \frac{\hat\mu}{4\hat\rho}$, we have $r^{(i)}>0$ and furthermore, $r^{(i)}\geq r = \norm{\gamma^{(i)} - \gamma^*}_2$.
    \end{proof}

    \begin{proof}[Proof of \cref{lem:parameters_for_qmmp}]
        Begin by noting that for all $\gamma\in\cU^{(i)}$, 
        \begin{align*}
        \frac{\hat\mu}{2}I\preceq
        A\left(\gamma^{(i)}\right) - r^{(i)}\hat\rho I \preceq
        A(\gamma) \preceq
        A\left(\gamma^{(i)}\right) + r^{(i)}\hat \rho I \preceq 2\hat L I.
        \end{align*}
        Let $\tilde \gamma\in\argmax_{\gamma\in\cU^{(i)}}q(\gamma,0_{(n-k)\times k})$.
        Then,
        \begin{align*}
        \norm{B(\tilde\gamma)}_F \leq \norm{B\left(\gamma^{(i)}\right)}_F + \hat\rho r^{(i)}\hat R_p \leq 2\hat L\hat R_p = LR.
        \end{align*}
        Next, for $\gamma\in\bS^{m-1}$
        \begin{gather*}
        \frac{D\norm{\sum_{i=1}^m \gamma_i A_i}_2}{\mu} \leq \frac{4r^{(i)}\hat\rho}{\hat\mu} \leq 2\\
        \frac{D\norm{\sum_{i=1}^m \gamma_i B_i}_F}{L R} \leq \frac{r^{(i)}\hat\rho}{\hat L} \leq 1/2.\qedhere
        \end{gather*}
        \end{proof}
        
        \begin{lemma}
            \label{lem:correctness_of_random_generation}
            Consider an instance of \eqref{eq:dist_min_qmp} generated by the random procedure in \cref{subsec:random_instances}. Then equality holds throughout \eqref{eq:dist_min_qmp}.
            \end{lemma}
            \begin{proof}
            It suffices to show that $\gamma^*$ and $T^*$ are feasible and achieve value $\norm{X^*}_F^2$ in the dual SDP (i.e., the third line of \eqref{eq:dist_min_qmp}).
            
            Note that by Schur Complement Theorem,
            \begin{align*}
            \begin{pmatrix}
                A(\gamma^*)/2 & B(\gamma^*)/2\\
                B(\gamma^*)^\intercal /2  & \frac{c(\gamma^*)}{k}I_k  - T^*
            \end{pmatrix}
            \sim
            \begin{pmatrix}
                I_{n-k} & \\
                  & \frac{c(\gamma^*)}{k}I_k  - T^* - \frac{B(\gamma^*)^\intercal A(\gamma^*)^{-1}B(\gamma^*)}{2}
            \end{pmatrix} = \begin{pmatrix}
                I_{n-k}&\\&0_k
            \end{pmatrix}.
            \end{align*}
            Here, $\sim$ indicates matrix similarity. Thus, $\gamma^*$ and $T^*$ are feasible in the dual SDP.
            
            Next,
            \begin{align*}
            \tr(T^*) &= \tr\left(\frac{c(\gamma^*)}{k}I_k  - \frac{B(\gamma^*)^\intercal A(\gamma^*)^{-1}B(\gamma^*)}{2}\right)\\
            &= \frac{\tr\left((X^*)^\intercal A(\gamma^*) X^*\right)}{2}+ \ip{B(\gamma^*), X^*} + c(\gamma^*)\\
            &= \frac{\norm{X^*}_F^2}{2} + \sum_{i=1}^m \gamma^*_i\left(\tr\left(\frac{(X^*)^\intercal A_i X^*}{2}\right) + \ip{B_i, X^*} + c_i\right) = \frac{\norm{X^*}_F^2}{2}.\qedhere
            \end{align*}
            \end{proof}

\section{Strict complementarity in quadratic matrix programs}
\label{sec:strict_comp_qmp}

In this section, we show that a generic quadratic matrix program (QMP) in an $n\by k$ dimensional matrix variable with at most $k$ constraints satisfies strict complementarity (assuming only existence of primal and dual solutions).

We will need the following lemma stating that a generic bilinear system has only the trivial solutions. This lemma follows from basic dimension-counting arguments in algebraic geometry. However, we will instead prove the lemma directly using only elementary tools.
\begin{lemma}
    \label{lem:generic_bilinear}
    Let $n,p\in\N$ and consider the space $(\R^{n\times p})^{n+p-1}$. Let the collection $(A_i) = (A_1,\dots, A_{n+p-1})$ denote an element of this space. Here, each $A_i\in\R^{n\by p}$. Then, the collections $(A_i)$ for which the bilinear system
    \begin{align*}
        \begin{cases}
        x^\intercal A_i y = 0 \qquad\forall i\in[n+p-1]
        \end{cases}
    \end{align*}
    has a nontrivial solution (i.e., where $x\in\R^n$ and $y\in\R^p$ are both nonzero) forms a set of measure zero in $(\R^{n\by p})^{n+p-1}$.
    \end{lemma}
    \begin{proof}
    Let $\cS$ be the exceptional set, i.e.,
    \begin{align*}
        \cS \coloneqq \set{(A_i)\in (\R^{n\by p})^{n+p-1}:\, \begin{array}{l}
        \exists x\in\R^n\setminus\set{0},\, y\in\R^p\setminus\set{0}\\
        x^\intercal A_i y = 0 ,\,\forall i\in[n+p-1]
        \end{array}}.
    \end{align*}
    By homogeneity, we may require that $x\in\R^n$ has some coordinate equal to one. Similarly, we will require that $y\in\R^p$ has some coordinate equal to one. Thus, we may decompose $\cS = \bigcup_{\ell=1}^n\bigcup_{r=1}^p \cS_{\ell,r}$, where
    \begin{align*}
        \cS_{\ell,r} = \set{(A_i)\in (\R^{n\by p})^{n+p-1}:\, \begin{array}{l}
        \exists x\in\R^n,\, y\in\R^p\\
        x_\ell = 1\\
        y_r = 1\\
        x^\intercal A_i y = 0 ,\,\forall i\in[n+p-1]
        \end{array}}.
    \end{align*}
    We will show that for each $\ell\in[n]$ and $r\in[p]$ that $\cS_{\ell,r}$ has measure zero. Without loss of generality, let $\ell = r= 1$.
    
    Consider the affine space
    \begin{align*}
    \cM \coloneqq \set{(x,y,B_1,\dots,B_{n+p-1})\in \R^n\times\R^p\times(\R^{n\times p})^{n+p-1}:\, \begin{array}{l}
    x_1 = 1\\
    y_1 = 1\\
    (B_i)_{1,1} = 0 ,\,\forall i\in[n+m-1]
    \end{array}}.
    \end{align*}
    Let $\cF_{1,1}: \cM\to(\R^{n\times p})^{n+p-1}$ send the element $(x,y,B_1,\dots,B_{n+p-1})$ to $(A_1,\dots,A_{n+p-1})$ where
    \begin{align*}
        A_i = \begin{pmatrix}
            1 & -x_2 & \dots & -x_n\\
            x_2 & 1 & \\
            \vdots & & \ddots\\
            x_n & & & 1
            \end{pmatrix}B_i \begin{pmatrix}
        1 & y_2 & \dots & y_p\\
        -y_2 & 1 & \\
        \vdots & & \ddots\\
        -y_p & & & 1
        \end{pmatrix}.
    \end{align*}

    One may verify that $\cF_{1,1}$ is a smooth map. Furthermore, its domain has dimension $(n-1) + (p-1) + (np - 1)(n+p-1) = np(n+p-1) - 1$. This is one less than the dimension of the space $\left(\R^{n\times p}\right)^{n+p-1}$.
    It is known that the image of a Euclidean space under a smooth map into a Euclidean space of higher dimension must have Lebesgue measure zero (see Sard's lemma~\cite{sard1942measure}). Thus, $\cF_{1,1}(\cM)$ has Lebesgue measure zero.
    
    It remains to verify\footnote{It is in fact true that the two sets are equal but only one direction is necessary in this proof.} that $\cS_{1,1}\subseteq \cF_{1,1}(\cM)$. Suppose $(A_i)\in\cS_{1,1}$ and let $x,y$ with $x_1 = y_1 = 1$ satisfy $x^\intercal A_i y = 0$ for all $i\in[n+p-1]$. Let
    \begin{align*}
        B_i = \begin{pmatrix}
            1 & -x_2 & \dots & -x_n\\
            x_2 & 1 & \\
            \vdots & & \ddots\\
            x_n & & & 1
            \end{pmatrix}^{-1} A_i \begin{pmatrix}
        1 & y_2 & \dots & y_p\\
        -y_2 & 1 & \\
        \vdots & & \ddots\\
        -y_p & & & 1
        \end{pmatrix}^{-1}.
    \end{align*}
    Note that $(B_i)_{1,1} = \frac{1}{\norm{x}^2\norm{y}^2}x^\intercal A_i y = 0$ for all $i\in[n+p-1]$. The remaining sets $\cS_{\ell,r}$ can be shown to have measure zero using analogous maps $\cF_{\ell,r}$. This concludes the proof.
    \end{proof}
    
    \begin{lemma}
    \label{lem:QMP_n_k}
    Let $n,k\in\N$ and consider the SDP relaxation of a QMP with $k$ constraints in a variable of size $n\by k$ and its dual:
    \begin{align*}
        &\inf_{Y\in\S^{n+k}}\set{\ip{\begin{pmatrix}
        A_\obj/2 & B_\obj/2\\
        B_\obj^\intercal/2 & \tfrac{c_\obj}{k}I_k
        \end{pmatrix}, Y}:\, \begin{array}{l}
            \ip{\begin{pmatrix}
                A_i/2 & B_i/2\\
                B_i^\intercal/2 & \tfrac{c_i}{k}I_k
                \end{pmatrix}, Y} = 0,\,\forall i\in[k]\\
                Y = \begin{pmatrix}
                * & *\\
                * & I_k
                \end{pmatrix}\succeq 0
        \end{array}}\\
        &\qquad\geq\sup_{\gamma\in\R^k,\, T\in\R^{k\by k}}\set{\tr(T):\, \begin{pmatrix}
        A(\gamma)/2 & B(\gamma)/2\\
        B(\gamma)^\intercal/2 & \frac{c(\gamma)}{k}I_k - T
        \end{pmatrix}\succeq 0}.
    \end{align*}
    There exists a subset $\cE\subseteq(\S^n)^{1+k} \times (\R^{n\by k})^{1+k}$ of measure zero such that if
    \begin{align*}
        (A_\obj, A_1,\dots,A_k,B_\obj, B_1,\dots,B_k)\notin \cE
    \end{align*}
    and the primal and dual SDPs are both solvable, then strict complementarity holds and the primal and dual SDPs both have unique optimizers.
    \end{lemma}
    \begin{proof}
    We will condition on the following bilinear system in the variables $(\gamma_\obj,\dots,\gamma_k)\in\R^{1+k}$ and $x\in\R^n$ having no nontrivial solutions:
    \begin{align*}
        \begin{cases}
        \left(\gamma_\obj A_\obj + \sum_{i=1}^k \gamma_i A_i \right)x = 0\\
        \left(\gamma_\obj B_\obj + \sum_{i=1}^k \gamma_i B_i\right)^\intercal x = 0
        \end{cases}.
    \end{align*}
    This is a homogeneous bilinear system in $n + 1 + k$ variables with $n+k$ constraints. Thus, by \cref{lem:generic_bilinear}, this system has no nontrivial solutions outside an exceptional set $\cE$ of measure zero.
    
    Let $(\gamma^*,T^*)$ denote a dual optimum solution. We claim that $A(\gamma^*)\succ 0$. For the sake of contradiction, assume that $x\in\ker(A(\gamma^*))$ is nonzero. Then, by assumption, $x^*$ and $(1,\gamma^*)$ are not a solution to the bilinear system above, i.e., $B(\gamma^*)^\intercal x \neq 0$ and there exists a column of $B(\gamma^*)$, say the first column, that has nonzero inner product with $x$.
    This contradicts the feasibility of $(\gamma^*, T^*)$. Specifically for $\alpha \in\R$,
    \begin{align*}
        \begin{pmatrix}
        \alpha x\\
        e_1
        \end{pmatrix}^\intercal
        \begin{pmatrix}
        A(\gamma^*)/2 & B(\gamma^*)/2\\
        B(\gamma^*)^\intercal/2 & \frac{c(\gamma^*)}{k}I_k - T^*
        \end{pmatrix}
        \begin{pmatrix}
        \alpha x\\
        e_1
        \end{pmatrix} = \alpha \ip{x, B(\gamma^*)e_1} + \left(c(\gamma^*)/k + T^*_{1,1}\right).
    \end{align*}
    Picking $\alpha$ large or small enough makes this quantity negative, contradicting that the matrix on the left is positive semidefinite.

    We have shown that for every dual optimum solution $(\gamma^*, T^*)$, the above slack matrix has rank at least $n$. Similarly, any primal optimum solution $Y^*$ must have rank at least $k$.
    We deduce that every primal optimum solution $Y^*$ has rank exactly $k$ and that for every dual optimum solution $(\gamma^*, T^*)$, the slack matrix has rank exactly $n$.
    Now, these optimizers must correspond to faces of slices of $\S^{n+k}_+$. As the only faces of slices of $\S^{n+k}_+$ with constant rank are singleton sets, we deduce that there is a unique primal optimizer and a unique dual optimizer.
\end{proof}

\section{Additional experiments on phase-retrieval inspired SDP instances}
\label{sec:additional_exp}

We perform additional experiments on SDP instances inspired by the phase retrieval problem.

The phase retrieval problem seeks to learn a vector $x^*$ given only the magnitudes of linear measurements of $x^*$, and finds applications in imaging.
In the Gaussian model of phase retrieval~\cite{candes2015phase}, we assume $x^*\in\R^n$ is arbitrary and $G\in\R^{m \by n}$ is entrywise Gaussian with an appropriate normalization. We are given
\begin{align*}
    \abs{Gx^*}.
\end{align*}
Here, the absolute value is taken entrywise. Equivalently, we are given the entrywise square of $Gx^*$, or $b = \diag(Gx^*(x^*)^\intercal G^\intercal)$. In this setting, it is known that the PhaseLift SDP,
\begin{align*}
    \min_{Y\succeq 0}\set{\tr(Y):\, \begin{array}{l}
        \diag(GYG^\intercal) = \beta
    \end{array}}
\end{align*}
has $(x^*)(x^*)^\intercal$ as its unique solution with high probability once the number of observations $m$ is roughly $O(n)$. Recent work~\cite{ding2023sharp} shows that strict complementarity holds between this SDP and its dual with high probability in the same regime.

We note that the Gaussian model of phase retrieval requires storing the matrix $G$ as part of the instance. This is a matrix of size $O(n^2)$ and thus limits the size of our current experiments. Nonetheless, we expect the behavior we observe with these experiments to hold in the real setting of phase retrieval where the matrix $G$ is highly structured and can be stored implicitly. We leave this as important future work.

We compare CertSDP (\cref{alg:exact_sdp_qmmp}), CSSDP~\cite{ding2019optimal}, SketchyCGAL~\cite{yurtsever2021scalable}, ProxSDP~\cite{souto2020exploiting}, SCS~\cite{odonoghue2016conic}, and Burer-Monteiro~\cite{burer2003nonlinear}.

\paragraph{Random instance generation.} 
We generate instances as follows. Suppose $n$ is given. We set $m = 5n$.
We generate $G\in\R^{m\by n}$ where each entry is independent $N(0,1/m)$. We then preprocess $G$ so that its $m$th observation vector, i.e., the $m$th row of $G$, is parallel to $e_n$.
Next, we sample $x^*$ uniformly form
\begin{align*}
    \bS^{n-1}\cap\set{x\in\R^n:\, x_n = 0.1}.
\end{align*}
Thus, this is a random instance of phase retrieval where we are given one highly-correlated observation.

\paragraph{Implementation details.} The algorithms we test are mostly as described in \cref{subsec:implementation_algs}. The major differences in implementation are described below:
\begin{itemize}
    \item In the instances tested in \cref{sec:numerical}, the $A_i$ matrices encountered were sparse. In the phase retrieval problems we test in this appendix, the $A_i$ matrices are dense but rank-one. The implementations of CertSDP, CSSDP, and SketchyCGAL are modified to handle such instances.
    \item Phase retrieval instances are likely to contain many dual optimal solutions that may not satisfy strict complementarity. Within CertSDP and CSSDP, we employ the Accelegrad algorithm to approximately solve
    \begin{align*}
        \max_{\gamma\in\R^m} \beta^\intercal \gamma + \textup{penalty} \min\left(0, \lambda_{\min}\left(I - G^\intercal \Diag(\gamma)G\right)\right) + \min(0, \lambda_{1+2}(I - G^\intercal\Diag(\gamma)G) - 0.1).
    \end{align*}
    Here, $\lambda_{1+2}(\cdot)$ denotes the sum of the two smallest eigenvalues of a given matrix and is a concave expression in its input. This penalization/regularization encourages solutions $\gamma$ for which the second eigenvalue of $I - G^\intercal\Diag(\gamma)G$ is positive, so that $A(\gamma)\succ 0$. We set $\textup{penalty} = 10$.
    \item Recall that in \cref{sec:numerical}, we replaced the random sketch in SketchyCGAL with a projection onto a submatrix to reflect the fact that for QMP instances, the goal is to recover the $(n-k)\by k$ top-right submatrix of the SDP optimizer. For the phase retrieval experiments, we employ the random sketch as originally described in \cite{yurtsever2019conditional}.
\end{itemize}

\paragraph{Numerical results.}
Due to memory constraints associated with storing $G\in\R^{m\by n}$, we test instances with size $n = 30,\, 100,\,300$.
We set the time limit to 50, 500, and 5000 seconds respectively.
The results are summarized in \cref{table:phret_1e2,table:phret_3e1,table:phret_3e2}.
The average memory usage of the algorithms is plotted in \cref{fig:phret_memory_usage}.
We compare the convergence behavior of CertSDP with that of CSSDP and SketchyCGAL on a single instance of each size in \cref{fig:phret_convergence_behavior}.

\begin{table}[htbp]
    \centering
    \begin{tabular}{lllllll}
\toprule 
Algorithm & time (s) & std. & $\norm{x - x^*}_2^2$ & std. & memory (MB) & std. \\
\midrule 
CertSDP & \num{3.8e+01} & \num{1.4e+01} & \num{3.0e-20} & \num{4.5e-20} & \num{0.0e+00} & \num{0.0e+00} \\
CSSDP & \num{5.0e+01} & \num{2.9e-02} & \num{7.3e-09} & \num{6.2e-09} & \num{0.0e+00} & \num{0.0e+00} \\
SketchyCGAL & \num{5.0e+01} & \num{7.8e-02} & \num{1.2e-08} & \num{9.2e-09} & \num{0.0e+00} & \num{0.0e+00} \\
ProxSDP & \num{2.5e+00} & \num{4.3e-01} & \num{6.2e-18} & \num{1.2e-17} & \num{6.2e-01} & \num{1.6e+00} \\
SCS & \num{5.1e+01} & \num{5.2e-02} & \num{9.4e-04} & \num{3.0e-03} & \num{0.0e+00} & \num{0.0e+00} \\
BM & \num{4.3e-01} & \num{1.2e-02} & \num{1.1e-17} & \num{1.5e-17} & \num{9.4e-02} & \num{2.5e-01}\\
\bottomrule 
\end{tabular}
     \caption{Experimental results for phase retrieval instances with $n = 30$ (10 instances) with time limit $50$ seconds.}
    \label{table:phret_3e1}
\end{table}

\begin{table}[htbp]
    \centering
    \begin{tabular}{lllllll}
\toprule 
Algorithm & time (s) & std. & $\norm{x - x^*}_2^2$ & std. & memory (MB) & std. \\
\midrule 
CertSDP & \num{2.6e+02} & \num{1.4e+02} & \num{1.2e-12} & \num{3.9e-12} & \num{0.0e+00} & \num{0.0e+00} \\
CSSDP & \num{5.0e+02} & \num{1.2e-02} & \num{1.1e-09} & \num{1.1e-09} & \num{0.0e+00} & \num{0.0e+00} \\
SketchyCGAL & \num{5.0e+02} & \num{2.1e-02} & \num{7.1e-08} & \num{6.8e-08} & \num{0.0e+00} & \num{0.0e+00} \\
ProxSDP & \num{4.8e+02} & \num{7.1e+01} & \num{2.9e-14} & \num{9.3e-14} & \num{8.3e+02} & \num{5.7e+01} \\
SCS & \num{5.1e+02} & \num{4.6e-01} & \num{8.0e-10} & \num{2.1e-10} & \num{3.7e+02} & \num{4.8e+01} \\
BM & \num{5.2e-01} & \num{6.7e-02} & \num{8.0e-17} & \num{1.6e-16} & \num{0.0e+00} & \num{0.0e+00}\\
\bottomrule 
\end{tabular}
     \caption{Experimental results for phase retrieval instances with $n = 100$ (10 instances) with time limit $500$ seconds.}
    \label{table:phret_1e2}
\end{table}

\begin{table}[htbp]
    \centering
    \begin{tabular}{lllllll}
\toprule 
Algorithm & time (s) & std. & $\norm{x - x^*}_2^2$ & std. & memory (MB) & std. \\
\midrule 
CertSDP & \num{5.0e+03} & \num{7.4e+01} & \num{2.8e-10} & \num{2.5e-10} & \num{0.0e+00} & \num{0.0e+00} \\
CSSDP & \num{5.0e+03} & \num{5.8e-02} & \num{1.3e-09} & \num{8.4e-10} & \num{0.0e+00} & \num{0.0e+00} \\
SketchyCGAL & \num{5.0e+03} & \num{2.5e-02} & \num{1.6e-06} & \num{2.1e-06} & \num{0.0e+00} & \num{0.0e+00} \\
ProxSDP & \num{5.1e+03} & \num{2.4e+00} & \num{3.4e-04} & \num{2.5e-04} & \num{6.9e+03} & \num{1.3e+03} \\
SCS & \num{5.5e+03} & \num{6.9e+00} & \num{4.3e-11} & \num{6.9e-11} & \num{6.3e+03} & \num{1.2e+03} \\
BM & \num{2.1e+00} & \num{2.5e+00} & \num{1.4e-14} & \num{2.4e-14} & \num{0.0e+00} & \num{0.0e+00}\\
\bottomrule 
\end{tabular}
     \caption{Experimental results for phase retrieval instances with $n = 300$ (10 instances) with time limit $5000$ seconds.}
    \label{table:phret_3e2}
\end{table}

\begin{figure}
	\centering
	\includegraphics[scale=\figurescale]{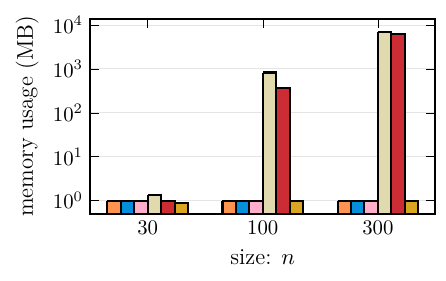}

	\includegraphics[scale=\figurescale]{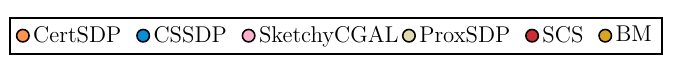}

	\caption{Memory usage of different algorithms on our phase retrieval instances as a function of the size $n$. In this chart, we plot $0.0$ MB at $1.0$ MB (see \cref{rem:virtual_size_memory} for a discussion on measuring memory usage).}
	\label{fig:phret_memory_usage}
\end{figure}

\begin{figure}
	\centering
	\includegraphics[scale=\figurescale]{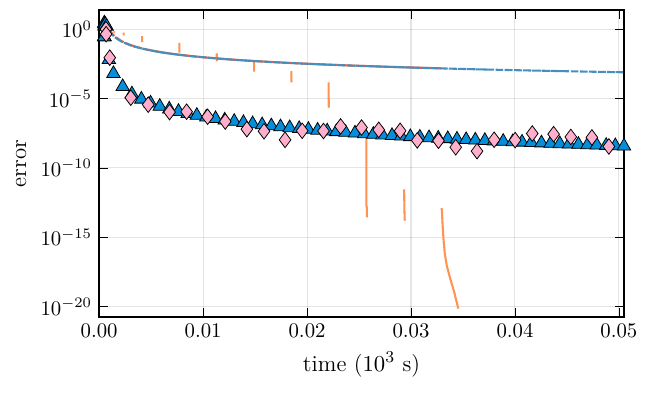}

	\vspace{-0.5em}

	\includegraphics[scale=\figurescale]{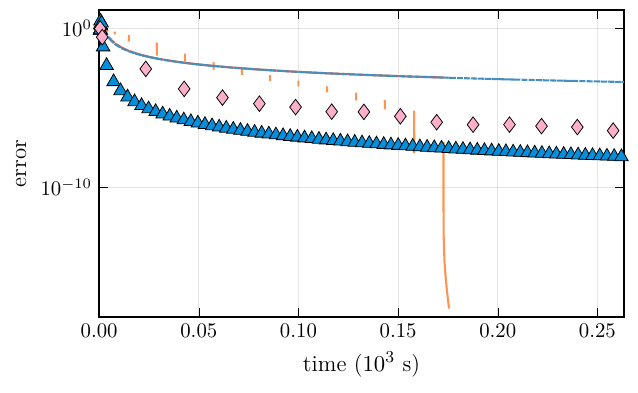}

	\vspace{-0.5em}

	\includegraphics[scale=\figurescale]{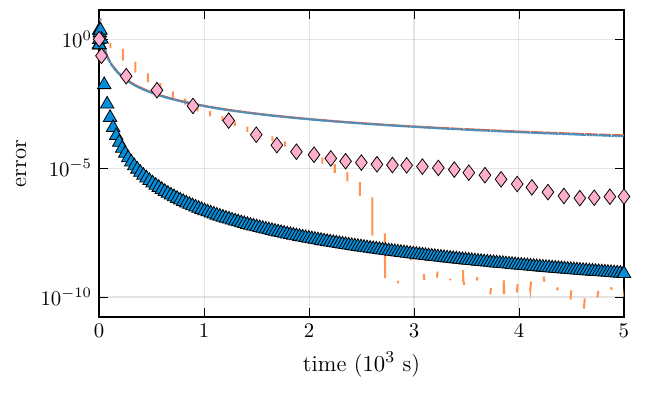}

	\vspace{-0.5em}

	\includegraphics[scale=\figurescale]{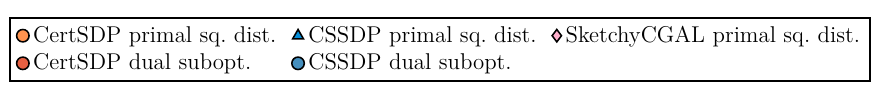}

	\caption{Comparison of convergence behavior between CertSDP (\cref{alg:exact_sdp_qmmp}), CSSDP, and SketchyCGAL on our phase retrieval instances.
	The first, second, and third rows show experiments with $n=30$, $100$, and $300$ respectively.}
	\label{fig:phret_convergence_behavior}
\end{figure}

The results for these experiments are qualitatively similar to those of \cref{sec:numerical}. We make a few additional observations:

\begin{itemize}
    \item On these phase retrieval instances, the dual suboptimality decreases to $\approx 10^{-3}$ before CertSDP seems to find a certificate of strict complementarity (see \cref{fig:phret_convergence_behavior}). This suggests that the value of $\mu^*$ in these instances is relatively small.
    \item CSSDP outperforms SketchyCGAL and also outperforms CertSDP initially. The ``crossover'' point where CertSDP outperforms CSSDP occurs only after CSSDP is able to produce a primal iterate with squared error $\approx 10^{-7}$.
    \item CertSDP seems to suffer from numerical issues for $n = 300$ and is unable to decrease the primal squared error beyond $10^{-10}$. Nonetheless, CertSDP outperforms CSSDP and SketchyCGAL on all instances tested.
\end{itemize} \end{appendix}

\end{document}